\documentclass[12pt,reqno]{amsart}
\usepackage{amssymb,amsmath,dsfont}
\usepackage{pgfplots}
\pgfplotsset{compat=1.14}
\usepackage{enumitem}
\usepackage{calc}

\title{Mixing times and cutoff for the TASEP in the high and low density phase}
\date{}
\author{Dor Elboim and Dominik Schmid}

\address{Dor Elboim, Princeton University, United States}
\email{delboim@math.princeton.edu}
\address{Dominik Schmid, University of Bonn}
\email{d.schmid@uni-bonn.de}
\date{\today}

\usepackage[margin=1in]{geometry}
\usepackage[all]{xy}
\usepackage{amssymb}
\usepackage{amsfonts}
\usepackage{amsthm}
\usepackage{amsmath}
\usepackage{hyperref}
\usepackage{mathtools}
\usepackage{url}
\usepackage{xcolor}
\usepackage{dsfont}
\usepackage{pgfplots}
\usepackage{comment}

\usetikzlibrary{shapes,arrows}
\usetikzlibrary{hobby}
\usetikzlibrary{bending}

\newtheorem{thm}{Theorem}[section]

\newtheorem{lem}[thm]{Lemma}  
\newtheorem{prop}[thm]{Proposition}
\newtheorem{cor}[thm]{Corollary}

\newtheorem{conj}[thm]{Conjecture} 

\newtheorem{claim}[thm]{Claim}

\renewcommand{\P}{\mathbb{P}}

\newcommand{\R}{\mathbb{R}}

\newcommand{\Z}{\mathbb{Z}}
\newcommand{\N}{\mathbb{N}}
\newcommand{\E}{\mathbb{E}}

\newcommand{\var}{\mathrm{Var}}

\newcommand{\TV}[1]{{\lVert #1 \rVert}_{\normalfont
\text{TV}}}
\newcommand\abs[1]{\left| #1\right|}
\mathtoolsset{showonlyrefs}

\newcommand{\eone}{\textup{e}_1}
\newcommand{\etwo}{\textup{e}_2}

\numberwithin{equation}{section}

\begin{document}

\maketitle

\begin{abstract}
    We study the totally asymmetric simple exclusion process with open boundaries in the high density and the low density phase. In the bulk of the two phases, we show that the process on a segment of length $N$ exhibits cutoff at order $N$, while in the intersection of the phases, the coexistence line, the mixing time is of order $N^2$, and no cutoff occurs. In particular, we determine the $\varepsilon$-mixing time in the coexistence line up to constant factors, which do not dependent on $\varepsilon$. Combined with previous results on the maximal current phase, this completes the picture on mixing times for the TASEP with open boundaries. 
\end{abstract}

\section{Introduction}

The simple exclusion process is among the most investigated  interacting particle system. We study the non-conservative exclusion process on a segment with open boundaries, and with total asymmetry, which is a central model in probability, statistical mechanics and combinatorics \cite{CW:TableauxCombinatorics,DEHP:ASEPCombinatorics,L:ErgodicI,S:PASEPpolynomials,USW:PASEPcurrent}. In this setup, the \textbf{TASEP with open boundaries}, particles enter at the left boundary of the segment according to rate $\alpha$ Poisson clocks and exit at the right boundary at rate $\beta$, for some $\alpha,\beta>0$. Within the bulk of the segment, particles attempt moves to the right according to rate $1$ Poisson clocks under an exclusion rule, i.e.\ a jump is performed if and only if the target site is not occupied; see Figure \ref{fig:BDEP}. \\

Depending on the particle density within the segment in equilibrium, it is well known that we see three different regimes for the TASEP with open boundaries \cite{DEHP:ASEPCombinatorics}. In the \textbf{high density phase}, where $\beta< \min(\alpha,\frac{1}{2})$, the expected density within the bulk approaches $1-\beta$ when the size of the segment grows. Similarly, in the \textbf{low density phase}, where $\alpha< \min(\beta,\frac{1}{2})$, we see that the expected density of particles converges to $\alpha$. In the \textbf{maximum current phase}, where $\alpha,\beta \geq \frac{1}{2}$, the expected density within the bulk approaches $\frac{1}{2}$. A visualization of the three phase is given in Figure \ref{fig:Phases}. In contrast to these three regimes, the expected density in equilibrium does not stabilize at the boundary between the high and the low density phase, called the \textbf{coexistence line}, where $\alpha=\beta<\frac{1}{2}$, and it turns out that the expected density in equilibrium interpolates between $\alpha$ and $1-\beta$; see \cite[Part III, Theorem 3.41]{L:Book2}. \\

In this paper, we investigate the speed of convergence of the TASEP with open boundaries towards equilibrium in the high and low density phase as well as on the coexistence line. This is formalized using the notion of total variation mixing times; see \cite{LPW:markov-mixing} for an introduction. In particular, we are interested for which choices of parameters the cutoff phenomenon, an abrupt transition from unmixed to mixed, occurs. 

\subsection{Model and results} We give now a formal definition of the TASEP with open boundaries and mixing times. Fix parameters $\alpha,\beta>0$ and consider the segment $[N]:=\{1,\dots,N\}$ for some $N \in \N$. The TASEP with open boundaries is defined as the continuous-time Markov chain $(\eta_t)_{t\geq 0}$ on $\Omega_N :=\{ 0,1\}^N$ given by the generator
\begin{align*}
\mathcal{L} f(\eta) &=  \sum_{x =1}^{N-1}  \eta(x)(1-\eta(x+1))\left[ f(\eta^{x,x+1})-f(\eta) \right] \\
 &+ \alpha (1-\eta(1)) \left[ f(\eta^{1})-f(\eta) \right] \hspace{2pt}  + \beta \eta(N)\left[ f(\eta^{N})-f(\eta) \right] 
\end{align*}
 for all cylinder functions $f$. In this definition, for a given configuration $\eta \in \Omega_N$ and sites $x,y\in [N]$, we set 
\begin{equation*} 
\eta^{x,y} (z) = \begin{cases} 
 \eta (z) & \textrm{ for } z \neq x,y\\
 \eta(x) &  \textrm{ for } z = y\\
 \eta(y) &  \textrm{ for } z = x,
 \end{cases}
 \quad
 \text{and}
 \quad
 \eta^{x} (z) = \begin{cases} 
 \eta (z) & \textrm{ for } z \neq x\\
1-\eta(z) &  \textrm{ for } z = x \, .
 \end{cases}
\end{equation*}

\begin{figure} 
\centering
\begin{tikzpicture}[scale=1]
\def\spiral[#1](#2)(#3:#4:#5){% \spiral[draw options](placement)(end angle:revolutions:final radius)
\pgfmathsetmacro{\domain}{pi*#3/180+#4*2*pi}
\draw [#1,
       shift={(#2)},
       domain=0:\domain,
       variable=\t,
       smooth,
       samples=int(\domain/0.08)] plot ({\t r}: {#5*\t/\domain})
}

\def\particles(#1)(#2){

  \draw[black,thick](-3.9+#1,0.55-0.075+#2) -- (-4.9+#1,0.55-0.075+#2) -- (-4.9+#1,-0.4-0.075+#2) -- (-3.9+#1,-0.4-0.075+#2) -- (-3.9+#1,0.55-0.075+#2);
  
  	\node[shape=circle,scale=0.6,fill=red] (Y1) at (-4.15+#1,0.2-0.075+#2) {};
  	\node[shape=circle,scale=0.6,fill=red] (Y2) at (-4.6+#1,0.35-0.075+#2) {};
  	\node[shape=circle,scale=0.6,fill=red] (Y3) at (-4.2+#1,-0.2-0.075+#2) {};
   	\node[shape=circle,scale=0.6,fill=red] (Y4) at (-4.45+#1,0.05-0.075+#2) {};
  	\node[shape=circle,scale=0.6,fill=red] (Y5) at (-4.65+#1,-0.15-0.075+#2) {}; }

  \def\annhil(#1)(#2){	  \spiral[black,thick](9.0+#1,0.09+#2)(0:3:0.42);
  \draw[black,thick](8.5+#1,0.55+#2) -- (9.5+#1,0.55+#2) -- (9.5+#1,-0.4+#2) -- (8.5+#1,-0.4+#2) -- (8.5+#1,0.55+#2); }

	\node[shape=circle,scale=1.5,draw] (Z) at (-3,0){} ;
    \node[shape=circle,scale=1.5,draw] (A) at (-1,0){} ;
    \node[shape=circle,scale=1.5,draw] (B1) at (1,0){} ;
	\node[shape=circle,scale=1.5,draw] (B2) at (3,0){} ;
	\node[shape=circle,scale=1.5,draw] (C) at (5,0) {};
 	\node[shape=circle,scale=1.5,draw] (D) at (7,0){} ; 

 		\node[shape=circle,scale=1.2,fill=red] (YZ2) at (3,0) {};
 	
	\node[shape=circle,scale=1.2,fill=red] (YZ) at (-1,0) {};

	\node[shape=circle,scale=1.2,fill=red] (YZ3) at (5,0) {};

		\draw[thick] (Z) to (A);	
	\draw[thick] (A) to (B1);	
	\draw[thick] (B1) to (B2);		
		\draw[thick] (B2) to (C);	
  \draw[thick] (C) to (D);

\particles(0)(0);
\particles(6.9+4.9+1)(0);

%\annhil(0)(0.7);
%\annhil(-13.4)(-0.7);
\draw [->,line width=1pt]  (A) to [bend right,in=135,out=45,->] (B1);

%\draw [->,line width=1pt]  (B2) to [bend right,in=135,out=45,->] (C);
  
  % \draw [->,line width=1pt] (B2) to [bend right,in=-135,out=-45] (C);
  % \draw [->,line width=1pt] (B2) to [bend right,in=-135,out=-45] (A);
 %     \draw [->,line width=1pt] (A) to [bend right,in=-135,out=-45] (Z2);
    \node (text1) at (0,0.8){$1$} ;    
 
	\node (text3) at (-2.5-1,0.8){$\alpha$}; 
	\node (text4) at (7.1+0.4,0.8){$\beta$};

    \node[scale=0.9] (text1) at (-3,-0.7){$1$} ;    
    \node[scale=0.9] (text1) at (7,-0.7){$N$} ;   
  	
  \draw [->,line width=1pt] (-3.9,0.475) to [bend right,in=135,out=45] (Z);
%   \draw [->,line width=1pt] (Z) to [bend right,in=135,out=45] (-3.9,-0.475); 
  % \draw [->,line width=1pt] (6.9+1.6,-0.475) to [bend right,in=135,out=45] (D);	
   \draw [->,line width=1pt] (D) to [bend right,in=135,out=45] (6.9+1,0.475);

	\end{tikzpicture}	
\caption{The TASEP with open boundaries for parameters $\alpha,\beta >0$.}\label{fig:BDEP}
 \end{figure}
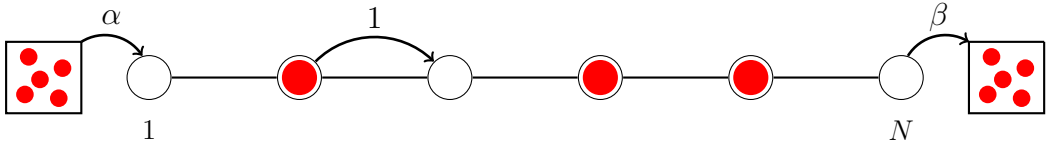
 A visualization of the TASEP with open boundaries is provided in Figure \ref{fig:BDEP}. For all choices of $\alpha,\beta>0$ and $N\in \N$, the TASEP with open boundaries is an irreducible Markov chain on $\Omega_N$, and thus converges to a unique stationary distribution $\mu=\mu_N$ on $\Omega_N$. Note that in our setup, the TASEP with open boundaries is not reversible with respect to $\mu$. To quantify the speed of converge towards $\mu$, we let, for a probability measure $\nu$ on $\Omega_N$, 
\begin{equation}\label{def:TVDistance}
\TV{ \nu - \mu } := \frac{1}{2}\sum_{x \in \Omega_N} \abs{\nu(x)-\mu(x)} = \max_{A \subseteq \Omega_N} \left(\nu(A)-\mu(A)\right) 
\end{equation} be the \textbf{total variation distance} of $\nu$ and $\mu$. We define the $\boldsymbol\varepsilon$\textbf{-mixing time} of $(\eta_t)_{t \geq 0}$ by
\begin{equation}\label{def:MixingTime}
t^N_{\text{\normalfont mix}}(\varepsilon) := \inf\left\lbrace t\geq 0 \ \colon \max_{\eta \in \Omega_{N}} \TV{\P\left( \eta_t \in \cdot \ \right | \eta_0 = \eta) - \mu_N} < \varepsilon \right\rbrace
\end{equation} for all $\varepsilon \in (0,1)$. %Our goal is to study the order of $t^N_{\text{\normalfont mix}}(\varepsilon)$ when $N$ goes to infinity. 
In the following, in order to simplify the notation, we set
\begin{equation}\label{def:ABs}
b:= \frac{1-\beta}{\beta} \qquad a:= \frac{1-\alpha}{\alpha} \qquad \rho_\alpha:=\alpha(1-\alpha)\qquad \rho_\beta:=\beta(1-\beta) \, .
\end{equation} Moreover, let $\hat{a}:=\max(1,a)$ and $\hat{b}:=\max(1,b)$. Note that $b>\hat{a}$ in the high density phase, whereas $a>\hat{b}$ in the low density phase.  We are now ready to state our main results.
\begin{thm}\label{thm:HighLow}
Suppose that $\beta< \min(\alpha,\frac{1}{2})=: \hat{\alpha}$, i.e.~ we consider the high density phase. For all $\varepsilon \in (0,1)$
\begin{equation}\label{eq:MixingTimeHigh}
    \lim_{N \to \infty}\frac{t^N_{\textup{mix}}(\varepsilon )}{N} = \frac{(\hat{a}+1)^2(b+1)(b-1)}{(\hat{a}b-1)(b-\hat{a})} = \frac{1-2\beta}{(1-\hat{\alpha}-\beta)(\hat{\alpha}-\beta)} \, .
\end{equation} Suppose that $\alpha< \min(\beta,\frac{1}{2})=: \hat{\beta}$, i.e.~ we consider the low density phase. For all $\varepsilon \in (0,1)$
\begin{equation}\label{eq:MixingTimeLow}
    \lim_{N \to \infty}\frac{t^N_{\textup{mix}}(\varepsilon )}{N} = \frac{(\hat{b}+1)^2(a+1)(a-1)}{(a\hat{b}-1)(a-\hat{b})} =  \frac{1-2\alpha}{(1-\hat{\beta}-\alpha)(\hat{\beta}-\alpha)} \, .
\end{equation} 
\end{thm}
Previously, a mixing time of order $N$ was shown by Gantert et al. in \cite{GNS:MixingOpen}. Note that the first  order of the $\varepsilon$-mixing times in \eqref{eq:MixingTimeHigh} and \eqref{eq:MixingTimeLow} does not depend on $\varepsilon$ when the size of the underlying state space grows. This is known as the \textbf{cutoff phenomenon}, and it is in general a challenging question to determine for a family of Markov chains whether it exhibits cutoff. %; see \cite{} for recent progress on an effective criterion. 
Theorem \ref{thm:HighLow}  partially answers Conjecture 1.8 in \cite{GNS:MixingOpen}, where cutoff was predicted for the high and low density phase of general asymmetric exclusion processes, including the TASEP with open boundaries.
In contrast to the linear mixing time with cutoff in the high and the low density phase, we have the following result for the coexistence line.
\begin{thm}\label{thm:1}
Suppose that $\alpha=\beta<\frac{1}{2}$, i.e.~ we consider the coexistence line. There exist constants $C,c>0$ such that for all $\varepsilon \le 1/4$, we find some $N_0(\varepsilon)$ so that for all $N\ge N_0$
\begin{equation}
    c N^2\log(1/\varepsilon) \leq t^N_{\textup{mix}}(\varepsilon) \leq C N^2\log(1/\varepsilon)\, .
\end{equation}
In particular, the cutoff phenomenon does not occur.
\end{thm}
Combined with previous results on the maximal current phase $\alpha,\beta \geq \frac{1}{2}$, where the mixing time was shown in \cite{S:MixingTASEP} and \cite{SS:TASEPcircle} to be of order $N^{3/2}$, this provides a full picture of the mixing times for the TASEP with open boundaries. In particular, for open boundaries, we see a richer mixing behavior than on the closed segment, where a linear mixing time and cutoff were shown by Labbé and Lacoin in \cite{LL:CutoffASEP} for the asymmetric simple exclusion process.

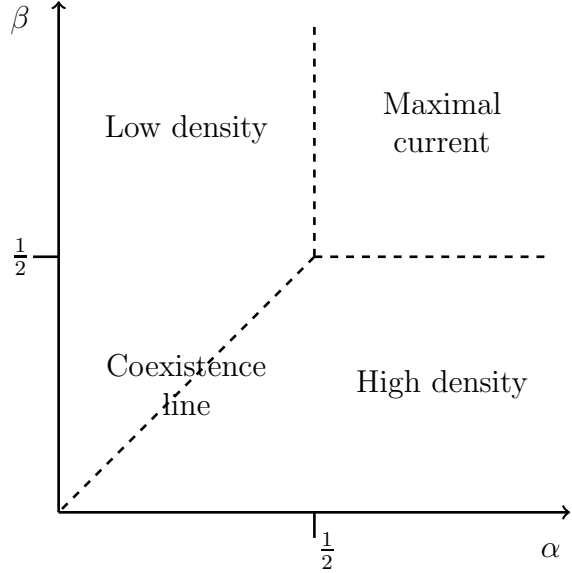
\begin{figure}
    \centering
\begin{tikzpicture}[scale=1.7]
\draw [->,line width=1pt] (0,0) to (4,0); 	
\draw [->,line width=1pt] (0,0) to (0,4); 	
\draw [line width=1pt] (0,2) to  (-0.2,2); 	
\draw [line width=1pt] (2,0) to  (2,-0.2); 		
\draw [line width=1pt,dashed] (2,2) to  (0,0); 		
\draw [line width=1pt,dashed] (2,2) to (3.8,2); 	
\draw [line width=1pt,dashed] (2,2) to (2,3.8); 		
 \node (H1) at (-0.3,2) {$\frac{1}{2}$};	 
 \node (H2) at (2.1,-0.3) {$\frac{1}{2}$};	 
 \node (H3) at (-0.3,3.85) {$\beta$};	 
 \node (H4) at (3.85,-0.3) {$\alpha$};

 \node (X1) at (3,3.2) {Maximal};
 \node (X11) at (3,2.9) {current};
 \node (X2) at (1,3) {Low density};
 \node (X3) at (3,1) {High density};
 
 \node (X4) at (1,1.15) {Coexistence}; 
  \node (X4) at (1,0.85) {line};

%\visible<2>{
% \node[scale=0.8] (X1) at (3.2,3.2) {[S. '21; S. and Sly '22]};
% \node[scale=0.9] (X11) at (3.2,2.9) {$t^N_{\textup{mix}} \asymp N^{5/2}$ and no cutoff};

% \node[scale=0.8] (X1) at (1,3.5) {[Gantert, Nestoridi, S. '21]};
% \node[scale=0.9] (X11) at (1,3.2) {$t^N_{\textup{mix}} \asymp N^2$}; 
%\node[scale=0.8] (X1) at (1,2.9) {[Georgiou, S. '22+]};
% \node[scale=0.9] (X11) at (1,2.6) {cutoff occurs};  

% \node[scale=0.8] (X1) at (1.03,1.2) {[Elboim, S. \quad '22+]};
% \node[scale=0.9] (X11) at (1,0.9) {  $t^N_{\textup{mix}} \asymp N^3$ \ and no cutoff}; 

 %\node[scale=0.8] (X1) at (3.2,1.5) {[Gantert, Nestoridi, S. '21]};
% \node[scale=0.9] (X11) at (3.2,1.2) {$t^N_{\textup{mix}} \asymp N^2$}; 
%\node[scale=0.8] (X1) at (3.2,0.9) {[Georgiou, S. '22+]};
% \node[scale=0.9] (X11) at (3.2,0.6) {cutoff occurs};  
 
% \node (X2) at (1,3) {High density};
% \node (X3) at (3,1) {Low density};
% 
% \node (X4) at (0.8,1.4) {Coexistence}; 
%  \node (X4) at (0.8,1.1) {line}; 
% } 

\end{tikzpicture}
    \caption{Visualization of the different phases for the TASEP with open boundaries. Theorem \ref{thm:HighLow} provides a linear mixing time in the high and in the low density phase, where cutoff occurs. In Theorem \ref{thm:1} on the coexistence line, we see a mixing time which is quadratic in the length of the segment, while no cutoff occurs.}
    \label{fig:Phases}
\end{figure}

\subsection{Difference between the high and low density phase and the coexistence line} Let us comment on the difference in the behavior of the mixing times in the bulk of the high and low density phase and their intersection. 
Heuristically, one reason for cutoff in the bulk of the two phases is the creation of shocks; see also Section 1.2 in \cite{GNS:MixingOpen} for a more quantitative explanation. 
In short, in the high density phase starting from the empty initial configuration, particles enter at the left boundary, and the bulk density raises until it reaches a critical density of $\beta$ at the right boundary, when the density jumps to $1-\beta$.
At this point, we see the formation of a shock traveling from the right to the left boundary at a linear speed. When the shock reaches the left boundary, the process has abruptly mixed. \\
 
 Using an alternative interpretation as a last passage percolation model on a slab with boundary parameters $\alpha$ and $\beta$ -- see Section \ref{sec:TASEPLPPcorres} -- this behavior can be interpreted as follows.
 In the high density phase, the lower boundary of the slab is attractive for paths of maximal weight, called geodesics. Using moderate deviation estimates to control the fluctuations of geodesics, we see that the parts of the segment have mixed whenever the respective geodesics touch the lower boundary with high probability.  \\ 

In the coexistence line, the behavior is drastically different. It turns out that in equilibrium, we see a shock at a random location; see Theorem 3.41 in Part III of \cite{L:Book2}. Again, this can be explained heuristically using the interpretation as a last passage percolation model on the slab. In the coexistence line, both boundaries are equally attractive. When rescaling space by $N$ and time by $N^2$, the total weight collected at each boundary is given by an independent Brownian motion with drift, and the shock location depends on which boundary yields the larger weight to a given site in the slab. In \cite{SS:TASEPcircle}, it is shown that the mixing time corresponds to the coalescence time of the outer geodesics. Intuitively, in the rescaled process, this coalescence time between the geodesics corresponds to the first time that the difference between the two independent Brownian motions  becomes larger than a given threshold, which depends only on the choice of the boundary parameters. Since the coalescence time is random, and has full support on $\mathbb{R}_+$, this justifies the absence of cutoff.

\subsection{Related work}
Over the last years, mixing times for exclusion processes are intensively studied. For the symmetric simple exclusion process, where the underlying particles perform symmetric random walks,  and the total number of particles $k\leq N/2$ is conserved, Lacoin proves that the mixing time is of leading order $\pi^{-2}N^{2}\log(k)$ \cite{L:CutoffSEP}. Moreover,  cutoff occurs whenever the number of particles $k$ goes to infinity with $N$. Previously, a lower bound of the correct first order was obtained by Wilson in  \cite{W:MixingLoz}. When the particles perform biased random walks on the segment under the exclusion rule, Labbé and Lacoin show that the mixing time is linear in the size of the segment, and cutoff occurs for any number of particles $k$~\cite{LL:CutoffASEP}. More recently, similar mixing time results were achieved for exclusion processes on a segment in a weakly asymmetric setup and random environments \cite{LL:CutoffWeakly,LY:RandomEnvironment,S:MixingBallistic}. \\
 
 Note that in all of the above mentioned works, the respective exclusion processes on the segment are reversible Markov chains. For symmetric exclusion processes on general graphs, where the chain is reversible with respect to the uniform distribution, mixing times are studied by Oliveira in \cite{O:MixingGeneral}; see also \cite{HP:EPmixing} and \cite{L:CutoffCircle} for refined results. Very recently, for reversible symmetric exclusion processes on general graphs with reservoirs, a simple characterization for the occurrence of cutoff was provided by Salez \cite{S:CutoffExclusion}. 
 
 Note that the TASEP with open boundaries is not reversible, and its invariant measure $\mu$ has in general not a simple closed form. This makes the TASEP with open boundaries a so-called non-equilibrium system, which is of great importance in statistical mechanics, and intensively studied over the last decades \cite{GE:ExactSpectralGap,DEHP:ASEPCombinatorics,S:DensityProfilePASEP,USW:PASEPcurrent,UW:Correlations}. For non-equilibrium particle-conserving chains, only a few results on mixing times are known, e.g.\  \cite{SS:TASEPcircle} for the TASEP on the circle.  Gantert et al.\ study various regimes of non-reversible symmetric and asymmetric simple exclusion processes with open boundaries~\cite{GNS:MixingOpen}. This includes upper and lower bounds of order $N$ on the mixing time in the high and low density phase. For the TASEP with open boundaries, the mixing time was shown in  \cite{S:MixingTASEP} and \cite{SS:TASEPcircle} to be of order $N^{3/2}$ in the maximal current phase. These results are also studied numerically in  \cite{GKM:SamplingTASEP}. \\

For the TASEP, a key tool is its connection to (directed) last passage percolation. In this paper, we use an interpretation as a last passage percolation model on a slab, and rely on sharp moderate deviations results for last passage percolation on the full space and the half-quadrant \cite{BG:TimeCorrelation,BFO:HalfspaceStationary,BSV:SlowBond,B:ModerateDeviationsStationary,EGO:OptimalMoments}. The availability of exact formulas from integrable probability, together with probabilistic arguments, allows us to achieve remarkably sharp estimates on various quantities of the TASEP, for example on the coalescence of geodesics; see also   \cite{BBS:NonBiinfinite,BSS:Coalescence,DEP:FirstPassage,MSZ:EnvironmentGeodesic,Z:OptimalCoalescence} for related results.
Results of this form are believed to hold for a wide range of models, which are (conjecturally) in the KPZ universality phase; see \cite{C:KPZReview} for an overview. \\ %-- but are currently not available for these other classes of interacting particle systems.  \\

When the stationary distribution is not explicitly known, mixing times give the possibility to obtain quantitative bounds on the time it takes to sample from the stationary distribution. At this point, let us also mention that elaborate tools, like the matrix product ansatz, and more recently permutation tableaux and weighted Catalan paths for the TASEP with open boundaries, and staircase tableaux for general asymmetric exclusion processes were developed in order to characterize the stationary distribution of one-dimensional exclusion processes with open boundaries \cite{CW:CombinatoricsTASEP,CW:TableauxCombinatorics,M:TASEPCombinatorics}; see also \cite{NS:Approximate,S:ASEPReverse,WWY:ASEPshock} for very recent developments.

\subsection{Main ideas for the proofs}

As a reader's guide, let us now give a summary of the main ideas and concepts for the proofs. Throughout this article, we crucially rely on interpreting the TASEP with open boundaries as a last passage percolation model on a strip. We transfer well-known notions and results for last passage percolation on the full space and on the half-quadrant to our setup. To achieve this, we strongly rely on moderate deviation results for geodesics, as well as stochastic domination and coupling ideas in order to compare different last passage percolation environments. \\

For the upper bound in the high and low density phase, we build on a strategy recently introduced by Sly and the second author in \cite{SS:TASEPcircle} for the TASEP on the circle. Intuitively speaking, the ideas from \cite{SS:TASEPcircle} allow us to express the mixing time of the TASEP with open boundaries by  the coalescence time of certain geodesics in the strip started from two initial growth interfaces in a common last passage percolation environment; see Section \ref{sec:Couplings} for a definition of these terms. We argue in the high and low density phase that the coalescence time on a strip of width $N$ is with probability tending to $1$ as $N \rightarrow \infty$ at most $cN$ for some suitable constant $c>0$. More generally, we show that for any pair of geodesics between two points with a last passage time of at least $cN$, both paths must spend a significant part in a certain area close to one of the boundaries, where it is very likely that the two geodesics will intersect. \\ 

For the lower bound in the high and low density phase, we consider the dynamics started from the all empty and all full configuration, respectively. We argue that for every $\varepsilon>0$, there exists some $\delta,\delta^{\prime}>0$, depending only on $\varepsilon$, and a subinterval $I$ of size at least $\delta N$ such that the density of particles in $I$ does at time $t^{N}_{\text{mix}}(1/4)-\delta^{\prime} N$ with probability tending to $1$ as $N \rightarrow \infty$ not agree with the density under the stationary distribution. To show this, we use again moderate deviation results for last passage times and geodesics on the strip in order to determine particular areas where geodesics are likely to pass through.  \\

For the upper bound in the coexistence line, we rely on a strategy recently established by the second author in \cite{S:MixingTASEP} for mixing times of the TASEP with open boundaries in the triple point. In words, the arguments in \cite{S:MixingTASEP} ensure that the mixing time is bounded from above by the weight collected by a semi-infinite geodesic on the strip until it has switched between boundaries at least twice. We show that a semi-infinite geodesic in a  last passage percolation environment corresponding to the coexisting line is likely to collect at most a weight of order $N^2$ before switching between boundaries at least twice. Again, this is achieved using a suitable comparison to last passage percolation on the half-quadrant. \\

For the lower bound in the coexistence line, and in order to rule out that cutoff occurs, we introduce a novel rescaling argument for last passage percolation with two attractive boundaries. Starting from the all empty configuration, this allows us to argue that with probability at least $\frac{1}{2}+\varepsilon$, we see at least $(\frac{1}{2}+\delta)N$ particles at time $C\log(\varepsilon^{-1})N^2$ in the segment of length $N$,  for suitable constants $\delta,\varepsilon,C>0$, and all $N$ sufficiently large. Since by a symmetry argument the probability to see at least $N/2$ particles in the coexistence line under the stationary distribution is close to $\frac{1}{2}$, this allows us to conclude.

\subsection{Outline of the paper}

In the remainder of the introduction, we state related open questions. In Section \ref{sec:Couplings}, we discuss two couplings for the TASEP with open boundaries, one for different configurations, and one to last passage percolation on the slab, following \cite{S:MixingTASEP} for the maximal current phase. 
In Section \ref{sec:LPPestimates}, we collect estimates on last passage percolation. %in the full space and the half-quadrant.
While some of the presented results are well-known, we require improved bounds  on moderate deviations for half-quadrant last passage percolation. In Sections \ref{sec:UpperBoundHighLow} and \ref{sec:LowerBoundHighLow}, we give bounds on the mixing time in the high and low density phase. We give a precise localization of geodesics and moderate deviation estimates on the slab. We then  apply a recent random extension and time shift technique from \cite{SS:TASEPcircle}, which allows us to convert the geodesic estimates into mixing time bounds. For the coexistence line, we give in Section \ref{sec:UpperCoexistence} upper bounds on the mixing time. We relate the first time that a semi-infinite geodesic hits both boundaries of the slab to exit times of second-class particles by a coupling argument, building on ideas of \cite{S:MixingTASEP} for the triple point. 
For a lower bound on the mixing time in Section \ref{sec:LowerCoexistence}, we use a fine analysis of moderate deviations in order to compare the intersection of geodesics throughout the slab, together with a deep use of symmetry arguments.

\subsection{Open questions}

While we provide a full characterization for the TASEP with open boundaries, determining sharp mixing time estimates on the asymmetric simple exclusion process with open boundaries remains wide open. Apart from a bound of order $N$ in the (analogously defined) high and low density phase in Theorem 1.5 of \cite{GNS:MixingOpen}, as well as a bound of order $N^3$ in Theorem 1.6 of \cite{GNS:MixingOpen} in the triple point, mixing times and cutoff are not known. 

\begin{conj}
For the mixing time of the asymmetric simple exclusion process on a segment of size $N$ and open boundaries, the following three claims hold:
\begin{itemize}
    \item[1.] In the high and low density phase, the mixing time is of order $N$, and cutoff occurs.
    \item[2.] In the maximum current phase, the mixing time is of order $N^{3/2}$, and no cutoff occurs.
    \item[3.] In the coexistence line, the mixing time is of order $N^2$, and no cutoff occurs.
\end{itemize} 
\end{conj}
The second statement is believed to hold in more generality for particle-conserving systems, including the TASEP on the circle; see also Theorem 1.1 and Conjecture 1.4  in \cite{SS:TASEPcircle}. 

\subsection{Constant policy}

Throughout the paper we regard $\alpha$ and $\beta$ as fixed and our emphasis
is on the behavior of the various quantities of interest as the parameter $N$ for the size of the segment goes to infinity. Constants such as $C, c$ denote positive numerical values which 
may depend on $\alpha$ and $\beta$, but are independent of all other parameters (in particular, of $N$). The constants $C, c$  are regarded as generic constants in the sense that their value may change from one appearance to the next, with the value of $C$ usually increasing and the value of $c$ decreasing. However, constants labeled with a fixed number, such as $C_0,c_0$, have a fixed value throughout the section or proof that they appear in.

\section{Couplings for the TASEP with open boundaries}\label{sec:Couplings}

We discuss the disagreement process and the interpretation of the TASEP with open boundaries as a last passage percolation model. Both representations are standard -- see for example Sections 2.1 and 3.1 in \cite{S:MixingTASEP}.

\subsection{Componentwise coupling of the TASEP with open boundaries}\label{sec:Disagreement}

In this section, we discuss a  way to compare different initial configurations for the TASEP with open boundaries, which allow us to give upper bounds on the mixing time. Our key tool is the \textbf{canonical coupling}, also called \textbf{basic coupling}, for the TASEP with open boundaries. \\

For  boundary parameters $\alpha,\beta>0$, let $(\eta_t)_{t \geq 0}$ and $(\zeta_t)_{t \geq 0}$ be two TASEPs with open boundaries on a segment of size $N$, started from initial configurations $\eta$ and $\zeta$, respectively. For all sites $x\in [N-1]$, we choose the same rate $1$ Poisson clocks, and when a clock rings at  $x$, we move a particle from $x$ to $x+1$ in both processes, provided the target site is empty. For the boundaries, we consider independent rate $\alpha$ and rate $\beta$ Poisson clocks, and whenever one of the two clock rings, we replace site $1$ (site $N$) by a particle (empty site), irrespective of its current occupation. \\

We denote in the following by $\mathbf{P}$ the law of this coupling. Observe that whenever we start with two configurations $\eta$ and $\zeta$, which satisfy a component-wise ordering $\succeq$, i.e.\ 
\begin{equation}
    \eta \succeq \zeta \qquad \Leftrightarrow \qquad \eta(i) \geq \zeta(i) \text{ for all } i \in [N] \, ,
\end{equation} then this ordering is preserved by the above coupling. In other words, this means that
\begin{equation}
    \mathbf{P}\left( \eta_t \succeq \zeta_t \text{ for all } t \geq 0\mid \, \eta_0 \succeq \zeta_0 \right) = 1 \, .
\end{equation} Whenever $\eta_0 \succeq \zeta_0$ holds and the respective exclusion processes are coupled according to $\mathbf{P}$, we define the \textbf{disagreement process} $(\xi_t)_{t \geq 0}$ between  $(\eta_t)_{t \geq 0}$ and  $(\eta_t)_{t \geq 0}$ as
\begin{equation}
    \xi_t(x) := \mathds{1}_{\eta_t(x)=\zeta_t(x)=1} + 2 \mathds{1}_{\eta_t(x)\neq \zeta_t(x)}
\end{equation} for all $x \in [N]$ and $t\geq 0$. We say that a site $x$ is occupied by a \textbf{second-class particle} at time $t$ if $\xi_t=2$. The following standard lemma, which can found as Lemma 2.2 in \cite{S:MixingTASEP} or Corollary 2.5 in \cite{GNS:MixingOpen}, relates the disagreement process to the mixing time.

\begin{lem}\label{lem:DisagreementProcess}
Consider the disagreement process between two TASEPs with open boundaries under the canonical coupling $\mathbf{P}$, started from $\eta= \mathbf{1}:=(1,1,\dots,1)$ and $\zeta=\mathbf{0}:=(0,0,\dots,0)$, respectively. Let $\tau$ be the first time that $\xi_{\tau}$ contains no second-class particles. If $\mathbf{P}(\tau>s) \leq \varepsilon$ for some $\varepsilon>0$ and $s\geq 0$, then the TASEP with open boundaries satisfies $t^N_{\text{\normalfont mix}}(\varepsilon)\leq s$.
\end{lem}

\subsection{The TASEP with open boundaries as a last passage percolation model}\label{sec:TASEPLPPcorres}

Fix some $\alpha,\beta >0$. We define in the following directed last passage percolation on the slab 
\begin{equation}
    \mathcal S _N:=\big\{ (x,y)\in \mathbb Z ^2 :   y\le x \le y+N  \big\}.
\end{equation} with upper and lower boundaries
\begin{equation}\label{def:Boundaries12}
    \partial _1 (\mathcal S_N):=\big\{ (x,x) :  x\in \mathbb Z   \big\} ,\quad \text{and} \quad \partial _2 (\mathcal S _N):=\big\{ (x+N,x) :  x\in \mathbb Z   \big\} \, ,
\end{equation} where we set $\partial (\mathcal S _N) =\partial _1 (\mathcal S _N) \cup \partial _2(\mathcal S _N)$. 
% and $\text{Int}(\mathcal S _N):=A_N\setminus  \partial ( \mathcal S _N )$ be the interior of the set.
%On any $v\in \text{Int}(\mathcal S _N)$ there is an $\exp (1)$ random variable and on any $v\in \partial \mathcal S  _N$ there is an $\exp (\alpha )$ random variable. All of them are independent.
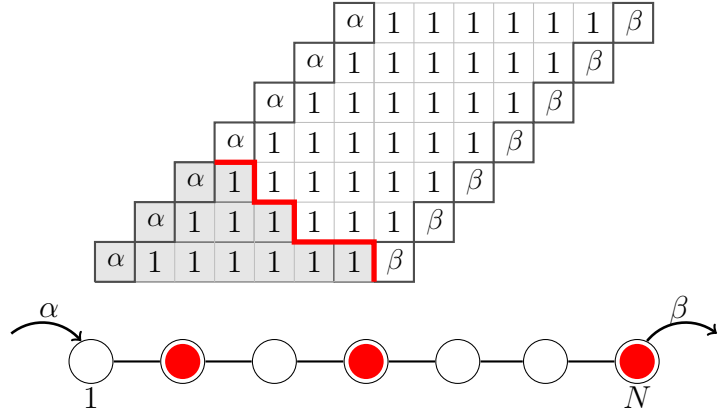
\begin{figure}
    \centering
    
    \begin{tikzpicture}[scale=0.53]

\filldraw [fill=gray!20] (1,1) rectangle ++(6,1);
\filldraw [fill=gray!20] (2,2) rectangle ++(3,1);

\filldraw [fill=blue!20] (3,3) rectangle ++(1,1);

\filldraw [fill=gray!20] (3,3) rectangle ++(1,1);

\filldraw [fill=blue!20] (7,1) rectangle ++(1,1);

\filldraw [fill=gray!20] (7,1) rectangle ++(1,1);

\filldraw [fill=gray!20] (5,2) rectangle ++(1,1);

\filldraw [fill=gray!20] (4,3) rectangle ++(1,1);

%
%\visible<8->{
%\filldraw [fill=gray!20] (6,2) rectangle ++(1,1);
%
%}
%
%
%\visible<9->{
%\filldraw [fill=gray!20] (8,1) rectangle ++(1,1);
%
%}
%

\foreach \x in{1,...,8}{
	\draw[gray!50,thin](\x,\x) to (7+\x,\x); 
	\draw[gray!50,thin](\x,1) to (\x,\x);
	\draw[gray!50,thin](7+\x,\x) to (7+\x,8);  }
	
	\foreach \x in{1,...,7}{
	
	\draw[black!70,thick](\x,\x) -- (\x+1,\x) -- (\x+1,\x+1) -- (\x,\x+1)-- (\x,\x);
	\draw[black!70,thick](\x+7,\x) -- (\x+1+7,\x) -- (\x+1+7,\x+1) -- (\x+7,\x+1)-- (\x+7,\x);
	}

%\draw[darkblue,line width =2pt,dashed] 	 (6.5,6.5) -- ++(0,1.5);	
%\draw[darkblue,line width =2pt,dashed] 	 (9.5,2.5) -- ++(1.5,0);	

	\node[scale=0.9]  (x1) at (1.53,1.55){$\alpha$} ;
	\node (x2) at (2.5,1.5){$1$} ;
	\node (x3) at (3.5,1.5){$1$} ;
	\node (x4) at (4.5,1.5){$1$} ;
	\node (x5) at (5.5,1.5){$1$} ;
	\node (x6) at (6.5,1.5){$1$} ;
	\node (x7) at (7.5,1.5){$1$} ;
	\node[scale=0.9] (x8) at (8.53,1.5){$\beta$} ;
	
	\node[scale=0.9] (x1) at (2.53,2.55){$\alpha$} ;
	\node (x2) at (3.5,2.5){$1$} ;
	\node (x3) at (4.5,2.5){$1$} ;
	\node (x4) at (5.5,2.5){$1$} ;
	\node (x5) at (6.5,2.5){$1$} ;
	\node (x6) at (7.5,2.5){$1$} ;	
	\node (x7) at (8.5,2.5){$1$} ;	
	\node[scale=0.9] (x8) at (9.53,2.5){$\beta$} ;	
	
%	\node[darkblue,scale=1.5] (y) at (11.5,2.5){$\gamma_t$} ;		

%	\node[red!40,scale=1.5] (y) at (3.5,5.5){$\mathcal{A}_t$} ;			
	
	\node[scale=0.9] (x1) at (3.53,3.55){$\alpha$} ;
	\node (x2) at (4.5,3.5){$1$} ;
	\node (x3) at (5.5,3.5){$1$} ;
	\node (x4) at (6.5,3.5){$1$} ;
	\node (x3) at (7.5,3.5){$1$} ;
	\node (x4) at (8.5,3.5){$1$} ;
	\node (x5) at (9.5,3.5){$1$} ;
	\node[scale=0.9] (x6) at (10.53,3.5){$\beta$} ;	
	
	\node[scale=0.9] (x1) at (4.53,4.55){$\alpha$} ;
	\node (x2) at (5.5,4.5){$1$} ;
	\node (x3) at (6.5,4.5){$1$} ;	
	\node (x3) at (7.5,4.5){$1$} ;	
	\node (x4) at (8.5,4.5){$1$} ;
	\node (x5) at (9.5,4.5){$1$} ;	
	\node (x6) at (10.5,4.5){$1$} ;
	\node[scale=0.9] (x7) at (11.53,4.5){$\beta$} ;

	\node[scale=0.9] (x1) at (5.53,5.55){$\alpha$} ;
	\node (x1) at (6.5,5.5){$1$} ;	
	\node (x2) at (7.5,5.5){$1$} ;
	\node (x3) at (8.5,5.5){$1$} ;	
	\node (x4) at (9.5,5.5){$1$} ;
	\node (x5) at (10.5,5.5){$1$} ;	
	\node (x6) at (11.5,5.5){$1$} ;
	\node[scale=0.9] (x7) at (12.53,5.5){$\beta$} ;	
	
		\node[scale=0.9] (x1) at (6.53,6.55){$\alpha$} ;
	\node (x1) at (7.5,6.5){$1$} ;	
	\node (x2) at (8.5,6.5){$1$} ;
	\node (x3) at (9.5,6.5){$1$} ;	
	\node (x4) at (10.5,6.5){$1$} ;
	\node (x5) at (11.5,6.5){$1$} ;	
	\node (x6) at (12.5,6.5){$1$} ;
	\node[scale=0.9] (x7) at (13.53,6.5){$\beta$} ;		
	
	\node[scale=0.9] (x1) at (7.53,7.55){$\alpha$} ;		
	\node (x2) at (8.5,7.5){$1$} ;	
	\node (x3) at (9.5,7.5){$1$} ;
	\node (x4) at (10.5,7.5){$1$} ;	
	\node (x5) at (11.5,7.5){$1$} ;
	\node (x6) at (12.5,7.5){$1$} ;	
	\node (x7) at (13.5,7.5){$1$} ;
	\node[scale=0.9] (x8) at (14.53,7.5){$\beta$} ;		

	\node[shape=circle,scale=1.5,draw] (B) at (2.3-1.4,-1){} ;
	\node[shape=circle,scale=1.5,draw] (C) at (4.6-1.4,-1) {};
 	\node[shape=circle,scale=1.5,draw] (D) at (6.9-1.4,-1){} ; 
	\node[shape=circle,scale=1.5,draw] (E) at (9.2-1.4,-1) {};
 	\node[shape=circle,scale=1.5,draw] (F) at (11.5-1.4,-1){} ;  	
 	\node[shape=circle,scale=1.5,draw] (G) at (13.7-1.4,-1){} ;  
 	\node[shape=circle,scale=1.5,draw] (H) at (16-1.4,-1){} ;  
 	
    \draw [->,line width=1pt] (2.3-1.4-2,0.7-1) to [bend right,in=135,out=45] (B);	
    \draw [->,line width=1pt] (H) to [bend right,in=135,out=45] (16-1.4+2,0.7-1);
 	
   \node (text3) at (2.3-1.4-1.05,0.3){$\alpha$}; 
   \node (text4) at (16-1.4+1.05,0.3){$\beta$};

    \node (text1) at (2.3-1.4,-1.9){$1$} ;    
    \node (text1) at (16-1.4,-1.9){$N$} ;    	
 	
%
%   \node[line width=0pt,shape=circle,scale=1.6] (B2) at (2.3,0){};
%	\node[line width=0pt,shape=circle,scale=2.5] (D2) at (6.9,0){};
%		\node[line width=0pt,shape=circle,scale=2.5] (Z2) at (-2.3,0){};
%		
%	\node[line width=0pt,shape=circle,scale=2.5] (X10) at (6.8,0){};
%		\node[line width=0pt,shape=circle,scale=2.5] (X11) at (-2.2,0){};	
%			\node[line width=0pt,shape=circle,scale=2.5] (X12) at (8.4,0){};
%		\node[line width=0pt,shape=circle,scale=2.5] (X13) at (-3.8,0){};		

		\draw[thick] (B) to (C);	
  \draw[thick] (C) to (D);	 
  		\draw[thick] (D) to (E);	
  \draw[thick] (E) to (F);	 
 \draw[thick] (F) to (G);	  
 \draw[thick] (G) to (H);

	\node[shape=circle,scale=1.2,fill=red] (CB) at (9.2-1.4,-1) {};
	\node[shape=circle,scale=1.2,fill=red] (CB) at (16-1.4,-1) {};
  	\node[shape=circle,scale=1.2,fill=red] (CB) at (4.6-1.4,-1) {};
  	
  	\draw[red,line width =2pt] 	 (4,4) -- ++(1,0) -- ++(0,-1) -- ++(1,0) -- ++(0,-1) -- ++(1,0) -- ++(1,0) -- ++(0,-1) ;

\end{tikzpicture}

    \caption{Last passage percolation on the slab and the corresponding configuration of the TASEP with open boundaries.}
    \label{fig:LPPandTASEP}
\end{figure}
Let $(\omega_v)_{v \in \mathcal{S}_N}$ be a family of independent Exponential distributed random variables. For $v \in \partial _1 (\mathcal S_N)$, we let $\omega_v$ have rate $\alpha$, when $v \in \partial _2 (\mathcal S_N)$ then each $\omega_v$ has rate $\beta$. For $v \in \mathcal{S}_N \setminus \partial (\mathcal S_N)$, we assign independent rate $1$ Exponential random variables.  With a slight abuse of notation, let $\succeq$ be the component-wise ordering on $\Z^2$. For $v\succeq u$, let $\pi(u,v)$ be a directed up-right \textbf{lattice path} from $u$ to $v$
\begin{equation*}
\pi(u,v) = \{ z^0=u, z^{1},\dots, z^{|v-u|}=v \, \colon \, z^{i+1}-z^{i} \in \{ \eone,\etwo\} \text{ for all } i \} \, ,
\end{equation*} where $\eone:=(1,0)$ and $\etwo := (0,1)$. For all $A \subseteq \Z^2$, we let $\Pi_{A}^{u,v}$ contain all lattice paths connecting $u$ to $v$, which are fully contained in $A$, and set
\begin{equation}\label{def:LastPassageTime}
T_{\alpha,\beta}(u,v):= \max_{\pi \in \Pi_{\mathcal{S}_N}^{u,v}} \sum_{z \in \pi \setminus \{v\}} \omega_z 
\end{equation} as the \textbf{last passage time} from $u$ to $v$ in the slab $\mathcal{S}_N$. We write $T(u,v)$ whenever the value of $\alpha$ and $\beta$ is clear from the context, and $T(\pi)$ for the passage time along a fixed lattice path~$\pi$.
%Note in contrast to the standard definition, we exclude the weight at $v$ in the definition of $T_{u,v}$ to allow for a simple concatenation of last passage times along paths. 
We refer to a path $\gamma(u,v)$ maximizing the right hand in side in \eqref{def:LastPassageTime} as a \textbf{geodesic}, and denote for all $x\in \mathbb \R$, and fixed $N$ 
\begin{equation}\label{def:PointsPQ}
    p_x:=(\lfloor x \rfloor,\lfloor x \rfloor) \in \partial_1(\mathcal{S}_N) \quad \text{and} \quad  q_x:=(\lfloor x+ N/2 \rfloor,\lceil x-  N/2\rceil) \in \partial_2(\mathcal{S}_N) \, .
\end{equation}
A standard property of geodesics, which we will frequently use later on without explicitly mentioning, is that the restriction of a geodesic between two sites is again a geodesic. 
%
%
%%In the remainder of this paper, we use the following notations for last passage times. When $u=(x,x)$ and $v=(y,y)$ for some $x,y \in \Z$, we let $T_{\mathbf{x} ,\mathbf{y}}:=T_{u,v}$. For $u=(u_1,u_2)$  and $v=(v_1,v_2)$ with $u,v\in \R^2$, we set
%\begin{equation}
%T_{u,v} = T_{ (\lfloor u_1 \rfloor,\lfloor u_2 \rfloor), (\lfloor v_1 \rfloor,\lfloor v_2 \rfloor)} \, .
%\end{equation}  
%For $A,B\subseteq \Z^2$, we let the last passage time $T_{A,B}$ between $A$ and $B$ be given by
%\begin{equation}\label{def:LastPassageTimeSets}
%T_{A,B} := \sup \Big\{ T_{u,v}  \colon u\in A, v\in B \text{ and } v \succeq u \Big\} \, ,
%\end{equation} provided that $A$ and $B$ contain at least one comparable pair of sites, and we let $\gamma
%_{A,B}$ be the corresponding geodesic if the supremum in \eqref{def:LastPassageTimeSets} is attained.  \\
For last passage percolation on $\Z^2$, the correspondence to the TASEP on the integers $\Z$ is well-known -- see \cite{R:PDEresult} -- and we have the following similar construction for the slab. Fix an initial configuration $\eta_0$ in the state space $\Omega_{N}$ for some $N\in \N$. Let $G_0=\{g_0^{i} \in \Z^2 \colon i \in \Z\}$ be the  \textbf{initial growth interface} with $g^0_0:= p_0$, and recursively
\begin{equation}\label{def:GrowthInterface}
g_0^{i} := \begin{cases} g_0^{i-1} + \eone & \text{ if } \eta(i)=0 \\
 g_0^{i-1} - \etwo & \text{ if } \eta(i)=1
\end{cases}
\end{equation} for all $i \geq 1$; see also Figure \ref{fig:LPPandTASEP}. For all $t \geq 0$, we consider the random variables 
\begin{equation}\label{def:GrowthInterface2}
G_t = \{g_t^{i-1} \in \Z^2 \colon i \in [N+1] \} =\{ u \in \Z^2 \colon \max_{w\in G_0}T(w,u) \leq t \text{ and }  \max_{w\in G_0}T(w,u+(1,1))> t\}  
\end{equation} with the convention that $g_t^{0}=p_x$ for some $x\in \Z$, and
\begin{equation}
g_t^{i} - g_t^{i-1} \in \{ \eone, -\etwo \}
\end{equation} for all $i\in [N]$. The process $(G_t)_{t \geq 0}$ is called the \textbf{growth interface} with respect to $(\omega_v)_{v \in \mathcal{S}_N}$.
The next lemma can be found as Lemma 3.1 in \cite{S:MixingTASEP}, so we omit the proof.
\begin{lem}\label{lem:CornerGrowthRepresentation} Let $N\in \N$, and let $(\eta_t)_{t\geq0}$ be a TASEP with open boundaries with respect to boundary parameters $\alpha$ and $\beta$. There exists a coupling between $(\eta_t)_{t\geq0}$ and the  environment $(\omega_v)_{v \in \mathcal{S}_N}$ such that the respective growth interface $(G_t)_{t \geq 0}$ of $(\eta_t)_{t\geq0}$ satisfies almost surely for all $t\geq 0$ and $i\in [N]$
\begin{equation}
\{ \eta_t(i) = 0 \} = \{ g_t^{i} - g_t^{i-1}  = \eone \} \, .
\end{equation}
\end{lem}
In the remainder, we focus on the last passage percolation representation of the TASEP with open boundaries, as well as related last passage percolation models.

\section{Last passage percolation on the full quadrant and the half-quadrant}\label{sec:LPPestimates}

In this section we review some results from exponential last passage percolation on $\mathbb Z ^2$ and on the half-quadrant. For the half-quadrant, we also obtain improved moderate deviation estimates of last passage times and properties about the geometry of the corresponding geodesics.

\subsection{Exponential last passage percolation on $\mathbb Z ^2$}\label{sec:exponential lpp}

Consider (directed) last passage percolation on $\mathbb Z ^2$. In this model, for any $x\in \mathbb Z^2$, we assign an independent Exponential-$1$-random variable $\omega _x$. For sites $v,u \in \Z^2$ with $v \succeq u$, the corresponding last passage time in $\Z^2$ is defined to be 
\begin{equation}
    Q(u,v):=\max_{\pi \in \Pi_{\Z^2}^{u,v}} \sum _{z\in \pi \setminus \{v\} }\omega _z \, .
\end{equation} 
The following theorem is due to Ledoux and Rider \cite{LR:BetaEnsembles}.
\begin{thm}\label{thm:Ledoux}
There exist universal constants $c,C>0$ such that for all $x=(x_1,x_2) \in \Z^2$ and $y=(y_1,y_2) \in \Z^2$ with $y_1-x_1 \geq y_2-x_2 >0$, 
\begin{equation}
    \mathbb P  \Big( \big| Q(x,y) -\big(\sqrt{y_1-x_1}+\sqrt{y_2-x_2}\big)^2 \big|\ge t (y_1-x_1)^{\frac{1}{2}}  (y_2-x_2)^{-\frac{1}{6}}\Big) \le Ce ^{-ct}
\end{equation} for all $t\geq 0$ sufficiently large.
\end{thm}

\subsection{Exponential last passage percolation on the half-quadrant}

Next, we define (directed) last passage percolation on the half-quadrant. Fix some $\alpha >0$. We consider the set $\mathcal H:=\{(x_1,x_2)\in \mathbb Z ^2: x_1\ge x_2\}$, and assign again a set of independent random variables $(\omega _v)_{v\in \mathcal{H}}$. Here, we have that $\omega _{v}$ is an Exponential-$\alpha$-random variable when $\omega _v=p_i$ for some $i\in \Z$, and an Exponential-$1$-random variable otherwise. As before, for $u,v\in \Z^2$ with $v \succeq u$, we define the half-quadrant last passage times
\begin{equation}\label{def:RestrictedLPTs}
    H(u,v):=\max_{\pi \in \Pi_{\mathcal{H}}^{u,v}} \sum _{z\in \pi \setminus \{v\} }\omega _z \, .
\end{equation}
The following theorem is due to Baik et al.  \cite{BBCS:Halfspace}.
\begin{thm}\label{thm:corwin}
Suppose that $\alpha <1/2$. Then 
\begin{equation}
    \frac{1}{\sqrt{n}} \big( H(0,p_n)- n/\rho _\alpha  \big) \overset{d}{\longrightarrow } N(0,\sigma ^2 ),
\end{equation}
where $\rho _{\alpha }:=\alpha (1-\alpha )$ and $\sigma ^2: =(1-2\alpha )/(\alpha ^2 (1-\alpha )^2)$.
\end{thm}

Next, we require the following variance estimate on the  half-quadrant last passage times.

\begin{lem}\label{lem:variance bound} Fix $\alpha<\frac{1}{2}$. Then there exists some universal constant $C>0$ such that $\big| \mathbb E [H(0,p_n)] -n/\rho _\alpha  \big|\le C\sqrt{n}$ and $\text{Var}\big( H(0,p_n) \big) \le Cn$ for all $n\in \N$.
\end{lem}

While the methods in \cite{BBCS:Halfspace}, combined with results from \cite{B:ModerateDeviationsStationary} on stationary last passage percolation, can be used to provide a bound on the above variance -- see also Section 4 in \cite{S:MixingTASEP} -- we include a self-contained elementary proof of Lemma \ref{lem:variance bound} in the appendix. Next, we turn to improve the last estimate to the moderate deviation regime. The result we obtain is not tight,  but suffices for our proof.

\begin{prop}\label{prop:123}
Suppose that $\alpha <\frac{1}{2}$. Then, for all $n\in \mathbb N$ and $t\ge \log ^3 n$
\begin{equation}
    \mathbb P \big( \big| H(0,p_n) - n/\rho _{\alpha }  \big| \ge t\sqrt{n} \big) \le Ce^{-c\sqrt{t}} 
\end{equation}
\end{prop}

\begin{proof}
For $v\in \mathcal H$, we define a modified environment $\tilde{\omega} _v:=\min (\omega _v,\sqrt{t})$. Let $\mathcal A$ be the event that for all $v=(v_1,v_2)\in \mathcal H $ with $v_1\le n$, we have $\tilde{\omega} _v=\omega _v$. Clearly, it holds that $\mathbb P (\mathcal A ) \ge 1-n^2 e^{-\sqrt{t}}$. Let $\tilde{H}(x,y)$ be the last passage times on the half-quadrant computed with respect to the modified environment. We start by proving that for all $t>0$ sufficiently large, and some constant $C_0>0$
\begin{equation}\label{eq:3}
    \mathbb P \bigg( \big| \tilde{H}(0,p_n)- \mathbb E \big[ \tilde{H}(0,p_n) \big] \big| \ge \frac{t\sqrt{n}}{2} \bigg) \le C_0e^{-ct} \, .
\end{equation}
To this end, for $k\le 2n$, we define the random variables $Y_k:=(\omega _{(x_1,x_2)}: x_1+x_2=k)$. The passage time $\tilde{H}(0,p_n)$ is a $\sqrt{t}$-Lipschitz function of the variables $Y_k$ for $k\le n$. Indeed, for all $k\le n$ changing the value of $Y_k$ changes the passage time $\tilde{H}(0,p_n)$ by at most $\sqrt{t}$. Thus, \eqref{eq:3} follows from Azuma's inequality; see \cite[Theorem~16]{chung2006concentration}. 
Next, on the event $\mathcal A $ we have that $\tilde{H}(0,p_n)=H(0,p_n)$, and therefore
\begin{equation}\label{eq:4}
    \mathbb P \bigg( \big| H(0,p_n)- \mathbb E \big[ \tilde{H}(0,p_n) \big] \big| \ge \frac{t\sqrt{n}}{2} \bigg) \le Cn^2e^{-c\sqrt{t}} \, .
\end{equation}
Integrating \eqref{eq:4} with respect to $t$ we obtain 
\begin{equation}
0 \leq   \mathbb E \big[ H(0,p_n) \big] -  \mathbb E \big[ \tilde{H}(0,p_n) \big] \leq \mathbb E \big[ |H(0,p_n)-\tilde{H}(0,p_n)| \big] \le C\sqrt{n}\log ^2n.
\end{equation}
%
%Substituting $t= 2 C_1\log ^2 n$ for a sufficiently large constant $C_1>0$, we obtain 
%\begin{equation}
%    \mathbb P \Big( \big| H(0,p_n)- \mathbb E \big[ \tilde{H}(0,p_n) \big] \big| \ge C_1\log^2 n \sqrt{n} \Big) \le \frac{1}{2} \, .
%\end{equation}
%It follows from the first statement in Lemma \ref{lem:variance bound} that 
for some constants $C_2,C_3>0$.
%\begin{equation}
% \big| \mathbb E \big[ \tilde{H}(0,p_n) \big] - n/\rho _{\alpha } \big| \le C_2 \log^2 n \sqrt{n} \, .
%\end{equation}
Substituting this estimate back into \eqref{eq:4}, we get for all $t\ge \log ^3 n$ that 
\begin{equation}
    \mathbb P \big( \big| H(0,p_n)- n/\rho _{\alpha } \big| \ge t\sqrt{n} \big) \le Ce^{-c\sqrt{t}} 
\end{equation}
for some $c,C>0$, as required to finish the proof.
\end{proof}

\subsection{Estimates on geodesics in the half-quadrant}
 The following statement  shows that the geodesic $\gamma (0,p_n)$ in half-quadrant last passage percolation is with high probability very close to the diagonal.

\begin{prop}\label{cor:1}
Fix $\alpha<\frac{1}{2}$, and consider half-quadrant last passage percolation. There exist constants $c,C>0$ such that for all $n \in \N$, we have that
 \begin{equation}\label{eq:2}
     \mathbb P \bigg( \begin{matrix} \exists x_1,x_2\in \mathbb N \text{ with }  x_1>x_2+\log ^8 n  \\
    \text{ such that } (x_1,x_2)\in \gamma (p_0,p_n)
    \end{matrix}
    \bigg) \le C\exp (-c\log ^2 n) \, .
 \end{equation}
\end{prop}

Observe that for the event on the left hand side of \eqref{eq:2} to occur, the geodesic from $p_0$ to $p_n$ must clearly avoid the diagonal $\{p_x \colon x \in \Z \}$ for some part of length at least $\frac{1}{2}\log^{8}(N)$ before touching it again. The following lemma shows that this is very unlikely. 

\begin{lem}\label{lem:1} In the setup of Proposition \ref{cor:1}, there exist $c,C>0$ such that for all $m \in \N$
\begin{equation}
    \mathbb P \big(  p_k\notin \gamma (0,p_m) \text{ for all } k \in [m-1] \big) \le C\exp \big( -cm^{1/4} \big) \, .
\end{equation}
\end{lem}

\begin{proof}
Define a modified environment in the following way. For any $x_1,x_2\in \mathbb Z$ if $x_1>x_2$ we let $\tilde{\omega} _{(x_1,x_2)}:=\omega _{(x_1,x_2)}$, and if $x_1\le x_2$, let $\tilde{\omega} _{(x_1,x_2)}$ be Exponential-$1$-distributed, independently of the other weights. Note that the modified environment corresponds to last passage percolation on the full space, and let, with a slight abuse of notation,  $Q(u,v)$ be the respective last passage time in the environment $\tilde{\omega}$ between sites $u,v$. On the event in the lemma, we clearly have that $H(0,p_m)\le Q(0,p_m)+ \omega _{0}$. Intuitively, note that this inequality can only hold if either $Q(0,p_m)$ is much larger than expected or $H(0,p_m)$ is much smaller than expected. Indeed, recalling \eqref{def:ABs}, we let $\varepsilon :=\rho _\alpha^{-1} -4>0$, and get for some constant $c>0$ 
\begin{equation}
\begin{split}
    \mathbb P \big( H(0,p_m)\le Q(&0,p_m)+ \omega _{0} \big) \le \mathbb P \big( H(0,p_m)\le m/\rho _\alpha - \varepsilon m/4 \big)+ \\
    &+\mathbb P \big( Q(0,p_m)\ge 4m + \varepsilon m/4 \big)
    +\mathbb P \big( \omega _{0} \ge \varepsilon m/4 \big) \le \exp \big( -cm^{1/4} \big) \, ,
\end{split}
\end{equation}
where in the last inequality, we used Theorem~\ref{thm:Ledoux} and Proposition~\ref{prop:123}.
\end{proof}

%We can now show the above proposition on the location of geodesics.

\begin{proof}[Proof of Proposition~\ref{cor:1}]
By Lemma~\ref{lem:1} for $m=\frac{1}{2}\log^{8}(n)$, translation invariance and a union bound over the all pairs of sites $(p_i,p_j)$ for some $1 \leq i \leq j \leq n$, we see that 
\begin{equation}\label{eq:1}
    \mathbb P \bigg( \begin{matrix} \exists  0\le i\leq j \le n \text{ with } j> i+\log ^8 n  \\
    \text{ such that } \forall k \in [i,j] \cap \Z , \ p_k\notin \gamma (p_i,p_j)
    \end{matrix}
    \bigg) \le C\exp (-c\log ^2 n) 
\end{equation} holds for some constants $c,C>0$. 
Recalling our observation that for the event on the left hand side of \eqref{eq:2} to occur, the geodesic from $p_0$ to $p_n$ must avoid the boundary for a part of length at least $\frac{1}{2}\log^{8}(n)$ before returning. Together with the fact that a geodesic restricted between two sites is again a geodesic, this finishes the proof since the event at the left hand side of \eqref{eq:2} is contained in the event at the left hand side of \eqref{eq:1}.
\end{proof}

%We have the following consequence on the last passage times in the half-quadrant.

\begin{cor}\label{claim:triangle}
 In the setup of Proposition \ref{cor:1}, there exist constants $c,C>0$ such that for all $0\le i< j <k \le n$, and $n$ large enough, we have that 
\begin{equation}\label{eq:triangle}
    \mathbb P \Big( \big| H(p_i,p_k)-H(p_i,p_j)-H(p_j,p_k)\big| \le \log ^{11}n \Big) \ge 1-Ce^{-c\log ^2n} \, .
\end{equation}
\end{cor}
 
\begin{proof}
From the definition in \eqref{def:RestrictedLPTs}, we clearly get that $H(p_i,p_k)\ge H(p_i,p_j)+H(p_j,p_k)$. We turn now to upper bound $H(p_i,p_k)$.
Define the event
\begin{equation}
    \mathcal{B} _0:= \big\{ \forall x=(x_1,x_2)\in \mathcal H \text{ with } x_1\le n, \ \omega _x\le \log ^2 n   \big\}
\end{equation}
and note that $\mathbb P (\mathcal{B} _0)\ge 1-Ce^{-\log ^2n}$ for some constant $C>0$. Next, let $\mathcal{B}_1$ be the event from Proposition~\ref{cor:1}. Let $z$ and $z'$ be the first intersections of the geodesic $\gamma (p_i,p_k)$ with the lines $\{(j,x_2) : x_2\in \mathbb R \}$ and $\{(x_1,j):x_1\in \mathbb R \}$, respectively,  and note that on the event $\mathcal{B} _1 $ we have that $||z-z'||_1\le 4\log ^8n$. Thus, on the event $\mathcal{B} _0\cap \mathcal{B} _1$ we have  
\begin{equation}
    H(p_i,p_k) = H(p_i,z)+H(z,z')+H(z',p_k) \le H(p_i,p_j)+ 8\log ^{10}n +H(p_j,p_k) \, .
\end{equation}
Using Proposition \ref{cor:1} to bound the probability of $\mathcal{B}_1$, this finishes the proof of \eqref{eq:triangle}.
\end{proof}

As another consequence of Proposition \ref{prop:123} and Proposition \ref{cor:1}, we have the following lemma which allows to us to compare last passage times on (subsets of) the strip with last passage times on the half-quadrant. To formulate the lemma, let $H^{(m)}$ for $m\in \N$ denote the last passage time on the half-quadrant restricted to $\mathcal{S}_m$, i.e.\ for all $u,v \in \mathcal{S}_m$ with $v \succeq u$
 \begin{equation}\label{def:ReallyRestrictedLPTs}
    H^{(m)}(u,v):=\max_{\pi \in \Pi_{\mathcal{S}_m}^{u,v}} \sum _{z\in \pi \setminus \{v\} }\omega _z \, .
\end{equation}
for the environment $(\omega _v)_{v\in \mathcal{H}}$ on $\mathcal{H}$ defined at the beginning of this section. 
\begin{lem}\label{lem:CombinationLemma}
There exist constants $c,C>0$ such that for all $n$ sufficiently large
\begin{equation}\label{eq:CombinedLemma1}
 \mathbb P \Big( H^{(m)}(p_0,p_n) = H(p_0,p_n) \text{ for all } m \geq \log^{8}(n) \Big) \ge 1-Ce^{-c\log ^2n} \, .
\end{equation}
Moreover, for all $\tilde{C}>0$, we can choose $N$ sufficiently large such that for any $n \geq c_1\log^{8}(N)$ 
\begin{equation}\label{eq:CombinedLemma2}
 \mathbb P \Big( 
H^{(n)}(p_i,p_j) \in \Big( \frac{i-j}{\rho _{\alpha }} - c_2\sqrt{N}\log^4(N), \frac{i-j}{\rho _{\alpha }} + c_2\sqrt{N}\log^4(N) \Big) \, \forall  i,j \in [\tilde{C}N]  \Big) \ge 1-c_3 e^{-c_4\log ^2N} \, .
\end{equation} with constants $(c_i)_{i \in [4]}$, depending only on $\tilde{C}$ and $\alpha$.
\end{lem}
\begin{proof}
Note that the first statement \eqref{eq:CombinedLemma1} is immediate from Proposition \ref{cor:1} as by restricting the space of available lattice paths, we will only decrease the last passage times. The second statement \eqref{eq:CombinedLemma2} is an immediate consequence of \eqref{eq:CombinedLemma1} and a union bound in Proposition \ref{prop:123} over all pairs of sites $(p_i,p_j)$ with $1 \leq i \leq j \leq \tilde{C}N$. 
\end{proof}

\subsection{Improved bounds on last passage times in the half-quadrant}

As a consequence of the above bounds on geodesics, we obtain improved bounds on last passage times in the half-quadrant. We start with an estimate on the expectation from Lemma~\ref{lem:variance bound}.

\begin{lem}\label{claim:expectation}
There exists some constant $C>0$ such that for all $n\in \N$
\begin{equation}
    \big| \mathbb E [H(0,p_n)] -n/\rho _\alpha  \big| \le C\log ^{12}n \, .
\end{equation}
\end{lem}

\begin{proof}
Using Corollary~\ref{claim:triangle} we obtain that for some constant $c>0$
\begin{equation}\label{eq:91}
    \mathbb P \Big( \Big| H(0,p_{n^2}) -\sum _{k=1}^{n} H(p_{(k-1)n},p_{kn}) \Big| \le n\log ^{12}n  \Big) \ge 1-e^{-c\log ^2 n} \, .  
\end{equation}
Moreover, by Lemma~\ref{lem:variance bound} and Chebyshev's inequality, there exists $C_1>0$ such that
\begin{equation}\label{eq:92}
    \mathbb P \Big( \Big|n\mathbb E [H(0,p_n)]- \sum _{k=1}^{n} H(p_{(k-1)n},p_{kn})  \Big| \le C_1 n  \Big) \ge 3/4
\end{equation}
and thus
\begin{equation}\label{eq:93}
    \mathbb P \big( \big| H(0,p_{n^2})-\mathbb E[H(0,p_{n^2})]  \big| \le C_1 n  \big) \ge 3/4 \, .
\end{equation}
Since with positive probability the events in \eqref{eq:91}, \eqref{eq:92} and \eqref{eq:93} hold simultaneously, uniformly in the choice of $n$, we must have
\begin{equation}
    \big| \mathbb E [H(0,p_{n^2})]-n \mathbb E [H(0,p_n)] \big| \le 2n\log ^{12}n \, .
\end{equation}
This finishes the proof of the claim using Lemma~\ref{lem:variance bound}.
\end{proof}

Next, we present two more refined bounds on the fluctuations of the half-quadrant last passage times. 

\begin{lem}\label{lem:brownian 2}
There exists $c,C>0$ such that for all $ D \ge 1$, we find $n_0=n_0(D)$ where for all $n\ge n_0$
\begin{equation}
    \mathbb P \Big(  \big| H(p_i,p_j) - (j-i)/\rho _{\alpha } \big| \le D\sqrt{n} \ \text{ for all }  i\le j\le n \Big) \ge 1-Ce^{-cD^2}.
\end{equation}
\end{lem}

For the proof of Lemma~\ref{lem:brownian 2}, we need the following martingale concentration inequality due to Freedman~\cite{freedman1975tail}. This version of the inequality is from \cite[Theorem~18]{chung2006concentration}.

\begin{thm}[Freedman's inequality]\label{thm:free}
    Let $M_k$ be a martingale and suppose that for all $k\le n $ we have $|M_{k+1}-M_k|\le M$ and ${\rm Var}(M_k \mid \mathcal F _{k-1}) \le \sigma _k^2$. Then, for all $\lambda >0$ we have 
\begin{equation}
    \mathbb P \big( |M_n-M_0|\ge \lambda   \big) \le 2\exp \Big( -\frac{\lambda ^2}{ 2M\lambda + 2\sum_{k=1}^n \sigma _k ^2} \Big).
\end{equation}
\end{thm}

We can now prove Lemma~\ref{lem:brownian 2}.

\begin{proof}[Proof of Lemma~\ref{lem:brownian 2}] In the following, let $c_i>0$ for $i\in [10]$ be suitable constants. Fix some $\tilde{C}>1/\rho _{\alpha }$ and let $m:= \lfloor n^{3/4} \rfloor $ and $k_1:=\lfloor n/m \rfloor $. For $k\le k_1$, we define the random variables
\begin{equation}
    X_k:=\min \big( H(p_{(k-1)m},p_{km}), \tilde{C} m \big) \, .
\end{equation}
Note that by Proposition~\ref{prop:123} we have that $X_k=H(p_{(k-1)m},p_{km})$ with probability at least $1-c_1\exp (-c_2n^{3/16})$. It follows that $\var (X_k) \le c_3m$. Define the martingale 
\begin{equation}
    M_k:=\sum _{j=1}^{k \wedge \tau }(X_j-\mathbb E [X_j]) \quad \text{where} \quad \tau :=\min \bigg\{ k\ge 1 : \Big| \sum _{j=1}^{k}(X_j-\mathbb E [X_j]) \Big| \ge D\sqrt{n}/5 \bigg\}  \, .
\end{equation}
By Freedman's inequality given in Theorem~\ref{thm:free} we have that 
\begin{equation}
    \mathbb P (\tau \le k_1) = \mathbb P \big(  |M_{k_1}| \ge D\sqrt{n}/5  \big) \le 2e^{-c_4D^2} \, .
\end{equation}
It follows that 
\begin{equation}
    \mathbb P \bigg( \Big| \sum _{k=k_2}^{k_3}  X_k-\mathbb E [X_k]  \Big| \le 2D\sqrt{n}/5  \text{ for all } k_2\le k_3 \le k_1\bigg)\ge 1-c_5e^{-c_6D^2} \, .
\end{equation}
Thus, by Lemma~\ref{claim:expectation}, recalling that  $X_k=H(p_{(k-1)m},p_{km})$ with probability at least $1-c_1\exp (-c_2n^{3/16})$, we obtain 
\begin{equation}
    \mathbb P \bigg(  \Big| (k_3-k_2)m/\rho _\alpha -  \sum _{k=k_2}^{k_3}  H(p_{(k-1)m},p_{km})  \Big| \le D\sqrt{n}/2  \ \text{ for all } k_2\le k_3 \le k_1 \bigg)\ge 1-c_7e^{-c_8D^2} \, .
\end{equation}
Denote by $\mathcal{C} _1$ the event in the last equation and define 
\begin{equation}
    \mathcal{C} _2 : = \big\{  \big| H(p_i,p_j) - (j-i)/\rho _{\alpha } \big| \le \sqrt{n}/5  \ \text{ for all }  i\le j\le n \text{ with } |j-i|\le n^{3/4} \big\} \, .
\end{equation}
By Proposition~\ref{prop:123} we have that $\mathbb P (\mathcal{C} _2 )\ge 1-c_9e^{-n^{c_{10}}}$. Finally, let $\mathcal{C} _3$ be the intersection of the events in Corollary~\ref{claim:triangle} over all $i\le j\le n$.
We claim that on $\mathcal{C} _1 \cap \mathcal{C} _2 \cap \mathcal{C} _3$ we have 
\begin{equation}\label{eq:745}
    \big| H(p_i,p_j) - (j-i)/\rho _\alpha  \big| \le D \sqrt{n} 
\end{equation} for all $i\le j\le n$.
Indeed, fix $i<j\le n$ and let $k_2:=\lceil i/m \rceil $ and $k_3:= \lfloor j/m \rfloor $. On $\mathcal{C} _3 $ 
\begin{equation}
    \Big| H(p_i,p_j) - H(p_i, p_{k_2m})-H(p_{k_3m},p_j) -\sum _{k=k_2}^{k_3}H(p_{(k-1)m},p_{km}) \Big| \le n^{1/4} \log ^{12}n
\end{equation}
holds, and by the definition of $\mathcal{C} _1$ and $\mathcal{C} _2$ we get \eqref{eq:745} on the event $\mathcal{C} _1 \cap \mathcal{C} _2 \cap \mathcal{C}_3$, allowing us to conclude.
\end{proof}

\begin{lem}\label{lem:brownian}
For all $\varepsilon \in (0,1)$, and all $n\ge n_0$ for some $n_0=n_0(\varepsilon )$, we have
\begin{equation}
    \mathbb P \Big( \big| H(p_i,p_j) -(j-i)/\rho _\alpha  \big| \le \varepsilon \sqrt{n}  \ \text{ for all }  i\le j\le n \Big) \ge \exp \big( -C/\varepsilon^2  \big) \, .
\end{equation}
\end{lem}

\begin{proof}
We omit some of the detail of the proof as it is similar to the proof of Lemma~\ref{lem:brownian 2}. Let $D_0>1$ be chosen such that for all $m$ sufficiently large 
\begin{equation}\label{eq:use of brownian 2}
    \mathbb P \Big(  \big| H(p_i,p_j) - (j-i)/\rho _{\alpha } \big| \le D_0\sqrt{m} \ \text{ for all }  i\le j\le m \Big) \ge 5/6 \, .
\end{equation}
The existence of such a $D_0$ is guaranteed by Lemma~\ref{lem:brownian 2}. Next, let $m:= \lfloor \varepsilon ^2 n/(99D_0^2) \rfloor $ and $k_1:=\lceil n/m  \rceil $. Define the random variables
\begin{equation}
 X_k:=H(p_{(k-1)m},p_{km})-m/
\rho _\alpha  
\end{equation}
for all $k \le k_1$, and note that $(X_k)$ are  independent. Next, define the events
\begin{align}
    \mathcal D _k &:= \big\{ |X_k|\le \varepsilon \sqrt{n}/5 \big\}\cap \bigg\{  \bigg(\sum _{j=1}^{k-1}X_j \bigg) X_k \le 0 \bigg\}  \\
    \mathcal E _k &:= \Big\{ \forall (k-1)m\le i<j\le km, \ \big| H(p_i,p_j)-(j-i)/\rho _\alpha  \big| \le \varepsilon \sqrt{n}/5 \Big\} \, .
\end{align}
By Theorem~\ref{thm:corwin} we have $\mathbb P (\mathcal D _k \ | \ \mathcal F _{k-1}) \ge 1/3$ where $\mathcal F_{k-1}=\sigma (X_1,\dots ,X_{k-1})$ is the sigma-algebra generated by $X_1,\dots ,X_{k-1}$. Moreover, the inequality in \eqref{eq:use of brownian 2}, the fact that $D_0\sqrt{m}\le \varepsilon \sqrt{n}/5$ and translation invariance yield $\mathbb P ( \mathcal E_k \ | \ \mathcal F _{k-1}) \ge 5/6$, and therefore $\mathbb P (\mathcal D _k \cap \mathcal E _k \ | \ \mathcal F _{k-1} ) \ge 1/6$. It follows that $ \mathbb P \big( \bigcap _{k=1}^{k_1} (\mathcal D _k \cap \mathcal E _k)  \big) \ge \exp (-C/\varepsilon ^2)$ where the constant $C$ depends only on $D_0$.  Note that on the intersection $\bigcap _{k=1}^{k_1}\mathcal D _k$, we have that $\big| \sum _{j=1}^kX_j \big| \le \varepsilon \sqrt{n}/5$ for all $k\le k_1$ and therefore $\big| \sum _{j=k}^{k'}X_j \big| \le 2\varepsilon \sqrt{n}/5$ for all $k,k'\le k_1$. Thus, using the same arguments as in the proof of Lemma~\ref{lem:brownian 2}, the event $\bigcap _{k=1}^{k_1} (\mathcal D _k \cap \mathcal E _k) $ implies the event of the lemma, allowing us to conclude. 
\end{proof}

\section{Upper bound in the high and low density phase} \label{sec:UpperBoundHighLow}

In this section, we prove the upper bound in Theorem \ref{thm:HighLow}. By the symmetry between particles and holes, we consider in the following only the high density phase, i.e.\ we assume throughout this section that $\beta<\min(\alpha,\frac{1}{2})$. Moreover, we fix $\alpha $ and $\beta $ with $\beta<\min(\alpha,\frac{1}{2})$ and allow the constants $C,c,c_1,\dots$ to depend on $\alpha $ and $\beta $.

\subsection{Strategy for the proof of Theorem \ref{thm:HighLow}}

Let us start by outlining the main steps in order to show the upper bound in Theorem \ref{thm:HighLow}. In order to apply the results on last passage percolation on the full space and on the half-quadrant from Section \ref{sec:LPPestimates} to last passage percolation on the strip, we first establish in Lemma \ref{lem:NoCrossing} a non-traversing result for geodesics in the high density phase.  In words,once a geodesic of length of order $N$ has reached the lower diagonal, it will with high probability not touch the upper diagonal and then come back to the lower diagonal. In particular, this allows us to ensure that once the geodesic has reached the lower diagonal, we can compare its properties to geodesics  on the half-quadrant. \\

Next, in Section \ref{sec:EstimatesOnLPP}, we show that the last passage times to a sufficiently far away point on the lower diagonal of the strip are concentrated. Our main result in this part is Proposition~\ref{pro:LineToBoundary} which states that along certain down-right paths in the strip, all last passage times to a given site on the lower diagonal are roughly equal. A key step is to determine the optimal slope of geodesics for hitting the lower boundary from a given point inside the slab, which we achieve in Lemma~\ref{lem:CoalescenceAlpha}. In Section \ref{sec:Simultaneous}, we  argue that we can place any pair of initially growth interfaces in such a way that the last passage times to all sufficiently far away sites on the lower diagonal agree up to an error of order $o(N)$; see Lemma~\ref{lem:InterfaceHittingDiagonal}. Moreover, Lemma~\ref{lem:CoalescenceAllGeodesic} and Corollary~\ref{cor:ImprovedHitting} ensures that in this case, the geodesics must intersect with high probability a particular part of the strip close to the lower diagonal. \\

In a last step, we use in Section \ref{sec:RandomExtensionTimeChange} the ideas from \cite{SS:TASEPcircle} in order to translate the results on the coalescence of geodesics to mixing time bounds by a random extension of the environment and a time change. Let us stress that in contrast to methods presented in \cite{SS:TASEPcircle}, we can only modify in the last passage percolation environment in a narrow part of the strip, and thus again crucially rely on moderate deviation results for  locating the geodesics in the high density phase; see also Corollary \ref{cor:ImprovedHitting} for a quantitative statement.  

\subsection{No traversing of geodesics}\label{sec:NoCrossing}

Recall from \eqref{def:Boundaries12} that we denote by $\partial_1 (\mathcal S_N)$ and  $\partial_2 (\mathcal S_N)$ the upper and the lower diagonal of the slab $\mathcal S_N$, and that $p_i\in \partial_1 (\mathcal S_N)$ and $q_i \in \partial_2 (\mathcal S_N)$ for all suitable $i$, respectively. The following is a key lemma for our proof, which will allow us in the upcoming sections to compare last passage times on the strip to last passage times on the half-quadrant. Recall that we assume $\beta < \min(\alpha,\frac{1}{2})$. 

\begin{lem}\label{lem:NoCrossing} Let $\tilde{C}>0$ be arbitrary, but fixed. Then there exist constants $c_1,c_2>0$, depending only on $\tilde{C}$ and $\beta$, such that 
\begin{equation}
    \Big\{ T(u,v) + T(v,w) <T(u,w) \text{ for all } u,w \in \{ p_i \colon |i|\leq \tilde{C}N \}    \text{ and } v \in \{ q_j \colon |j|\leq \tilde{C}N \}  \Big\} 
\end{equation} holds with probability at least $1-c_1 \exp(-c_2\log^2 N)$ for all $N$ sufficiently large.
\end{lem}

\begin{proof}
Note that by Lemma \ref{lem:CombinationLemma} with $n=N-1$, there exist constants $(c_i)_{i \in [3]}$ such that for all $N$ large enough 
\begin{align}\label{eq:BoundaryToBoundary}
 \P \left( T(q_x,q_w) \geq  \frac{y-x}{\rho_{\beta}} - c_3 \sqrt{N}\log^4 N \, \text{for all } |x|,|y| \leq \tilde{C}N\right) \geq 1-c_1 e^{-c_2\log^2 N}
\end{align}
 We show now that each path touching $\partial_1 (\mathcal S_N)$ yields with high probability a much smaller passage time than \eqref{eq:BoundaryToBoundary}. To do so, we consider the following path decomposition extending the ideas from Section 4 in \cite{S:MixingTASEP} for  boundary parameters $\alpha=\beta=\frac{1}{2}$; see also Figure \ref{fig:pathDecomposition}. \\

For a lattice path $\pi=(u=z_0,z_1,\dots,z_\ell=w)$ with some $\ell \in [4\tilde{C}N]$, let $\mathcal{I}_{1}$ and $\mathcal{I}_{2}$ denote the sets of indices, where $z_i$ is contained in the upper diagonal  $\partial_1(\mathcal S_N)$, respectively in the lower diagonal $\partial_2(\mathcal S_N)$. We apply a recursive decomposition of $\pi$ into sub-paths, which are partitioned into $\Pi_{1,1},\Pi_{1,2},\Pi_{2,1},\Pi_{2,2}$. A path in $\Pi_{i,j}$  connects the diagonal  $\partial_i(\mathcal S_N)$ to $\partial_j(\mathcal S_N)$. 
If $\mathcal{I}_{1}=\emptyset$, we let $\pi \in\Pi_{2,2}$. Otherwise, set 
\begin{equation}
i^1_{\min} := \min \{i \in \N \colon i \in \mathcal{I}_{1} \} \quad \text{ and } \quad
i^2_{\max} := \max \{ i < i^1_{\min} \colon i\in \mathcal{I}_{2} \}
\end{equation} to be the smallest index such that $i^1_{\min} \in \mathcal{I}_{1}$, and the largest index smaller than $i^1_{\min}$ such that $i^2_{\max} \in \mathcal{I}_{2}$. We add the path $(z_0,\dots,z_{i^2_{\max}})$ to $\Pi_{2,2}$, and 
$(z_{i^2_{\max}},z_{i^1_{\min}})$ to  $\Pi_{2,1}$. Let 
\begin{equation}
i^2_{\min} := \min \{i \geq i^1_{\min} \colon i \in \mathcal{I}_{2} \} \quad \text{ and } \quad
i^1_{\max} := \max \{ i < i^2_{\min} \colon i\in \mathcal{I}_{1} \} \, .
\end{equation} We add the path $(z_{i^1_{\min}},\dots,z_{i^1_{\max}})$ to $\Pi_{1,1}$, and the path 
$(z_{i^1_{\max}},z_{i^2_{\min}})$ to $\Pi_{1,2}$.
Note that the remaining path $(z_{i^2_{\min}},z_{\ell})$ is either empty or again a path between two sites in $\partial_2(\mathcal S_N)$. In the latter case, apply the above decomposition recursively for the remaining sub-path. %Suppose now that the path $\pi$ is a geodesic, and hence so are all the sub-paths in $\Pi_{1,1},\Pi_{1,2},\Pi_{2,1},\Pi_{2,2}$. Note that 
Note that by Lemma \ref{lem:CombinationLemma}, using that the passage time of a path $\pi$ is smaller than the respective last passage time between its endpoints, we see that for some $c_4,c_5,c_6>0$, with probability at least $1-c_4 \exp(-c_5\log^2 N)$ for all $N$ is sufficiently large
\begin{align}
 T(\pi) \leq  4n + c_6 \sqrt{N}\log^4 N  \text{ for all } \pi\in \Pi_{1,1} \text{ with } |\pi|=2n \text{ and } n \in [2\tilde{C}N]\label{eq:11PassageTimes}\\
 T(\pi) \leq  \frac{n}{\beta(1-\beta)} + c_6 \sqrt{N}\log^4 N  \text{ for all } \pi\in \Pi_{2,2} \text{ with } |\pi|=2n \text{ and } n \in [2\tilde{C}N] . \label{eq:22PassageTimes}
\end{align} 
Similarly, by Theorem \ref{thm:Ledoux} and a union bound over all pairs $(q_x,p_y)$ with $|x|,|y| \leq \tilde{C}N$ (and a standard tail bound for the random variables at the endpoints) we have that for some constants $c_7,c_8,c_9>0$, with probability at least $1-c_7 \exp(-c_8\log^2 N)$
\begin{align} \label{eq:1221PassageTimes}
 T(\pi) \leq 4n +  c_9 \sqrt{N}\log^2 N  \text{ for all } \pi\in \Pi_{1,2}\cup\Pi_{2,1}  \text{ with } |\pi|=2n \text{ and } n \in [2\tilde{C}N]  \, ,
\end{align} uniformly in the choice of $\pi$.
Observe that when $\pi$ intersects $\partial_1(\mathcal{S}_N)$, we see at least two paths of length at least $N$, which are contained in $\Pi_{1,2} \cup \Pi_{2,1}$. Let $| \Pi_{i,j}|$ be the number of paths in $\Pi_{i,j}$, using the above decomposition for a given path $\pi$, and observe that
\begin{equation}\label{eq:PathLengths1221}
| \Pi_{1,2} | + | \Pi_{2,1} | \geq | \Pi_{1,1} | + | \Pi_{2,2} | \, .
\end{equation} Note that since every path in $\Pi_{1,2}$ and $\Pi_{2,1}$ connects two sites at opposite boundaries of the strip, each path in $\Pi_{1,2}$ and $\Pi_{2,1}$ must have a length of at least $N$. Thus, by combining \eqref{eq:11PassageTimes}, \eqref{eq:22PassageTimes} and \eqref{eq:1221PassageTimes} for a uniform upper bound on the passage time of the heaviest path $\pi$ connecting two sites in $\partial_2(\mathcal S_N)$ and touching $\partial_1(\mathcal S_N)$, together with  \eqref{eq:PathLengths1221} for a uniform lower bound on the last passage time between the respective sites in  $\partial_2(\mathcal S_N)$, we conclude.
\end{proof}

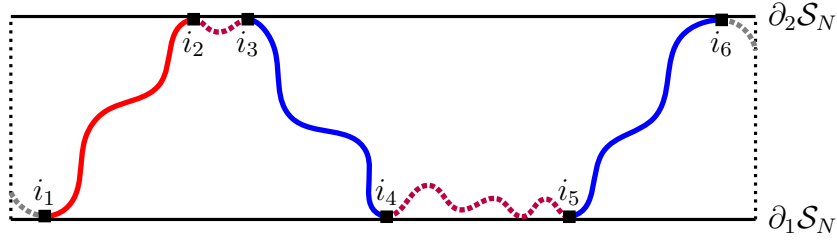
\begin{figure}
    \centering
\begin{tikzpicture}[scale=.9]

%   \draw[line width=1.2pt] (0,0) -- (10,0) -- (10,4) -- (0,4) -- (0,0);
%   
%\draw[nicos-red, line width =1.7 pt] (0,2) to[curve through={(0.4,2.2) .. (0.7,2.7) .. (1.6,3.4) .. (1.9,3.35) .. (2.3,3.7)..(2.65,3.95) }] (2.7,3.96);
%
%\draw[darkblue, line width =1.7 pt] (2.7,3.96) to[curve through={(2.75,3.95) .. (2.95,3.8).. (3.2,3.8) .. (3.45,3.95) }] (3.5,3.96);
%
%
%\draw[nicos-green, line width =1.7 pt] (3.5,3.96) to[curve through={(3.55,3.95).. (3.75,3.8) .. (3.9,3.5) .. (4.15,2.5) .. (4.4,2.4).. (4.8,1.2) .. (5.25,1)..  (5.5,0.05) }] (5.55,0.04);
%
%\draw[darkblue, line width =1.7 pt] (5.55,0.04) to[curve through={(5.6,0.05) .. (5.8,0.2).. (6.2,0.5) .. (6.6,0.15) ..  (7.2,0.3).. (7.5,0.05) .. (7.55,0.04) .. (7.6,0.05) ..(7.9,0.3).. (8.2,0.05) }] (8.25,0.04);
%
%\draw[nicos-red, line width =1.7 pt] (8.25,0.04) to[curve through={(8.3,0.05) .. (8.5,0.2).. (8.9,0.5) ..  (9.2,0.5)..(9.8,0.95)}] (10,1);

   \draw[line width=1.2pt] (0,0) -- (11,0);
   \draw[line width=1.2pt] (0,3) -- (11,3); 
   \draw[line width=1.2pt,dotted] (11,0) -- (11,3); 
   \draw[line width=1.2pt,dotted] (0,0) -- (0,3); 
   
	\node (x1) at (11.7,0){$\partial_1 \mathcal{S}_N$} ;      
 	\node (x2) at (11.7,3){$\partial_2 \mathcal{S}_N$} ;        

\draw[densely dotted, line width =2 pt,black!50] (0,0.4)   to[curve through={(0.3,0.1)}] (0.5,0.05);

%\filldraw [fill=black] (0.5-0.1,0-0.1) rectangle (0.5+0.1,0+0.1);  

\draw[red, line width =2 pt] (0.5,0.05) to[curve through={(0.8,0.15) .. (1.1,1.2) .. (1.4,1.6) .. (2.2,2) ..(2.65,2.95) }] (2.7,2.96);

\draw[purple, line width =2 pt,densely dotted] (2.7,2.96) to[curve through={(2.75,2.95) .. (2.95,2.8).. (3.2,2.8) .. (3.45,2.95) }] (3.5,2.96);

\draw[blue, line width =2 pt] (3.5,2.96) to[curve through={(3.55,2.95).. (3.75,2.8) .. (3.9,2.5) .. (4.15,1.5) ..  .. (5.25,1)..  (5.5,0.05) }] (5.55,0.04);

\draw[purple, line width =2 pt, densely dotted] (5.55,0.04) to[curve through={(5.6,0.05) .. (5.8,0.2).. (6.2,0.5) .. (6.6,0.15) ..  (7.2,0.3).. (7.5,0.05) .. (7.55,0.04) .. (7.6,0.05) ..(7.9,0.3).. (8.2,0.05) }] (8.25,0.04);

\draw[blue, line width =2 pt] (8.25,0.04) to[curve through={(8.3,0.05)..(8.75,1)..
(9.5,1.5) .. (9.9,2.7).. (10.2,2.9)}] (10.5,2.95);

\draw[densely dotted, line width =2 pt,black!50] (10.5,2.95) to[curve through={(10.7,2.9)}] (11,2.5);

\filldraw [fill=black] (0.5-0.085,0.05-0.085) rectangle (0.5+0.085,0.05+0.085);     
\filldraw [fill=black] (2.7-0.085,2.96-0.085) rectangle (2.7+0.085,2.96+0.085);       
 
 \filldraw [fill=black] (3.5-0.085,2.96-0.085) rectangle (3.5+0.085,2.96+0.085);   
 \filldraw [fill=black] (5.55-0.085,0.04-0.085) rectangle (5.55+0.085,0.04+0.085);   
 \filldraw [fill=black] (8.25-0.085,0.04-0.085) rectangle (8.25+0.085,0.04+0.085);   
 \filldraw [fill=black] (10.5-0.085,2.95-0.085) rectangle (10.5+0.085,2.95+0.085);

	\node (x1) at (0.5,0.4){$i_1$} ;   
	\node (x2) at (2.7,2.6){$i_2$} ;   
	\node (x3) at (3.5,2.6){$i_3$} ;   
	\node (x4) at (5.55,0.4){$i_4$} ;   
	\node (x5) at (8.25,0.4){$i_5$} ;      
	\node (x6) at (10.5,2.6){$i_6$} ;     
   
	\end{tikzpicture}	
    \caption{Visualization of the path decomposition in Lemma \ref{lem:NoCrossing}.}
    \label{fig:pathDecomposition}
\end{figure}

\subsection{Estimates on last passage times and geodesics in the slab}\label{sec:EstimatesOnLPP}

Next, we estimate the last passage time between a line in the slab to a site at the boundary $\partial_2(\mathcal{S}_N)$. For all $x\in \mathbb{Z}$, we define the line segment
\begin{equation}\label{eq:line}
     \mathbb{L}_{x} := \left\{ (\lfloor y \rfloor+x, \lfloor - by \rfloor+x) \colon y\in  \left\{0,\dots,\frac{N}{b+1}\right\}  \right\} ,
\end{equation} where we recall that  $b=(1-\beta)\beta^{-1}>1$ from \eqref{def:ABs}. Note that $\mathbb{L}_x \subseteq \mathcal{S}_N$ for all $x\in \Z$ and that $\mathbb{L}_x$ contains the sites $p_x$ and $q_{x+((1+b)^{-1}-1/2)N}$. The following proposition states that the last passage time from some point on the line segment $\mathbb{L}_0$ to a sufficiently far away site on the lower diagonal $\partial_2(\mathcal{S}_N)$ is concentrated. %We will see that, intuitively, the growth interface of the TASEP in the high density phase in equilibrium is parallel to the line $\mathbb{L}_{0}$.
\begin{prop}\label{pro:LineToBoundary} Fix some  $\varepsilon,\tilde{C}>0$. Then there exist constants $(c_i)_{i \in [3]}$, depending only on $\varepsilon,\tilde{C}$ and $\beta$ such that 
\begin{equation}
 \P\left( \Big|   T(v,q_{x+N/2})) -  \frac{b+1}{b-1}N - \Big(x-\frac{N}{b^2-1}\Big)\frac{(b+1)^2}{b} \Big| \leq c N^{\frac{4}{5}} \, \forall v\in \mathbb{L}_{0} \right) \geq 1- c_2e^{-c_3 \log^{2}(N)}
\end{equation} for all $x \in [ (1+\varepsilon)(b^2-1)^{-1}N, \tilde{C}N ]$, and all $N$ sufficiently large. 
\end{prop}
Let us remark that the choice of the exponent $4/5$ in Proposition \ref{pro:LineToBoundary} is not optimal as any value strictly larger than $3/4$ would be covered by our arguments, and that the role of $\varepsilon>0$ and $\tilde{C}>0$ will become clear in the sequel. 
In order to show Proposition \ref{pro:LineToBoundary}, we require some setup. Recall that we denoted by $T_{\alpha,\beta}$ the last passage times on the slab with weight $\alpha >0 $ on the upper diagonal $\partial_1 (\mathcal S_N)$, weight $\beta$ on the lower diagonal $\partial_2 (\mathcal S_N)$, and weight $1$ in bulk. 
We start with a result on last passage times $T_{1,\beta}(v,w)$ with geodesics $\gamma_{1,\beta}(v,w)$ when we set $\alpha=1$. For $y\in [N]$ and $x$ of order $N$, our task is to obtain a bound on 
\begin{equation}\label{eq:zastPoint}
z^{\ast}=z_{\alpha,\beta}^{\ast}(x,y) := \min\left\{ z \in \N \cup \{0\}  \colon q_{z+N/2} \in \gamma_{1,\beta}((N-y,0),q_{x+N/2}) \right\} ,
\end{equation} which is the first site in $\partial_2(\mathcal{S}_N)$ that the geodesic from $(N-y,0)$ to $q_{x+N/2}$ intersects. 
\begin{lem}\label{lem:CoalescencePointToPoint} Let $\varepsilon,\tilde{C}>0$. Then for all $x \in [ (1+\varepsilon)(b^2-1)^{-1}N, \tilde{C}N ]$, for all $y \in [N]$, and all $N$ sufficiently large,
\begin{equation}\label{eq:LocateCoalesence}
\P\left(z_{1,\beta}^{\ast}(x,y) \in \Big[ \frac{y}{b^2-1} - c_1 N^{4/5}, \frac{y}{b^2-1} + c_1 N^{4/5} \Big] \right) \geq 1- c_2e^{-c_3 \log^{2}(N)} .
\end{equation}
for constants $(c_i)_{i \in [3]}$. Moreover, we see that for all $N$ sufficiently large
\begin{equation}\label{eq:LocatePassageTime}
\P \Big( \Big| T_{1,\beta}((N-y,0),q_{x+N/2}) - \frac{b+1}{b-1}y - \Big( x - \frac{y}{b^2-1} \Big) \frac{(b+1)^2}{b} \Big| \leq c_4 N^{4/5} \Big) \geq 1- c_5e^{-c_6 \log^{2}(N)}
\end{equation} for some constant $c_4,c_5,c_6>0$.
\end{lem}
\begin{proof} Recall that the key task is to give a bound on the location $q_{z^{\ast}}$ where the geodesic from $(N-y,0)$ to $q_{x+N/2}$ intersects $\partial_2(\mathcal{S}_N)$ for the first time. To do so, we start with the bound on the last passage time in  \eqref{eq:LocatePassageTime}. We claim that for all $N$ large enough,  and $z \in [-N^{4/5},N^{4/5}]$, 
\begin{equation}\label{eq:Lower1}
    \P\left( \Big| T_{1,1}\big((N-y,0),q_{\big(\frac{y}{b^2-1}+N/2+z\big)}\big) - \frac{(b+1)y}{b-1} \Big| \leq c_1( \sqrt{N}\log^{2}N+|z|) \right) \geq 1- c_2e^{-c_3 \log^{2}(N)}
\end{equation} for some constants $(c_i)_{i \in [3]}$. Note that the upper bound is immediate from Theorem \ref{thm:Ledoux}. The lower bound follows from Theorem 4.2(iii) in~\cite{BGZ:TemporalCorrelation}. More precisely, this states that the last passage time in \eqref{eq:Lower1} satisfies the same moderate deviation lower bound as in Theorem~\ref{thm:Ledoux} when restricting the geodesics to stay in the strip $\mathcal{S}_N$ (with $\alpha=\beta=1$), provided that the endpoints $(u_1,u_2),(v_1,v_2) \in \Z^2$ of the geodesics satisfy
\begin{equation}
\frac{v_2-v_1}{u_2-u_1} \in [\psi_1,\psi_2] 
\end{equation} for some positive constants $\psi_1,\psi_2>0$, which do not depend on $N$, and $t$ from Theorem \ref{thm:Ledoux}.
Next, note that by Lemma \ref{lem:CombinationLemma} with $n=N-1$ (and rotating the strip), we obtain that
\begin{align}\label{eq:Lower2}
\begin{split}
\P\left(\Big|  T_{1,\beta}(q_{\big(\frac{y}{b^2-1}+N/2+z\big)},q_{x+N/2}) -    \Big( x - \frac{y}{b^2-1} - z \Big) \frac{(b+1)^2}{b} \Big| \leq c_4  \sqrt{N}\log^{4}N \right) \\ \geq 1- c_5e^{-c_6 \log^{2}(N)} \end{split} \quad 
\end{align} for some constants $c_4,c_5,c_6>0$. Thus, in order to show the lemma, we have to argue that the above last passage times are much smaller when $z \notin [-N^{4/5},N^{4/5}]$, compared to $z=0$. 
%Again, using Lemma \ref{lem:CombinationLemma}, we have for all $N$ large enough that
%\begin{equation}\label{eq:DiagonalBound}
%\P\left(\Big|  T_{1,\beta}(q_i,q_j) - ( i-j) \frac{(b+1)^2}{b} \Big|  \leq c\sqrt{N}\log^{4}N   \text{ for all } i,j\in [N^2] \right) \geq 1- c_5e^{-c_6 \log^{2}(N)}  \, .
%\end{equation}
 Using the fact that $\sqrt{1+\varepsilon}=1+\frac{\varepsilon}{2}-\frac{\varepsilon^2}{8}+\mathcal{O}(\varepsilon^3)$ for all $\varepsilon>0$, Theorem~\ref{thm:Ledoux} and an elementary computation show that for all $N$ large enough
\begin{equation*}
    \E\left[ T_{1,1}\big((N-y,0),q_{\big(\frac{y}{b^2-1}+N/2+z\big)}\big)\right] - \frac{(b+1)y}{b-1} - \frac{(b+1)^2z}{b}  \leq  - \frac{(b^2-1)z^{2}}{2b N}\left( \frac{(b+1)^2}{4}-1 \right) \, .
\end{equation*}  
Hence, Theorem~\ref{thm:Ledoux} guarantees that for all $z$ with $|z| \geq N^{4/5}$, we have that
\begin{align}\label{eq:Lower4}
\begin{split}
\P\left( T_{1,1}\big((N-y,0),q_{\big(\frac{y}{b^2-1}+N/2+z\big)}\big) \leq \frac{(b+1)y}{b-1} - \frac{(b+1)^2z}{b} - c_7 N^{3/5} \right)  \\ \geq 1 - c_8e^{-c_9 \log^{2}(N)} 
\end{split}
\end{align} for some constants $c_7,c_8,c_9>0$.
Combining now \eqref{eq:Lower2} and \eqref{eq:Lower4} for an upper bound on the last passage time, together with a lower bound in \eqref{eq:Lower1} on the last passage time (with $z=0$), we see that for any choice of $z\in [-(1+b^{2})^{-1}N,\tilde{C}N] \setminus [-N^{4/5},N^{4/5}]$, the event
\begin{equation}\begin{split}
\left\{  T_{1,1}\big((N-y,0),q_{\big(\frac{y}{b^2-1}+N/2+z\big)} + T_{1,\beta}(q_{\big(\frac{y}{b^2-1}+N/2+z\big)},q_{x+N/2}) < T_{1,\beta}\big((N-y,0),q_{x+N/2}\big)   \right\} \end{split} 
\end{equation} holds with probability at least $1- C e^{-c \log^{2}(N)}$ for some $c,C>0$. Using now a union bound over $z\in [-(1+b^{2})^{-1}N,\tilde{C}N] \setminus [-N^{4/5},N^{4/5}]$, this finishes the proof.
%\vspace{3cm}
%
%
%
%
%
%From Theorem \ref{thm:Ledoux} and a computation
%\begin{equation}
%\max_{ j\in \Z \colon |j|\geq N^{4/5} }\E[T^{1,1}_{(N-y,0),\big(\frac{y}{b^2-1}+N+j,\frac{y}{b^2-1}+j\big)}] -  \frac{b+1}{b-1}N - \frac{(b+1)^2}{b}j \leq - c_4N^{11/20} 
%\end{equation} for some constant $c_4>0$ and all $N$ sufficiently large. Thus, we see that \begin{equation}
%    \P\left( T^{1,\beta}_{(N-y,0),(x,x+N)} - \frac{b+1}{b-1}y - \Big( x - \frac{y}{b^2-1} \Big) \frac{(b+1)^2}{b}  \leq -c_5N^{4/5} \,   \Big| \,  z_{x^{\ast}} \notin A_{x,y,N} \right) \geq 1-N^{-3}
%\end{equation} for some constant $c_5>0$, where we define
%\begin{equation}
%    A_{x,y,N} := \Big[ \frac{1}{b^2-1}y - N^{4/5}, \frac{1}{b^2-1}y + N^{4/5} \Big] \, .
%\end{equation}
%Using \eqref{eq:LocateCoalesence}, the upper bound in \eqref{eq:LocatePassageTime} follows from Theorem \ref{thm:Ledoux} as
%\begin{equation}
%    \P\left( T^{1,1}_{(N-y,0),\big(\frac{y}{b^2-1}+N+N^{4/5},\frac{y}{b^2-1}+N^{4/5}\big)} - \frac{(b+1)y}{b-1}  \leq c_6N^{4/5} \right) \geq 1- N^{-4}
%\end{equation} for some constant $c_6>0$ together with Proposition \ref{cor:1} to see that
%\begin{equation}
%\P\left( T^{1,\beta}_{\big(\frac{y}{b^2-1}+N-N^{4/5},\frac{y}{b^2-1}-N^{4/5}\big),(x+N,x)} - \Big( x - \frac{y}{b^2-1} \Big) \frac{(b+1)^2}{b}  \leq  c_7 N^{4/5} \right) \geq 1- N^{-4}
%\end{equation} for some constant $c_7>0$ and all $N$ sufficiently large. 
\end{proof}

\begin{figure}
    \centering
\begin{tikzpicture}[scale=.5]

\draw[line width=1pt] (0,0) -- (5,0);

%\draw[densely dotted, line width=1pt] (0+6.5,0+6.5) -- (2.5+9,-2.5+9);

\draw[line width=0.6pt,>=triangle 45, <->] (5.2,0) -- (8.3,0);

\draw[line width=0.6pt,>=triangle 45, <->] (8.5,0.2) -- (8.5,3.3);

\node[scale=0.8] (x1) at (9.15,3.45) {$z_{\eta}$};

\node[scale=0.8] (x1) at (-0.2,0.4) {$p_0$};
\node[scale=0.8] (x1) at (4.5,0.4) {$q_{N/2}$};

\node[scale=0.8] (x1) at (11,2) {$(1+b^{2})^{-1}N$};

\draw[densely dotted, line width=1pt] (0,0) -- (8.5,3.5);

%\draw[densely dotted, line width=1pt] (1.525,1.525) -- (1.525+2.5,1.525-2.5);
%\draw[densely dotted, line width=1pt] (9-1.525,9-1.525) -- (9-1.525+2.5,9-1.525-2.5);

   \draw[line width=1.2pt] (0,0) -- (0+6.5,0+6.5);
   \draw[line width=1.2pt] (5,0) -- (11.5,6.5);

\draw[red, line width =1 pt] (0,0) to[curve through={(0.7,0.2) ..(1.5,0.3)..(2.45,1.3) .. (5,1.6).. (8.5-0.3,3.3)..(8.5,3.5)..(8.5+0.3,3.5+0.35).. (2.5+9-0.4,-2.5+9-0.1)}] (2.5+9,-2.5+9);
%\draw[red, line width =2 pt] (9,9) to[curve through={(9+0.2,9-0.7) ..(9-0.1,9-1.5)..(9-0.2,9-2.45) }] (9-0.4,9-2.65);

%\draw[blue, line width =2 pt] (2.5,-2.5) to[curve through={(2.6,-1.8) ..(3,-1.5)..(5.5,2.5)..(7,3)..(8,6)}] (9,9);

%\node[scale=1] (x1) at (8.3-0.15,4.2-0.15) {$\rho$};
%	
%		\node[scale=1] (x1) at (-0.7,0){$0$} ;
%		\node[scale=1] (x2) at (2.5+0.7,-2.5){$q_0$} ;
%		\node[scale=1] (x3) at (9+0.7,9){$p_m$} ;
%		\node[scale=1] (x4) at (9+2.5+0.7,9-2.5){$q_m$} ;
%

	%\filldraw [fill=black] (7.88-0.15,4.62-0.15) rectangle (7.88+0.15,4.62+0.15);  

%	\filldraw [fill=black] (2.5-0.15,-2.5-0.15) rectangle (2.5+0.15,-2.5+0.15);   

 	\filldraw [fill=black] (8.5-0.15,3.5-0.15) rectangle (8.5+0.15,3.5+0.15);     

	\filldraw [fill=black] (5-0.15,-0.15) rectangle (5.15,0.15);   
 	
	\filldraw [fill=black] (-0.15,-0.15) rectangle (0.15,0.15);   

 %	\filldraw [fill=black] (9-0.15,9-0.15) rectangle (9+0.15,9+0.15);     	
 	
	\end{tikzpicture}	
    \caption{\label{fig:HittingBoundary}Hitting of the boundary $\partial_2(\mathcal{S}_N)$ in Lemma \ref{lem:CoalescenceAlpha}.}
\end{figure}
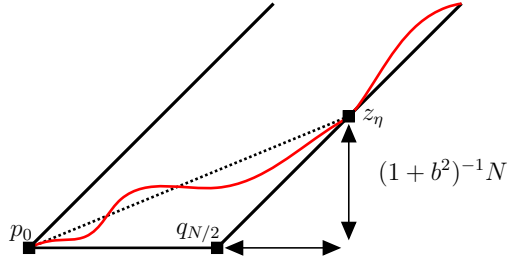

In the next statement, we use Lemma \ref{lem:NoCrossing}  together with Lemma \ref{lem:CoalescencePointToPoint} for the last passage time $T_{1,\beta}(v,w)$ between two sites $v \preceq w$, in order to provide a bound on the last passage time $T_{\alpha,\beta}(v,w)$.
\begin{lem} \label{lem:CoalescenceAlpha}
Let $\varepsilon,\tilde{C}>0$, and let $x \in [ (1+\varepsilon)(b^2-1)^{-1}N, \tilde{C}N ]$. Further, we let $y\in [N]$ and recall $z_{\alpha,\beta}^{\ast}(x,y)$ from \eqref{eq:zastPoint}.
%\begin{equation}
%z_{\alpha,\beta}^{\ast}(x,y) := \inf\left\{ z \in \N  \colon (z+N,z) \in %\gamma_{\alpha,\beta}((N+1-y,0),(x+N,x)) \right\} \, .
%\end{equation} 
Then for all $N$ sufficiently large
\begin{equation}\label{eq:LocateCoalesenceAlpha}
\P\left(z_{\alpha,\beta}^{\ast}(x,y) \in \Big[ \frac{y}{b^2-1} - c_1 N^{4/5}, \frac{y}{b^2-1} + c_1 N^{4/5} \Big] \right) \geq 1- c_2e^{-c_3 \log^{2}N} 
\end{equation} for constants $(c_i)_{i \in [3]}$. Moreover, we see that for all $N$ sufficiently large
\begin{equation}\label{eq:LocatePassageTimeAlpha}
\P \Big( \Big| T_{\alpha,\beta}((N-y,0),q_{x+N/2}) - \frac{b+1}{b-1}y - \Big( x - \frac{y}{b^2-1} \Big) \frac{(b+1)^2}{b} \Big| \leq c_4N^{4/5} \Big) \geq 1- c_5e^{-c_6 \log^{2}N}
\end{equation} for constants $c_4,c_5,c_6>0$.
\end{lem} 
\begin{proof} Note that the lower bound on the last passage time in \eqref{eq:LocatePassageTimeAlpha} is immediate from Lemma~\ref{lem:CoalescencePointToPoint}. For the geodesic $\gamma$ from $(N-y,0)$ to $q_{x+N/2}$ touching $\partial_1(\mathcal S_N)$, let $p_{z_1}$ and $p_{z_2}$ for some $z_1 < z_2$ be the first and last site where $\gamma$ intersects $\partial_1(\mathcal S_N)$, respectively. 
We argue that with probability at least $1-C e^{-c \log^{2}N}$ either $z_2 \leq 2N^{4/5}$ holds or the geodesic $\gamma$ does not intersect $\partial_1(\mathcal S_N)$ at all. Together with Lemma \ref{lem:NoCrossing} and Lemma~\ref{lem:CoalescencePointToPoint}, this implies the desired bounds in \eqref{eq:LocateCoalesenceAlpha} and \eqref{eq:LocatePassageTimeAlpha}; see also Figure \ref{fig:HittingBoundary}. In order to estimate $z_2$, recall the path decomposition from Lemma \ref{lem:NoCrossing}, and assume that all passage times of paths in the sets $\Pi_{i,j}$ for $i,j\in \{1,2\}$ satisfy the relations in \eqref{eq:11PassageTimes}, \eqref{eq:22PassageTimes} and \eqref{eq:1221PassageTimes}. We claim that $|z_2-z_1| > N^{4/5}$ holds with probability at most $C e^{-c \log^{2}N}$ for some $c,C>0$. To see this, fix some choice of $z_1$ and $z_2$. We compare the last passage time $T_1$ from $p_{z_1}$ to $q_{x+N/2}$ via $p_{z_2}$ with the last passage time $T_2$ from $p_{z_1}$ to $q_{x+N/2}$ via $q_{x+N/2-z_1+z_2}$. Note that by the estimates in \eqref{eq:11PassageTimes}, \eqref{eq:22PassageTimes} and \eqref{eq:1221PassageTimes}, we have  with probability at least $1-C e^{-c \log^{2}N}$ that $T_1<T_2$ uniformly in $z_1$ and $z_2$.
We conclude by a union bound over the at most $N^2$ many choices for $z_1$ and $z_2$. 
\end{proof}
%Combining the above statements, we can now obtain Proposition \ref{pro:LineToBoundary}.
\begin{proof}[Proof of Proposition \ref{pro:LineToBoundary}]  Fix some site $v=\big(\lfloor(N-y)\frac{1}{b+1}\rfloor,-\lfloor(N-y)\frac{b}{b+1}\rfloor \big)\in \mathbb{L}_{0}$ with some $y\in \{0,\dots,N\}$. From Lemma \ref{lem:CoalescenceAlpha}, we get that the first intersection point $q_{z_y^{\ast}+N/2}$ of $\gamma(v,q_{x+N/2})$ with the lower diagonal $\partial_2(\mathcal{S}_N)$ satisfies for some positive constants $(c_i)_{i \in [3]}$
\begin{equation}
  \P\left(   z_y^{\ast} \in \left[ - \frac{N}{b+1}+ \frac{yb}{b^2-1} - c_1N^{4/5}, - \frac{N}{b+1}+ \frac{yb}{b^2-1} + c_1N^{4/5} \right] \right) \geq 1- c_2e^{-c_3 \log^{2}N} . 
\end{equation}
By Proposition \ref{prop:123} and Lemma \ref{lem:NoCrossing}, we get that with probability at least $1- c_5e^{-c_6 \log^{2}N}$
\begin{equation}
T_{\alpha,\beta}(v,q_{x+N/2}) \in \left[ \frac{b+1}{b-1}N - \Big(x-\frac{N}{b^2-1}\Big)  - c_4N^{4/5},  \frac{b+1}{b-1}N - \Big(x-\frac{N}{b^2-1}\Big) + c_4N^{4/5}\right]
\end{equation}  for all $v \in \mathbb{L}_0$, and constants $c_4,c_5,c_6>0$. We conclude by a union bound over $y$.
\end{proof}

\begin{figure}
    \centering
\begin{tikzpicture}[scale=.5]

\draw[line width=1pt] (1,1) -- (5,0);
\draw[line width=1pt] (0,0) -- (5-1,-1);

%\draw[densely dotted, line width=1pt] (0+6.5,0+6.5) -- (2.5+9,-2.5+9);

%\draw[line width=0.6pt,>=triangle 45, <->] (5.2,0) -- (8.3,0);

%\draw[line width=0.6pt,>=triangle 45, <->] (8.5,0.2) -- (8.5,3.3);

\draw[densely dotted, line width=1pt] (0,0) -- (8.5,3.5);

\draw[densely dotted, line width=1pt] (0-1.5+8.5*0.15,0-1.5+3.5*0.15) -- (8.5-1.5,3.5-1.5);

\draw[densely dotted, line width=1pt] (0-3+8.5*0.5,0-3+3.5*0.5) -- (8.5-3,3.5-3);

%\draw[densely dotted, line width=1pt] (1.525,1.525) -- (1.525+2.5,1.525-2.5);
%\draw[densely dotted, line width=1pt] (9-1.525,9-1.525) -- (9-1.525+2.5,9-1.525-2.5);

   \draw[line width=1.2pt] (-0.8,-0.8) -- (0+6.5,0+6.5);
   \draw[line width=1.2pt] (4-0.8,-1-0.8) -- (11.5,6.5);

\node[scale=0.8] (x1) at (9.25-2,3.4-2) {$z_{\eta}$};
\node[scale=0.8] (x1) at (5+1.3-0.2,-0.2) {$\mathbb{L}_{cN^{4/5}}$};
\node[scale=0.8] (x1) at (5-1+0.9-0.2,-1-0.2) {$\mathbb{L}_{-1}$};

\draw[blue, line width =1 pt] (1.5,1.5) --++ (0.5,0)--++ (0,-0.7)--++(0.4,0)--++(0,-1)--++(0.8,0)--++(0,-0.2)--++(0.9,0)--++(0,-0.2)--++(0.3,0);

\draw[red, line width =1 pt] (2.4,-0.2) to[curve through={(2.7,0.2) ..(3.5,0.3)..(3.55,0.4) .. (6.6-0.45,1.6-0.3).. (6.6,1.6).. (8.5-0.3,3.3)..(8.5,3.5)..(8.5+0.3,3.5+0.35).. (2.5+9-0.4,-2.5+9-0.1)}] (2.5+9,-2.5+9);
%\draw[red, line width =2 pt] (9,9) to[curve through={(9+0.2,9-0.7) ..(9-0.1,9-1.5)..(9-0.2,9-2.45) }] (9-0.4,9-2.65);

%\draw[blue, line width =2 pt] (2.5,-2.5) to[curve through={(2.6,-1.8) ..(3,-1.5)..(5.5,2.5)..(7,3)..(8,6)}] (9,9);

%\node[scale=1] (x1) at (8.3-0.15,4.2-0.15) {$\rho$};
%	
%		\node[scale=1] (x1) at (-0.7,0){$0$} ;
%		\node[scale=1] (x2) at (2.5+0.7,-2.5){$q_0$} ;
%		\node[scale=1] (x3) at (9+0.7,9){$p_m$} ;
%		\node[scale=1] (x4) at (9+2.5+0.7,9-2.5){$q_m$} ;
%

	%\filldraw [fill=black] (7.88-0.15,4.62-0.15) rectangle (7.88+0.15,4.62+0.15);  

	\filldraw [fill=black] (8.5-2-0.15,3.5-2-0.15) rectangle (8.5-2+0.15,3.5-2+0.15);   

 %	\filldraw [fill=black] (8.5-0.15,3.5-0.15) rectangle (8.5+0.15,3.5+0.15);     

%	\filldraw [fill=black] (5-0.15,-0.15) rectangle (5.15,0.15);   
 	
%	\filldraw [fill=black] (-0.15,-0.15) rectangle (0.15,0.15);   

 %	\filldraw [fill=black] (9-0.15,9-0.15) rectangle (9+0.15,9+0.15);     	
 	
	\end{tikzpicture}	
    \caption{Hitting of the boundary $\partial_2 (\mathcal{S}_N)$ at $z_{\eta}$ when starting from the initial growth interface $\Gamma_{\eta}$. The dashed lines indicate the optimal slope of geodesics to a far away point when starting from some site in $\Gamma_{\eta}$.}
    \label{fig:HittingBoundaryInterface}
\end{figure}
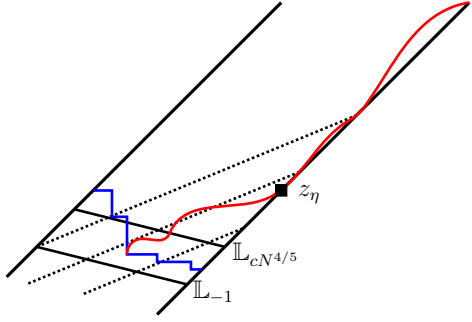

\subsection{Simultaneous coalescence of geodesics}\label{sec:Simultaneous}

Next, we argue that we can place each initial growth interface such that with high probability,  the last passage time to a far away site in the slab is comparable to the last passage time when starting instead from the line segment $\mathbb{L}_0$ defined in \eqref{eq:line}. More precisely, for initial configurations $\eta,\zeta \in \{0,1\}^{N}$, we choose the corresponding initial growth interfaces $\Gamma_{\eta}$ and $\Gamma_{\zeta}$ in $\mathcal S_N$ such that
\begin{align}\begin{split}
\Gamma_{\eta} \cap  \mathbb{L}_0 \neq \emptyset \quad &\text{ and} \quad  \Gamma_{\eta} \cap  \mathbb{L}_{-1} = \emptyset\\
\Gamma_{\zeta} \cap  \mathbb{L}_{0} \neq \emptyset \quad &\text{ and} \quad  \Gamma_{\zeta} \cap  \mathbb{L}_{-1} = \emptyset \, .\end{split}\label{eq:ChoiceOfGammas}
\end{align} 
In the following lemma, we control last passage times and first intersection points with the diagonal $\partial_2 (\mathcal S_N)$, starting from $\Gamma_{\eta}$ and $\Gamma_{\zeta}$, respectively; see also Figure \ref{fig:HittingBoundaryInterface}. 
\begin{lem}\label{lem:InterfaceHittingDiagonal} Let $\varepsilon,\tilde{C}>0$, and let $x \in [ (1+\varepsilon)(b^2-1)^{-1}N, \tilde{C}N ]$. 
Let $z_{\eta}$ and $z_{\zeta}$ denote the first intersection points of the boundary $\partial_{2}(\mathcal{S}_N)$ with the geodesics from $\Gamma_{\eta}$ and $\Gamma_{\zeta}$ to $q_{x+N/2}$, respectively. Then there exist postive $(c_i)_{i \in [3]}$ such that for all $N$ sufficiently large
\begin{equation}\label{eq:EventAcal}
\mathcal{A}:= \left\{ \max(z_{\eta}, z_{\zeta}) \leq  \frac{1}{b^2-1}N + c_1 N^{4/5}  \right\}
\end{equation} holds with probability at least $1-c_2e^{-c_3\log^{2}N}$. Moreover, we find constants $c_4,c_5,c_6>0$ such that with probability at least $1-c_5e^{-c_6\log^{2}N}$ for all $N$ large enough, 
\begin{align}
\left| \max_{v\in \Gamma_{\eta}}\big(T_{\alpha,\beta}(v,(x+N,x))\big) - \frac{b+1}{b-1}N - \Big(x-\frac{N}{b^2-1}\Big) \right| &\leq c_4N^{4/5} \\
\left| \max_{v\in \Gamma_{\zeta}}\big(T_{\alpha,\beta}(v,(x+N,x))\big) - \frac{b+1}{b-1}N - \Big(x-\frac{N}{b^2-1}\Big) \right| &\leq c_4N^{4/5} . 
\end{align}
\end{lem}
\begin{proof} Note that by Lemma \ref{lem:CoalescenceAlpha}, for all $N$ sufficiently large, the geodesics from $\Gamma_\eta$ and $\Gamma_\zeta$ to $q_{x+N/2}$ start with probability at least $1-c_2e^{-c_3\log^{2}N}$ from some pair of sites $v_{\eta}$ and $v_{\zeta}$ between the line segments $\mathbb{L}_{-1}$ and $\mathbb{L}_{c_1N^{4/5}}$, with positive  $(c_i)_{i \in [3]}$. Thus, we obtain the bound on the last passage times from $\Gamma_{\eta}$ and $\Gamma_{\zeta}$ to $q_{x+N/2}$ by Proposition \ref{pro:LineToBoundary}. Furthermore, the bound on $z_{\eta}$ and $z_{\zeta}$ in the event $\mathcal{A}$ follows from \eqref{eq:LocateCoalesenceAlpha} in Lemma \ref{lem:CoalescenceAlpha} together with a union bound over the at most $c_1N^{9/5}$ choices for  $v_{\eta}$ and $v_{\zeta}$ between the lines $\mathbb{L}_{-1}$ and $\mathbb{L}_{c_1N^{4/5}}$ as starting points of the geodesic. 
\end{proof}

From Lemma \ref{lem:InterfaceHittingDiagonal} we deduce the following bound on the coalescence of all geodesics starting from $\Gamma_\eta$ or $\Gamma_{\zeta}$, and connecting to some point on the line $\mathbb{L}_{x}$ for some suitably large $x$. %Let us remark that it is at this point that we require the constant $\tilde{C}$ from the previous statements to be chosen sufficiently large.
\begin{lem}\label{lem:CoalescenceAllGeodesic}
Fix $\varepsilon>0$ and let 
\begin{equation}\label{eq:xast}
    x^{\ast}:= \frac{\hat{a}(b+1)(b-1)}{(b-\hat{a})(\hat{a}b-1)} ,
\end{equation} where we recall $\hat{a}=\max(a,1)$. Then for all $N$ sufficiently large, there exist some $z_1<\frac{1}{2}(1+\varepsilon )x^{\ast}N<z_2$ such that with probability at least $1-Ce^{-c\log^2N}$
\begin{equation}
    q_{z_1},q_{z_2} \in \gamma(p_0,p_{(1+\varepsilon)x^{\ast}N}) 
\end{equation} for $c,C>0$, i.e.\  
the geodesic from $p_0$ to $p_{(1+\varepsilon)x^{\ast}N}$ intersects $\partial_2(\mathcal{S}_N)$ before and after $q_{x^{\ast}(1+\varepsilon)N/2}$.
\end{lem}
\begin{proof}We claim that it suffices to show that for every $\varepsilon>0$, we find suitable constants $c,C,\delta>0$, depending only on $\varepsilon,\alpha,\beta$, such that with probability at least $1-c e^{-C\log^2 N}$ for all $N$ sufficiently large
\begin{equation}\label{eq:CoalesceAll}
    T_{\alpha,1}(p_0,p_{x^{\ast}N(1+\varepsilon)/2}) < T_{\alpha,\beta}(p_0,q_{x^{\ast}N(1+\varepsilon)/2}) - \delta N . 
\end{equation} To see this, note that by using \eqref{eq:CoalesceAll} twice, together with a symmetry argument, we get that the geodesic from $p_0$ to $p_{(1+\varepsilon)x^{\ast}N}$ must intersect $\partial_2(\mathcal{S}_N)$ with probability at least $1-C e^{-c\log^2 N}$ for some constants $c,C>0$. Assume without loss of generality that this is the case at some site $q_{z}$ with $z \geq x^{\ast}N(1+\varepsilon)/2$. Then take $z_2=z$, and note that the existence of a site $q_{z_1}\in \gamma(p_0,p_{(1+\varepsilon)x^{\ast}N})$ with $z_1 \leq x^{\ast}N(1+\varepsilon)/2$ with probability at least $1-c e^{-C\log^2 N}$ follows from Lemma \ref{lem:CoalescenceAlpha}, choosing the constant $\tilde{C}>0$ in Lemma \ref{lem:CoalescenceAlpha} sufficiently large. 
Thus, it remains to show \eqref{eq:CoalesceAll}. Using Lemma \ref{lem:CombinationLemma}, we see that for all $N$ sufficiently large
\begin{equation}\label{eq:xastTastRelation1}
  \P\left( T_{\alpha,1}(p_0,p_{x^{\ast}N(1+\varepsilon)/2}) \leq  \frac{(\hat{a}+1)^2}{2\hat{a}} x^{\ast}N(1+\varepsilon) +  c_1\sqrt{N}\log^4N \right) \geq 1- c_2 e^{-c_3 \log^2 N}
\end{equation}
for positive $(c_i)_{i \in [3]}$. Note that by Lemma \ref{lem:NoCrossing}, for all $N$ sufficiently large
\begin{equation}\label{eq:xastTastRelation2}
T_{\alpha,\beta}(p_0,q_{x^{\ast}(1+\varepsilon)N/2}) \geq T_{1,1}(p_0,q_{\big( \frac{b^2}{b^2-1}-\frac{1}{2}\big)N})  + T_{1,\beta}(q_{\big( \frac{b^2}{b^2-1}-\frac{1}{2}\big)N},q_{x^{\ast}(1+\varepsilon)N/2})  - \log^2 N
\end{equation}  with probability at least $1- C e^{-c \log^2 N}$ for $c,C>0$, where the term $-\log^2 N$ is to account for the weight at the site $p_0$.
Equation \eqref{eq:Lower1}  for the first term and Lemma~\ref{lem:CombinationLemma} for the second term at the right-hand side of \eqref{eq:xastTastRelation2} guarantee that for all $N$ large enough
\begin{align}
\begin{split}
\P\left(  T_{1,1}(p_0,q_{\big( \frac{b^2}{b^2-1}-\frac{1}{2}\big)N}) \geq \frac{b-1}{b+1}N - c_4 \sqrt{N}\log^2N \right) &\geq 1- c_5e^{-c_6\log^2N} \\
\P\left(  T_{1,\beta}(q_{\big( \frac{b^2}{b^2-1}-\frac{1}{2}\big)N},q_{x^{\ast}(1+\varepsilon)N/2}) \geq \frac{(b+1)^2}{b} y^{\ast} N - c_4 \sqrt{N}\log^4N \right) &\geq 1- c_5e^{-c_6\log^2N}
\end{split} \quad \label{eq:xastTastRelation3}
\end{align} for constants $c_4,c_5,c_6>0$, where $y^{\ast}=\frac{x^{\ast}(1+\varepsilon)}{2} - \frac{b^2}{b^2-1} +\frac{1}{2}$. Observe that $x^{\ast}$ satisfies 
\begin{equation}\label{eq:xastTastRelation4}
\frac{x^{\ast}(\hat{a}+1)^2}{2\hat{a}} = \frac{b-1}{b+1} +\left( \frac{x^{\ast}}{2} - \frac{b^2}{b^2-1} +\frac{1}{2} \right) \frac{(b+1)^2}{b}   . 
\end{equation} 
Hence, recalling $b>\hat{a}$, we  combine \eqref{eq:xastTastRelation1}, \eqref{eq:xastTastRelation2}, \eqref{eq:xastTastRelation3} and \eqref{eq:xastTastRelation4} to obtain \eqref{eq:CoalesceAll}.
\end{proof}

While Lemma \ref{lem:CoalescenceAllGeodesic} and a shift argument guarantee that the geodesic from $p_{cN^{4/5}}$ to $p_{x^{\ast}N(1+\varepsilon)}$ intersects $\partial_2(\mathcal{S}_N)$ before and after $q_{x^{\ast}N(1+\varepsilon)/2}$ with high probability, we will also require that the intersection points are at least $\delta N$ apart from each other with $\delta=\delta(\varepsilon)>0$. This allow us to ensure that the geodesics remains close to the boundary $\partial_2(\mathcal{S}_N)$ for distance of at least $\delta N$, in which we can modify the last passage times in Section \ref{sec:RandomExtensionTimeChange}. This is the content of the following corollary. As it is immediate from Lemma \ref{lem:CoalescenceAlpha} and Lemma \ref{lem:CoalescenceAllGeodesic}, we omit the proof. 

\begin{cor} \label{cor:ImprovedHitting} For $\varepsilon,\varepsilon^{\prime},c>0$, let in the following $\mathcal{B}=\mathcal B (\varepsilon,\varepsilon^{\prime})$ be the event that the geodesics  $\gamma\big(p_{cN^{4/5}},q_{x^{\ast}N(1-\varepsilon)/2}+(-\log^{9}(N),\log^{9}(N))\big)$ and  $\gamma\big(q_{x^{\ast}N(1+\varepsilon)/2}+(-\log^{9}(N),\log^{9}(N)), p_{(1+\varepsilon)x^{\ast}}\big)$ intersect the boundary 
$\partial_2(\mathcal{S}_N)$ at sites $q_{z_1}$ and $q_{z_2}$ for some
\begin{equation}\label{eq:Z1Z2}
  \frac{1+\varepsilon^{\prime}}{b^2-1}N <  z_1 <  \frac{x^{\ast}N(1-\varepsilon)}{2} \quad \text{ and } \quad  x^{\ast}N(1+\varepsilon) -  \frac{1+\varepsilon^{\prime}}{b^2-1}N > z_2 > \frac{x^{\ast}N(1+\varepsilon)}{2} \, , 
\end{equation} respectively. Then for any $\varepsilon>0$, there exists some $\varepsilon_0=\varepsilon_0(\alpha,\beta,\varepsilon)>0$ such that for all $\varepsilon^{\prime}\in (0,\varepsilon_0)$, $\P(\mathcal{B}) \geq 1-Ce^{-\tilde{c}\log^{2}N}$ holds  for all $N$ large enough, and some $\tilde{c},C>0$.
\end{cor}

\subsection{A random extension and time  change of the environment}\label{sec:RandomExtensionTimeChange}

In the remainder, we follow the strategy  introduced in \cite{SS:TASEPcircle} for the TASEP on the circle to translate the coalescence of geodesics into mixing time bounds. 
In short, we use the coalescence of geodesics starting from two initial growth interfaces in order to argue that the corresponding two exclusion processes become time shifted copies of each other; see Lemma \ref{lem:CoaToFinish} for a precise statement. 
In a next step, we apply a random extension to the environment of the underlying corner growth model to control the time shift up to an order of $N^{1/5}$. For the remaining time shift, we apply a Mermin-Wagner style argument. In contrast to  \cite{SS:TASEPcircle}, we can only modify a narrow strip close to the lower boundary, thus requiring a refined analysis. \\

Recall $x^{\ast}$ from \eqref{eq:xast}, and the component-wise ordering $\succeq$ on $\mathbb Z ^2$. We partition $\mathbb{Z}^2$ as
\begin{equation}
W_- := \{ u\in \Z^2 \colon u \preceq p_{x^{\ast}N/2}+(k , -k) \text{ for some } k\in \Z  \}
\end{equation} and $W_+ := \Z^2 \setminus W_-$. In particular, we have that $\Gamma_{\eta},\Gamma_{\zeta} \subseteq W_-$ holds. Let $(\omega_v)_{v\in \mathcal{S}_N}$ and $(\omega^{\prime}_v)_{v\in \mathcal{S}_N}$ be two independent i.i.d.\ last passage percolation environments on the strip $\mathcal{S}_N$ with respect to boundary parameters $\alpha,\beta>0$. We define for all $i\in \N \cup \{0\}$ the environment $(\omega^{i}_v)_{v\in \mathcal{S}_N}$, with its law denoted by $\P_i$, as
\begin{equation}
\omega^{i}_v := \begin{cases}  \omega_v & \text{ if } v\in W_- \\
 \omega_{v- (i,i)} & \text{ if } v  \in  (i,i) +W_+ \\
 \omega^{\prime}_v & \text{ otherwise} \, .
\end{cases}
\end{equation} for all $v\in \mathcal{S}_N$. Intuitively, we get $(\omega^{i}_v)_{v\in \mathcal{S}_N}$ by keeping the environment $(\omega_v)_{v\in W_- \cap \mathcal{S}_N}$, shifting the environment $(\omega_v)_{v\in W_+ \cap \mathcal{S}_N}$ by $(i,i)$, and filling up the remaining sites using $(\omega^{\prime}_v)_{v\in \mathcal{S}_N}$; see Figure \ref{fig:ExtendStrip}. We denote the corresponding last passage time between $v$ and $w$ in $(\omega^{i}_v)_{v\in \mathcal{S}_N}$ by $T^{i}(v,w)=T_{\alpha,\beta}^{i}(v,w)$. 
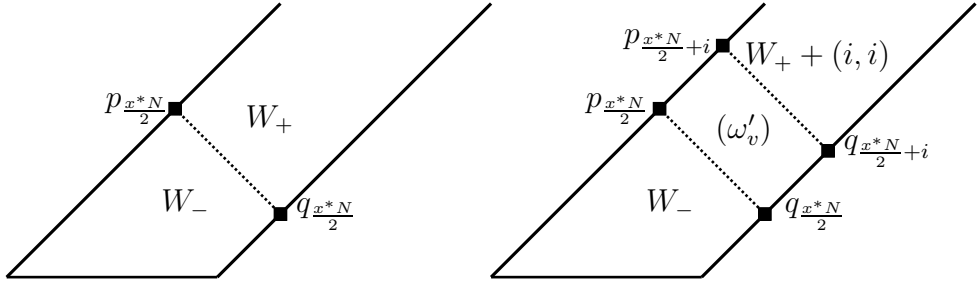
\begin{figure}
    \centering
\begin{tikzpicture}[scale=.56]

\draw[line width=1pt] (0,0) -- (5,0);
 \draw[line width=1.2pt] (0,0) -- (0+6.5,0+6.5);
   \draw[line width=1.2pt] (5,0) -- (11.5,6.5);
  
\node[scale=1] (x1) at (3+1.25,3-1.25) {$W_-$};  
\node[scale=1] (x1) at (5+1.25,5-1.25) {$W_+$};  
\node[scale=1] (x1) at (4-0.9,4) {$p_{\frac{x^{\ast}N}{2}}$};  
\node[scale=1] (x1) at (4+2.5+1.1,4-2.5) {$q_{\frac{x^{\ast}N}{2}}$};  
  
	\filldraw [fill=black] (4-0.15,4-0.15) rectangle (4+0.15,4+0.15);  
	\filldraw [fill=black] (4+2.5-0.15,4-2.5-0.15) rectangle (4+2.5+0.15,4-2.5+0.15);        

\draw[line width=1pt] (0+11.5,0) -- (5+11.5,0);
 \draw[line width=1.2pt] (0+11.5,0) -- (0+6.5+11.5,0+6.5);
   \draw[line width=1.2pt] (5+11.5,0) -- (11.5+11.5,6.5);

	\filldraw [fill=black] (4+11.5-0.15,4-0.15) rectangle (4+11.5+0.15,4+0.15);  
	\filldraw [fill=black] (4+11.5+2.5-0.15,4-2.5-0.15) rectangle (4+11.5+2.5+0.15,4-2.5+0.15);   

	\filldraw [fill=black] (5.5+11.5-0.15,5.5-0.15) rectangle (5.5+11.5+0.15,5.5+0.15);  
	\filldraw [fill=black] (5.5+11.5+2.5-0.15,5.5-2.5-0.15) rectangle (5.5+11.5+2.5+0.15,5.5-2.5+0.15);   
   
\node[scale=1] (x1) at (3+11.5+1.25,3-1.25) {$W_-$};  
\node[scale=1] (x1) at (6.5+11.5+1.25,6.5-1.25) {$W_++(i,i)$};    
\node[scale=1] (x1) at (4.75+11.5+1.25,4.75-1.25) {$(\omega_v^{\prime})$}; 
\node[scale=1] (x1) at (4+11.5-1,4) {$p_{\frac{x^{\ast}N}{2}}$};  
\node[scale=1] (x1) at (4+11.5+2.5+1.2,4-2.5) {$q_{\frac{x^{\ast}N}{2}}$};       
   
\node[scale=1] (x1) at (5.5+11.5-1.3,5.5) {$p_{\frac{x^{\ast}N}{2}+i}$};  
\node[scale=1] (x1) at (5.5+11.5+2.5+1.4,5.5-2.5) {$q_{\frac{x^{\ast}N}{2}+i}$};        
   
\draw[densely dotted, line width=1pt] (4,4) -- (4+2.5,4-2.5);

\draw[densely dotted, line width=1pt] (4+11.5,4) -- (4+2.5+11.5,4-2.5);
\draw[densely dotted, line width=1pt] (5.5+11.5,5.5) -- (5.5+2.5+11.5,5.5-2.5);

%\draw[line width=0.6pt,>=triangle 45, <->] (5.2,0) -- (8.3,0);
%
%\draw[line width=0.6pt,>=triangle 45, <->] (8.5,0.2) -- (8.5,3.3);

%\node[scale=0.8] (x1) at (9.15,3.45) {$z_{\eta}$};
%
%\node[scale=0.8] (x1) at (-0.2,0.4) {$p_0$};
%\node[scale=0.8] (x1) at (4.5,0.4) {$q_{N/2}$};
%
%\node[scale=0.8] (x1) at (11,2) {$(1+b^{2})^{-1}N$};

%\draw[densely dotted, line width=1pt] (0,0) -- (8.5,3.5);

%\draw[densely dotted, line width=1pt] (1.525,1.525) -- (1.525+2.5,1.525-2.5);
%\draw[densely dotted, line width=1pt] (9-1.525,9-1.525) -- (9-1.525+2.5,9-1.525-2.5);

%\draw[red, line width =1 pt] (0,0) to[curve through={(0.7,0.2) ..(1.5,0.3)..(2.45,1.3) .. (5,1.6).. (8.5-0.3,3.3)..(8.5,3.5)..(8.5+0.3,3.5+0.35).. (2.5+9-0.4,-2.5+9-0.1)}] (2.5+9,-2.5+9);
%\draw[red, line width =2 pt] (9,9) to[curve through={(9+0.2,9-0.7) ..(9-0.1,9-1.5)..(9-0.2,9-2.45) }] (9-0.4,9-2.65);

%\draw[blue, line width =2 pt] (2.5,-2.5) to[curve through={(2.6,-1.8) ..(3,-1.5)..(5.5,2.5)..(7,3)..(8,6)}] (9,9);

%	
%		\node[scale=1] (x1) at (-0.7,0){$0$} ;
%		\node[scale=1] (x2) at (2.5+0.7,-2.5){$q_0$} ;
%		\node[scale=1] (x3) at (9+0.7,9){$p_m$} ;
%		\node[scale=1] (x4) at (9+2.5+0.7,9-2.5){$q_m$} ;
%

%	\filldraw [fill=black] (2.5-0.15,-2.5-0.15) rectangle (2.5+0.15,-2.5+0.15);   

 %	\filldraw [fill=black] (8.5-0.15,3.5-0.15) rectangle (8.5+0.15,3.5+0.15);     

%	\filldraw [fill=black] (5-0.15,-0.15) rectangle (5.15,0.15);   
 	
%	\filldraw [fill=black] (-0.15,-0.15) rectangle (0.15,0.15);   

 %	\filldraw [fill=black] (9-0.15,9-0.15) rectangle (9+0.15,9+0.15);     	
 	
	\end{tikzpicture}	
    \caption{\label{fig:ExtendStrip}Visualization of the construction of the environment $(\omega^{i}_v)_{v\in \mathcal{S}_N}$ from $(\omega_v)_{v\in \mathcal{S}_N}$ for some $i\in \N$ by inserting extra line segments to the strip.}
\end{figure}
In the following, we specify the choice of $i$ and $j$ for the environments of two exclusion processes $(\eta_t)_{t \geq 0}$ and  $(\zeta_t)_{t \geq 0}$, started from $\eta$ and $\zeta$, respectively. The number of lines added to the environments for the processes $(\eta_t)_{t \geq 0}$ and $(\zeta_t)_{t \geq 0}$ is given as a coupled pair of random variables $Y(\eta)$ and $Y(\zeta)$, which are both marginally uniform distributed on the set  \begin{equation}\label{def:IterationPoints}
\mathbb{B}:=\left\{ \lfloor i N^{1/10} \rfloor \colon i \in  [N^{6/7}] \right\}    \, .
\end{equation}  In order to determine the coupling between $Y(\eta)$ and $Y(\zeta)$, which we do later on, we require the following lemma on the change of the last passage times when adding lines to the environment. 

\begin{lem}\label{lem:AddRows} For all $y\in [N/2]$, we set 
\begin{equation}
    v_y:= q_{x^{\ast}N/2} + (- (y-1), (y-1)) \, .
\end{equation}
Then there exist $c,C>0$ such that for all $N$ large enough, uniformly in $j\in \mathbb{B}$, the event 
\begin{equation*}
\left\{ T^{j}(v_y,v_{y^{\prime}}+(j,j)) \in \left[ \frac{(b+1)^2}{b}j (1- N^{-1/20}), \frac{(b+1)^2}{b}j (1+ N^{-1/20})\right] \, \forall y,y^{\prime} \in [\log^9(N)] \right\}
\end{equation*} holds with $\mathbb{P}_j$-probability at least $1- Ce^{-c\log^{2} N}$.
\end{lem}
\begin{proof} Note that the lower bound on the last passage time is immediate using Lemma \ref{lem:CombinationLemma} for the last passage time between sites $v_0+p_{\log^9(N)}$ and $v_0+p_{j-\log^9(N)}$. For the upper bound, note that the last passage time $T(v_y,v_{y^{\prime}}+p_j)$ is for all $j \in [N]$ stochastically dominated by the last passage time $H(p_0,p_{j+\log^9(N)})$ in the half-quadrant with boundary parameter $\beta$, and we conclude by again applying Lemma \ref{lem:CombinationLemma}.
\end{proof}
Before we continue with the proof of the upper bound in Theorem \ref{thm:HighLow}, let us give a few comments on Lemma \ref{lem:AddRows}, and on how it allows us to modify last passage times.
First, let us stress that the exponent $-1/20$ is not optimal, and so are the exponents in the definition of $\mathbb{B}$, and all further exponents in this section. In fact, for our arguments, we will only require that the last passage times between any pair of points $(v_y,v_{y^{\prime}}+(j,j))$ close to the lower diagonal $\partial_2(\mathcal{S}_N)$ (where $v_y$ lies on the line segment connecting $p_{x^{\ast}N/2}$ to $q_{x^{\ast}N/2}$, and $v_y+(j,j)$ lies on the line segment connecting $p_{x^{\ast}N/2+j}$ to $q_{x^{\ast}N/2+j}$) is concentrated with high probability. Since by Corollary \ref{cor:ImprovedHitting} the geodesics between sufficiently distant points are very likely pass through the strip close to $q_{x^{\ast}N/2}$, this allows us to control the change of last passage time in the randomly extended environments. Second, let us comment on why we add chunks of $N^{1/10}$ lines at a time. While an individual chunk of $N^{1/10}$ will only marginally change the last passage times, the choice of $\mathbb{B}$ allows us to change last passage times up to an order of $N^{6/7+1/10} \in [N^{0.95},N^{0.96}]$. In particular, together with the  concentration result for last passage times explained before, this allows us reduce a difference of last passage times of order at most $N^{4/5}$ to a difference of order at most $N^{3/20}$, where the exponent $3/20$ comes from the exponent $1/10$ in the chunk size and the concentration exponent $-1/20$. This will be made rigorous in Lemma \ref{lem:AlmostCoupled} using a balls and bins type argument.  \\

We now come back to the proof of the upper bound in Theorem \ref{thm:HighLow}. Recall the event $\mathcal{B}$ from Corollary \ref{cor:ImprovedHitting} which states that the geodesics $\gamma\big(p_{N^{5/6}},v)$ and $\gamma\big(v, p_{(1+\varepsilon)x^{\ast}N}\big)$ intersect $\partial_2 (\mathcal{S}_N)$ in sites $q_{z_1}$ and $q_{z_2}$ with $z_1$ and  $z_2$ satisfying \eqref{eq:Z1Z2} for all $v=q_{x^{\ast}N/2}+(-(k-1),k-1)$ with $k \in [\log^{9}(N)]$. 
 If in addition the event $\mathcal{A}$ from Lemma \ref{lem:InterfaceHittingDiagonal} holds, we set
 \begin{equation}
     S^{\ast} := \max_{v\in \Gamma_{\zeta}}T(v,q_{z_1}) -  \max_{u\in \Gamma_{\eta}}T(u,q_{z_1})
 \end{equation}
to be the difference between the last passage times  $\Gamma_{\eta}$ and $\Gamma_{\zeta}$ to $q_{z_1}$, respectively. Next, we let $\mathcal{C}$ be the event that for all $i,j\in \mathbb{B}$, as well as $v \in \Gamma_{\eta}$ and $u \in \Gamma_{\zeta}$ 
\begin{align}
\gamma^{i}(v, p_{x^{\ast}N(1+\varepsilon)}) \cap \{q_{x^{\ast}N/2}+(k-m,k+m) \text{ for some } k \in  [i] \text{ and } m > \log^{9}(N) \}  &= \emptyset \\
\gamma^{j}(u, p_{x^{\ast}N(1+\varepsilon)}) \cap \{q_{x^{\ast}N/2}+(k-m,k+m) \text{ for some } k \in  [j] \text{ and } m > \log^{9}(N) \} &= \emptyset \, .
\end{align}
Note that $\mathcal{C}$ has probability at least $1- Ce^{-c\log^{2} N}$ for some $c,C>0$ using the same arguments as in the proof of Lemma \ref{lem:AddRows}, together with the first statement in Lemma in \ref{lem:CombinationLemma}, and a union bound over $j\in \mathbb{B}$. Whenever the events $\mathcal{A}$, $\mathcal{B}$ and $\mathcal{C}$ hold, we set for all $i\in \N$
\begin{equation}\label{eq:DifferenceofLPTs}
    S_{i} := T^{i}( q_{z_1},q_{z_2}+(i,i)) - T^{0}(q_{z_1},q_{z_2}) \, .
\end{equation}
The purpose of considering the events $\mathcal{A}, \mathcal{B}, \mathcal{C}$ is the following. 
The next lemma, which is the analogue of Lemma 5.3 in \cite{SS:TASEPcircle}, relates the coalescence of geodesics from two initial interfaces to the property that the corresponding exclusion processes become time shifted versions of each other. Here, we use the convention that for all $i\in \N$ and $A,B \subseteq \Z^2$
\begin{equation}
T^{i}(A,B) := \max_{v \in A, u \in B} T^{i}(v,u) \, .
\end{equation}
\begin{lem}\label{lem:CoaToFinish} Let $\varepsilon>0$ and suppose that the events $\mathcal{A}$,  $\mathcal{B}$ and $\mathcal{C}$ occur. Then the processes $(\eta_t)_{t \geq 0}$ and $(\zeta_t)_{t \geq 0}$ in environments $(\omega^i_v)_{v\in \mathcal{S}_N}$ and $(\omega^j_v)_{v\in \mathcal{S}_N}$ for some $i,j \in \mathbb{B}$ satisfy
\begin{equation}
    \eta_{t+S^{\ast}-S_i}=\zeta_{t-S_j}
\end{equation}
 for all $t \geq \max( T^{i}(\Gamma_{\eta}, \mathbb{L}_{(1+\varepsilon)x^{\ast}N+i}), T^{j}(\Gamma_{\zeta}, \mathbb{L}_{(1+\varepsilon)x^{\ast}N+j}))$, provided that $N$ is sufficiently large.
\end{lem}
\begin{proof} Consider some site $v\in \mathcal{S}_N$ with $v \succeq w$ for some $w\in \mathbb{L}_{(1+\varepsilon)x^{\ast}N}$. Whenever the events $\mathcal{A}$, $\mathcal{B}$ and $\mathcal{C}$ hold, we have that
\begin{equation}\label{eq:LPPshifts}
    T^i(\Gamma_{\eta},v+p_i) + S^{\ast} - S_{i}  = T^j(\Gamma_{\zeta},v+p_j)-S_j
\end{equation} by the construction of $S^{\ast}$, $S_i$ and $S_j$, and the fact that the path between the first and last intersection with $\partial_2(\mathcal{S}_N)$ will, due to the event $\mathcal{B}$, have a distance of at most $\log^9(N)$ to the diagonal $\partial_2(\mathcal{S}_N)$.  Using Lemma \ref{lem:CornerGrowthRepresentation} to convert \eqref{eq:LPPshifts} to the particle representation of the TASEP with open boundaries, we conclude.
\end{proof}

\begin{lem}\label{lem:AlmostCoupled} Fix a pair of initial interfaces $\Gamma_{\eta}$ and $\Gamma_{\zeta}$. Then there exists a coupling between $Y(\eta)$ and $Y(\zeta)$ such that the corresponding exclusion processes $(\eta_t)_{t \geq 0}$ and $(\zeta_t)_{t \geq 0}$ in environments $(\omega^{Y(\eta)}_v)_{v \in \mathcal{S}_N}$ and $(\omega^{Y(\zeta)}_v)_{v \in \mathcal{S}_N}$ satisfy for some constant $C>0$
\begin{equation}
 | S^{\ast}+ S_{Y(\eta)} - S_{Y(\zeta)} | \leq C N^{3/20}
\end{equation}
 with probability at least $1-N^{-1/30}$ for all $N$ sufficiently large.
\end{lem}

\begin{proof} As already observed in the proof of Lemma \ref{lem:InterfaceHittingDiagonal}, note that $|S^{\ast}|\leq c_1N^{4/5}$ holds with probability at least $1-c_2e^{-c_3\log^2 N}$ for positive constants $(c_i)_{i \in [3]}$ by our choice of $\Gamma_{\eta}$ and $\Gamma_{\zeta}$ in \eqref{eq:ChoiceOfGammas} in a common environment $(\omega _{v})_{v \in \mathcal{S}_N}$. Recall that $\P(\mathcal{A} \cap\mathcal{B}\cap\mathcal{C})\geq 1-Ce^{-c\log^2 N}$ for all $N$ large enough. Assuming that the events $\mathcal{A}$, $\mathcal{B}$ and  $\mathcal{C}$ hold, we set $S_{i,j}=S_i-S_j$ for all $i,j\in \mathbb{B}$, recalling \eqref{eq:DifferenceofLPTs}. Let
\begin{equation}
    \mathbb{J}:=[N^{6/7-1/20}] \cup \{0\} \, ,
\end{equation}
and we define for all $x\in \mathbb{J}$ the sets
\begin{align}
\mathcal{I}^{1}_x &:= \Big\{ y\in \mathbb{B} \colon S_{y,0}-S_{0,0} \in \Big[x N^{3/20},(x+1)N^{3/20}\Big] \Big\} \\
\mathcal{I}^{2}_x &:= \Big\{ y\in \mathbb{B} \colon S_{0,y}-S_{0,0}+S^{\ast} \in \Big[x N^{3/20},(x+1)N^{3/20}\Big] \Big\} \, .
\end{align}
Then by Lemma \ref{lem:AddRows}, with probability at least $1-Ce^{-c\log^2 N}$ for some $c,C>0$, each $y\in \mathbb{B}$ is  contained in some $\mathcal{I}^{1}_x$ as well as some $\mathcal{I}^{2}_x$. Furthermore, there exist constants $c_4,c_5,c_6>0$ such that at least $N^{5/6}(1-c_4N^{-1/20})$ of the sets $(\mathcal{I}^{1}_x)_{x\in \mathbb{J}}$, respectively  $(\mathcal{I}^{2}_x)_{x\in \mathbb{J}}$, are non-empty and satisfy
\begin{equation}
|\mathcal{I}^{1}_x| \in \Big[ c_5 N^\frac{1}{20} - c_6, c_5 N^\frac{1}{20} + c_6 \Big] \quad \text{ and } \quad |\mathcal{I}^{2}_x| \in \Big[ c_5 N^\frac{1}{20} - c_6, c_5 N^\frac{1}{20} + c_6 \Big] \, .
\end{equation}
Next, we define two distributions $\pi_1$ and $\pi_2$ on $\mathbb{B}$ such that $(\tilde{Y}_1,\tilde{Y}_2)$ with $\tilde{Y}_1\sim\pi_1$ and $\tilde{Y}_2\sim\pi_2$ satisfy \begin{equation}\label{eq:ModifiedShift}
| S_{\tilde{Y}_1,\tilde{Y}_2}-S^{\ast} | \leq  N^{1/5} \, .
\end{equation} To do so, let $X$ be uniformly at random on  $\mathbb{J}$ and $\tilde{X}$ uniformly at random on $[c_5 N^{1/20}+c_6]$. Let $\tilde{Y}_1$ and $\tilde{Y}_2$ be the $\tilde{X}^{\text{th}}$ smallest element in $\mathcal{I}^{1}_X$ and $\mathcal{I}^{2}_X$, respectively, and $\tilde{Y}_1=\tilde{Y}_2=N^{1/10}$, otherwise. We claim that the total variation distance of $\pi_1$ and $\pi_2$ to the uniform distribution on $\mathbb{B}$ is at most $c_7N^{-1/20}$ for some $c_7>0$, allowing us to conclude the lemma using the coupling representation of the total variation distance; see \cite[Proposition~4.7]{LPW:markov-mixing}. Note that by construction of $\pi_1$ and $\pi_2$
\begin{equation}
\TV{\pi_1 - \pi_{\mathbb{B}} } \leq \pi_1(N^{1/10}) \qquad \text{ and } \qquad  \TV{\pi_2 - \pi_{\mathbb{B}} } \leq \pi_2(N^{1/10}) \, .
\end{equation}
Moreover, it holds that
\begin{align}
\pi_1(N^{1/10}) &\leq \P( X=0, \tilde{X}=1 )+ \P\Big( |\mathcal{I}^{1}_X|< N^{1/10}-2  \Big) + \P\Big( \tilde{X} \geq N^{1/10}-2 \Big) \\ &\leq \frac{1}{N^{6/7}(1+c_4N^{-1/20})} + \frac{c_8}{N^{1/10}} + \frac{c_9}{N^{1/20}} \, ,
\end{align}
for some constants $c_8,c_9>0$, and similarly for $\pi_2(N^{1/10})$, allowing us to conclude.
\end{proof}
Next, we make use of the following lemma on the total variation distance of a family of independent Exponential-$\beta$-random variables. This is given as Lemma 5.8 in \cite{SS:TASEPcircle} when $\beta=1$, relying on ideas from \cite{DEP:FirstPassage}. The proof follows verbatim for general $\beta>0$. 
\begin{lem}\label{lem:MerminWagner} Fix some $M\in \N$ and some $\beta>0$. Consider the random vector $X=(X_1,\dots,X_M)$ of independent Exponential-$\beta$-distributed random variables $X_i$. Let $\delta\in [-\varepsilon,\varepsilon]$ for some $\varepsilon>0$, and define $X^{\delta}=(X^{\delta}_1,\dots,X^{\delta}_M)$ where
\begin{equation}
X_i^\delta := X_i (1+\delta)
\end{equation} for all $i\in [M]$. There exists some $c=c(\beta)>0$ such that uniformly in $\delta$, we have for all $\varepsilon>0$
\begin{equation}
\TV{ X - X^{\delta} } \leq c (M \varepsilon^{2})^{\frac{1}{4}} \, .
\end{equation}
\end{lem}

Next, we prepare for applying Lemma \ref{lem:MerminWagner} in order to remedy the remaining time shift between the exclusion processes. However, in contrast to the arguments in \cite{SS:TASEPcircle}, we only modify the Exponential random variables in a thin rectangle $R_N$ of size $3N^{7/8}\times N^{1/4}$ close to the lower diagonal instead of the entire strip. To see this intuitively, note that when modifying order $N^2$ random variables in the  strip, using Lemma~\ref{lem:MerminWagner}, we could change each of them only at order $o(N^{-1})$, and thus would modify the last passage times for sites at distance $N$ by at most a constant. In contrast, changing only the environment in $R_N$, Lemma~\ref{lem:MerminWagner} ensures that we can modify each Exponential random variable at order $N^{-4/5}$ without changing the law of the environment significantly. Thus, ensuring that with high probability all relevant geodesics pass through order $N$ sites in $R_N$, we can modify the last passage times at order $N^{1/5}$.
%Let us now outline the remaining strategy for the proof of the upper bound on the mixing time for the TASEP in the high density phase using the above tools.  
To make this precise, for every $N \in \N$, we define the rectangle
\begin{equation}
    R_{N} := \left\{q_i+(-k,k) :  i\in \Big[\frac{x^{\ast}N}{2}- 2N^{7/8},\frac{x^{\ast}N}{2} - N^{7/8} \Big] \text{ and } k\in \{0\} \cup [N^{1/4}] \right\}  \, ,
\end{equation} and for all $i\in \mathbb{B}$ and $u \in [0,1]$ the environments $(\omega^{i,u}_v)_{v \in \mathcal{S}_N}$ with
\begin{equation*}
\omega^{i,u}_v := \begin{cases}  (1+  uN^{-27/40})\omega^{i}_v & \text{ if } v \in  R_{N} \\
\omega^{i}_v & \text{ otherwise} \, .
\end{cases}
\end{equation*} 
Again, let us stress that the above exponents are not optimal, but sufficient for our arguments. In order to simplify notation, we write in the following
\begin{equation}
q_{1}^R := q_{\frac{x^{\ast}N}{2}- 2N^{7/8}} \quad \text{ and }  \quad q_{2}^R := q_{\frac{x^{\ast}N}{2}- N^{7/8}} 
\end{equation} for the two corners of the rectangle $R_N$ which lie on $\partial_2(\mathcal{S}_N)$. Note that for all $\varepsilon>0$ sufficiently small and $N$ large enough,by Corollary~\ref{cor:ImprovedHitting} together with Lemma \ref{lem:1} in order to bound the gap between returns to $\partial_2(\mathcal{S}_N)$, we see that with probability at least $1-Ce^{-c\log^2 N}$ for some $c,C>0$, there exist some $z^{\ast}_{1}, z^{\ast}_{2} \in \N$ with
\begin{equation}\label{def:targetzs}
z^{\ast}_{1} <  \frac{x^{\ast}N}{2}- 3N^{7/8} \quad  \text{ and } \quad  \frac{x^{\ast}N}{2}- \frac{1}{2} N^{7/8} \leq   z^{\ast}_{2} \leq \frac{x^{\ast}N}{2}
\end{equation} such that we have for all $i\in \mathbb{B}$
\begin{equation}\label{eq:IntersectionFixEnvironment}
 q_{z^{\ast}_1},q_{z^{\ast}_2} \in \gamma^{i}( p_{N^{5/6}}, \mathbb{L}_{x^{\ast}N(1+\varepsilon)+i})   \, .
\end{equation}
For fixed $i \in \mathbb{B}$, the next lemma provides a similar statement uniformly in the choice of $u$ for the environments $(\omega^{i,u}_v)_{v \in \mathcal{S}_N}$.

\begin{lem}\label{lem:ControlInetersectionUniform} Fix some $\varepsilon>0$. Let  $\gamma^{i,u}$ be the geodesic with respect to the environment $(\omega^{i,u}_v)_{v \in \mathcal{S}_N}$ and take $z_1^{\ast},z_2^{\ast}$ from \eqref{def:targetzs}. Then the event $\mathcal{D}$ given by
\begin{equation}
\mathcal{D}:= \left\{ \exists z_1^{\ast} < \tilde{z}_1 <  \tilde{z}_2 < z_2^{\ast} \,  \colon \, q_{\tilde{z}_1},q_{\tilde{z}_2} \in\gamma^{i,u}(p_{N^{5/6}}, \mathbb{L}_{x^{\ast}N(1+\varepsilon)+i}) \text{ for all } i\in \mathbb{B} \text{ and } u\in [0,1] \right\} \, .
\end{equation} holds with probability at least $1-Ce^{-c\log^2 N}$ for some $c,C>0$, and all $N$ sufficiently large. 
\end{lem}
\begin{proof} Using the same arguments as in the proof of Lemma \ref{lem:NoCrossing}, we see that with probability at least $1-Ce^{-c\log^2 N}$, for all $N$ sufficiently large
\begin{equation}
 T(\pi) < T(q_{z^{\ast}_1},q_{z^{\ast}_2}) - N^{\frac{2}{9}}
\end{equation} holds for every lattice path $\pi$ between  $q_{z^{\ast}_1}$ and $q_{z^{\ast}_2}$, which intersects some site $q_{x}+(-k,k)$ with $x\in [z^{\ast}_1,z^{\ast}_2]$ and $k>N^{1/4}$. Thus, since $\max_{v,w\in R_N}|T^{0,1}(v,w)-T^{0,0}(v,w)|$ is of order at most $N^{1/5}$ by using Theorem \ref{thm:Ledoux} and Proposition \ref{prop:123}, we get that 
\begin{equation}\label{eq:NoLargeExit}
   \gamma^{i,u}(q_{z^{\ast}_1},q_{z^{\ast}_2}) \cap \{ q_{i}+(-k,k) \colon i\in \Z \text{ and } k>N^{1/4} \} = \emptyset \text{ for all } i\in \mathbb{B} \text{ and } u\in [0,1] 
\end{equation} 
holds probability at least $1-Ce^{-c\log^2 N}$. Using now Lemma \ref{lem:CoalescenceAlpha}, we see that with probability at least $1-c_1e^{-c_2\log^2 N}$ for some $c_1,c_2>0$, the last intersection  of $\gamma(p_{N^{5/6}},q_{1}^R+(-N^{1/4},N^{1/4}))$ with $\partial_2(\mathcal{S}_N)$, respectively the first intersection of  $\gamma(q_{2}^R+(-N^{1/4},N^{1/4}),p_{x^{\ast}N(1+\varepsilon)+i})$ with $\partial_2(\mathcal{S}_N)$, has the desired properties.
\end{proof}
Whenever the event $\mathcal{D}$ in Lemma \ref{lem:ControlInetersectionUniform} holds,  we set in the following
\begin{equation}
    T^{\ast,u} := T^{Y(\eta),u}(q_{\tilde{z}_1},q_{\tilde{z}_2}) = T^{Y(\zeta),u}(q_{\tilde{z}_1},q_{\tilde{z}_2}) \, .
\end{equation} In other words, we apply the same modification $T^{\ast,u}$ to the last passage times in the two environments corresponding to the initial configurations $\eta$ and $\zeta$, respectively.
The next lemma shows that $T^{\ast,u}$ can be well-approximated by a linear function in $u$.
\begin{lem}\label{lem:ApproximateT}
   Let $\hat{T}=T(q_{1}^R,q_{2}^R)$ be the passage time between the endpoints in $R_N$ on $\partial_2(\mathcal{S}_N)$. There exist constants $c,C>0$ such that for all $N$ large enough 
\begin{align}
   \mathbb P \Big( \big\{ \big| (T^{\ast ,u}- T^{\ast ,0})  - \hat{T}N^{-27/40}u \big| \le N^{1/6} \text{ for all } u\in [0,1] \big\}  \cap \mathcal{D} \Big) &\ge 1-Ce^{-c\log^2 N} \, .
 \end{align}
\end{lem}
\begin{proof} By Lemma \ref{lem:CombinationLemma} together with Corollary \ref{claim:triangle}, with probability at least $1-Ce^{-c\log^2 N}$ for some $c,C>0$, we have on the event $\mathcal{D}$ that
\begin{equation}\label{eq:Eqx1}
 | T^{\ast,0} -  ( T(q_{\tilde{z}_1},q_{1}^R) + \hat{T} + T(q_{2}^R,q_{\tilde{z}_2})) | \leq N^{1/8} 
\end{equation} for all $N$ large enough. At the same time, note that whenever the event in \eqref{eq:NoLargeExit} holds
\begin{equation}\label{eq:Eqx2}
   | T^{\ast ,u}- T^{\ast ,0} | \leq u N^{-27/40}\max_{v,w \in R_N} T(v,w)
\end{equation}
 for all $u\in [0,1]$. The lemma now follows from \eqref{eq:Eqx1} and \eqref{eq:Eqx2} using that by Proposition~\ref{pro:LineToBoundary} 
\begin{equation}
\P\left( | \hat{T} - \max_{v,w \in R_N} T(v,w) | \leq C_0 N^{7/10} \right) \geq  1-Ce^{-c\log^2 N}
\end{equation} for some $c,C,C_0>0$, and all $N$ sufficiently large.
\end{proof}
As a consequence of Lemmas \ref{lem:ControlInetersectionUniform} and \ref{lem:ApproximateT}, we have the following estimate, where the quantity $u$ is chosen uniformly at random on $[0,1]$.
\begin{cor}\label{cor:2} Let $\mathcal{U}$ and $\tilde{\mathcal{U}}$ be uniform random variables on the interval $[0,1]$. In the setup of Lemma \ref{lem:ApproximateT}, we have that for some $c,C>0$ and all $N$ large enough
\begin{align}
\P\left( \TV{\mathcal{U}N^{-27/40}\hat{T} - (T^{\ast, \tilde{\mathcal{U}}}-T^{\ast,0}) } \leq N^{-1/10} \right) \geq 1 - Ce^{-c\log^2 N} \, .
\end{align}
\end{cor}
\begin{proof}
This is a direct consequence of Lemma \ref{lem:ApproximateT}, observing that on the event $\mathcal{D}$ from Lemma \ref{lem:ControlInetersectionUniform}, the function $u \mapsto T^{\ast,u}$ is piecewise linear, convex and monotone, and thus, the densities of $\mathcal{U}N^{-27/40}\hat{T}$ and $(T^{\ast, \tilde{\mathcal{U}}}-T^{\ast,0})$ differ in almost every point by at most $N^{-1/10}$.
\end{proof}

\begin{proof}[Proof of the upper bound in Theorem \ref{thm:HighLow}] Recall from Lemma \ref{lem:AlmostCoupled} that for any pair of initial configurations $\eta$ and $\zeta$, we can define environments $(\omega^{Y(\eta)}_v)_{v \in \mathcal{S}_N}$ and $(\omega^{Y(\zeta)}_v)_{v \in \mathcal{S}_N}$ such that the corresponding exclusion processes are with probability at least $1-N^{-1/30}$ time shifted by some $\hat{S}$ with $|\hat{S}| \leq CN^{3/20}$ for some constant $C$, provided that $N$ is sufficiently large.
Using Lemma \ref{lem:MerminWagner} and the above modification of the environments on the rectangle $R_N$, we are going to construct two environments $(\bar{\omega}^{\eta}_v)_{v\in \mathcal{S}_N}$ and $(\bar{\omega}^{\zeta}_v)_{v\in \mathcal{S}_N}$ with
\begin{align}\label{eq:TVCoupling1}
\TV{  (\bar{\omega}^{\eta}_v)_{v\in \mathcal{S}_N} -  (\omega^{Y(\eta)}_v)_{v\in \mathcal{S}_N}  } &\leq N^{-1/40} \\
\label{eq:TVCoupling2}\TV{  (\bar{\omega}^{\zeta}_v)_{v\in \mathcal{S}_N} -  (\omega^{Y(\zeta)}_v)_{v\in \mathcal{S}_N}  } &\leq N^{-1/40}
\end{align} such that the exclusion processes according to $(\bar{\omega}^{\eta}_v)_{v\in \mathcal{S}_N}$ and $(\bar{\omega}^{\zeta}_v)_{v\in \mathcal{S}_N}$ coincide with probability at least $1-N^{-1/40}$ for all $N$ sufficiently large, at some time less than $(\hat{a}+1)^2\hat{a}^{-1}x^{\ast}N(1+\varepsilon)$, for $x^{\ast}$ from \eqref{eq:xast}, and where $\varepsilon>0$ is arbitrary, but fixed. Note that this gives the desired bound on the mixing time using the optimal coupling for the total variation distance in \eqref{eq:TVCoupling1} and \eqref{eq:TVCoupling2}; see \cite[Proposition~4.7]{LPW:markov-mixing}. \\

From Corollary \ref{cor:2} and the fact that $\hat{T}$ is of order at most $N^{7/8}$ by construction, note that with probability at least $1-N^{-1/30}$, there
exists a coupling between two random variables $U(\eta)$ and $U(\zeta)$, both uniformly at random on $[0,1]$, such that with probability at least $1-2N^{-1/30}$
\begin{equation}
    T^{\ast,U(\eta)}_{\eta} = T^{\ast,U(\zeta)}_{\zeta} +\hat{S} \, .
\end{equation}
Consider in the following the TASEPs with open boundaries with respect to environments
\begin{equation}
    (\bar{\omega}^{\eta}_v)_{v\in \mathcal{S}_N}=(\omega^{Y(\eta),U(\eta)}_v)_{v\in \mathcal{S}_N} \quad \text{ and } \quad (\bar{\omega}^{\zeta}_v)_{v\in \mathcal{S}_N}=(\omega^{Y(\zeta),U(\zeta)}_v)_{v\in \mathcal{S}_N}  \, . 
\end{equation} By Lemma \ref{lem:CoaToFinish}, the processes coalesce at $\max( \bar{T}^{\eta}(\Gamma_{\eta}, \mathbb{L}_{(1+\varepsilon)x^{\ast}N+Y(\eta)}), \bar{T}^{\zeta}(\Gamma_{\zeta}, \mathbb{L}_{(1+\varepsilon)x^{\ast}N+Y(\zeta)}))$ with probability at least $1-N^{-1/40}$. The desired upper bound on the mixing time follows now from bounding this last passage time from Lemma \ref{lem:CoaToFinish} using Lemmas \ref{lem:InterfaceHittingDiagonal} and \ref{lem:CoalescenceAllGeodesic}.  \\

It remains to verify that \eqref{eq:TVCoupling1} and \eqref{eq:TVCoupling2} are satisfied for this choice of the environments $(\bar{\omega}^{\eta}_v)_{v\in \mathcal{S}_N}$ and $(\bar{\omega}^{\zeta}_v)_{v\in \mathcal{S}_N}$. Note that we modify at most $N^{7/8+1/4}$ many Exponential-$1$ and Exponential-$\beta$ random variables in the rectangle $R_{N}$. Applying now  Lemma~\ref{lem:MerminWagner} twice, with respect to Exponential-$1$ and  Exponential-$\beta$ random variables separately, the claimed statements follow by choosing some suitable $\delta$ of order $N^{-27/40}$ in Lemma~\ref{lem:MerminWagner}.
\end{proof}

%\begin{remark}
%We argue with a random extension and time shift argument for the TASEP in the high and low density phase. Let us emphasize that a corresponding argument would fail for the coexistence line as we would require to modify too many random variables.
%\end{remark}

\section{Lower bound in the high and low density phase}\label{sec:LowerBoundHighLow}

In this section, our goal is to prove a sharp lower bound on the mixing time of the TASEP with open boundaries in the high and the low density phase. As for the upper bound, using the symmetry between particles and holes, it suffices to consider the high density phase where $\beta<\min(\frac{1}{2},\alpha)$. 
The lower bound on the mixing time in the high density phase is the content of the following proposition. 

\begin{prop}\label{pro:LowerHighDensity}
Let  $\varepsilon \in (0,1)$ and assume that $\beta<\min(\frac{1}{2},\alpha)$. Then we have that
\begin{equation}
 \liminf_{N \to \infty}\frac{t^N_{\textup{mix}}(\varepsilon )}{N} \geq \frac{(\hat{a}+1)^2(b+1)(b-1)}{(\hat{a}b-1)(b-\hat{a})} = \frac{(\hat{a}+1)^2}{\hat{a}}x^{\ast}\, .
\end{equation}
\end{prop} In order to show Proposition \ref{pro:LowerHighDensity}, we start by recalling the following well-known estimate on number of particles in a subset of $[N]$ under the stationary distribution $\mu_N$ of the TASEP with open boundaries. It can be found for example as equation (1.3) in \cite{BW:Density}.

\begin{lem}\label{lem:DensityEquilibrium} 
Let $c \in (0,1]$ be fixed, and assume that $\beta<\min(\frac{1}{2},\alpha)$. Then for all $\varepsilon >0$, the invariant measures $(\mu_N)_{N \in \N}$ of the TASEP with open boundaries satisfy
\begin{equation}\label{eq:InvaraintCLT}
\lim_{N \rightarrow \infty} \mu_N\left( \bigg| \frac{1}{N}\sum_{i=1}^{cN} \Big(\eta(i)  -  (1-\beta)\Big) \bigg| < \varepsilon \right) =  1 \, .
\end{equation} 
\end{lem}
Recall the quantity $x^{\ast}$ from \eqref{eq:xast}. We will argue that  there exists some $\varepsilon>0$ and some $\delta>0$, depending only on $\alpha,\beta,\varepsilon>0$, such that for all $y\in [\delta N]$, the geodesic from $p_0$ to $p_{x^{\ast}N(1-\varepsilon)}+(y,-y)$ does with high probability in $N$ not touch the lower boundary $\partial_2(\mathcal{S}_N)$ of the slab. Starting from the empty initial configuration, this will allow us to show that at time $(1-\varepsilon)(\hat{a}+1)^2\hat{a}^{-1}x^{\ast}N$, the number of particles in the segment $[\delta N]$ does not match  Lemma~\ref{lem:DensityEquilibrium} for the number of particles in equilibrium. For all $\delta,\varepsilon>0$ and $N\in \N$, we set
\begin{equation}
    B^N_{\delta,\varepsilon} := \left\{ (x+y-1,x) \in \mathcal{S}_N \colon (1-\varepsilon-2\delta)x^{\ast}N \leq x \leq (1-\varepsilon)x^{\ast}N \text{ and } y\in [\delta N] \right\} \, ,
\end{equation} and denote by $v_1^B,v_2^{B},v_3^{B},v_4^{B}$ the four corners of the parallelogram $B^N_{\delta,\varepsilon}$ in counter-clockwise order, starting at $v^B_1:=p_{(1-\varepsilon-2\delta)x^{\ast}N}$.

\begin{lem}\label{lem:NoCrossingLowerBound} For all $\varepsilon>0$ sufficiently small, there exists $\delta=\delta(\varepsilon,\alpha,\beta)>0$ such that $\gamma(p_0,v) \cap \partial_2(\mathcal{S}_N) = \emptyset$ holds for all $v \in B^N_{\delta,\varepsilon}$ and $N$ is sufficiently large with probability at least $1 - C e^{-c\log^2 N}$ with the constants $c,C>0$ depending only on $\alpha,\beta,\varepsilon>0$.
\end{lem}
\begin{proof} 
Note that it suffices to show that for sufficiently small $\delta>0$
\begin{equation}\label{eq:TargetLowerBound}
\P\left( T(p_0,v_1^B)  >  \max_{w \in \partial_2(\mathcal{S}_N)} T(p_0,w) + T(w,v_3^B)\right) \geq 1 - C e^{-c\log^2 N} 
\end{equation} for some $c,C>0$, and all $N$ sufficiently large. Using the second statement in Lemma \ref{lem:CombinationLemma} with $n=N-1$, we see that
\begin{equation}\label{eq:UpperBoundLPTLower}
  \P\left(  T(p_0,v^B_1) \geq (1-\varepsilon-2\delta) \frac{(\hat{a}+1)^2}{\hat{a}} x^{\ast}N - c_1\sqrt{N}\log^{4}N \right) \geq 1 - c_2 e^{-c_3\log^2 N} 
\end{equation} for some positive constants $(c_i)_{i \in [3]}$. Hence, to show \eqref{eq:TargetLowerBound}, it is sufficient to prove that
\begin{equation}\label{eq:TargetSmall}
\P\left(  \max_{w \in \partial_2(\mathcal{S}_N)} T(p_0,w) + T(w,v_3^B) <  (1-\varepsilon-3\delta) \frac{(\hat{a}+1)^2}{\hat{a}} x^{\ast}N \right) \geq 1 - C e^{-c\log^2 N} 
\end{equation} for some $\delta>0$. We claim that for all $\varepsilon>0$ sufficiently small, there exist constants $c_4,c_5,c_6>0$ such that for all $N$ sufficiently large the event 
\begin{equation}\label{eq:CompareTouches}
\left\{   \max_{w \in \partial_2(\mathcal{S}_N)} T(p_0,w) + T(w,v_3^B) <  T(p_0,q_{\frac{x^{\ast}N(1-\varepsilon)}{2}})+ T(q_{\frac{x^{\ast}N(1-\varepsilon)}{2}},v^B_3) + c_4N^{4/5} \right\}
\end{equation} holds with probability at least $1 - c_5 e^{-c_6\log^2 N}$. To see this, we recall the estimates from Section \ref{sec:UpperBoundHighLow}. More precisely, similarly to the proof of Lemma \ref{lem:CoalescenceAllGeodesic}, we can assume without loss of generality that the maximum at the left-hand side of \eqref{eq:CompareTouches} is attained for some $w=q_x$ with $x \geq x^{\ast}N(1-\varepsilon)/2$. 
Then for all $\varepsilon>0$ sufficiently small, Lemma \ref{lem:CoalescenceAlpha} guarantees that for all $N$ large enough, with probability at least $1 - C e^{-c\log^2 N}$ for $c,C>0$, there exists some $z \in [-N^{4/5},N^{4/5}]$ such that $q_{z^{\ast}}:=q_{\big(\frac{N}{b^2-1}+N/2+z\big)} \in \gamma(p_0,q_x)$. Together with Lemma~\ref{lem:NoCrossing} to control the probability that the geodesic from $q_{z^{\ast}}$ to $q_x$ will not intersect $\partial_1(\mathcal{S}_N)$, and Corollary \ref{claim:triangle} together with Lemma \ref{lem:CombinationLemma} to bound the difference of the last passage times $T(q_{z^{\ast}},q_{ x^{\ast}N(1-\varepsilon)/2})$ and $T(q_{z^{\ast}},q_{x})$ uniformly in the choice of $z$, we obtain \eqref{eq:CompareTouches}. 
Using Proposition \ref{pro:LineToBoundary} and Lemma \ref{lem:CoalescenceAllGeodesic}, a computation now shows that the event 
\begin{equation}
\left\{ T(p_0,q_{\frac{x^{\ast}N(1-\varepsilon)}{2}})+ T(q_{\frac{x^{\ast}N(1-\varepsilon)}{2}},v^B_3)  \leq  \frac{(\hat{a}+1)^2}{\hat{a}} x^{\ast}N -  \frac{(b+1)^2}{b}\varepsilon x^{\ast}N + \frac{3\delta(b+1)^2}{b}N \right\} \label{eq:CompareTouchesExtra} 
\end{equation} holds with probability at least $1 - C e^{-c\log^2 N}$ with  constants $c,C>0$, and for all $N$ large enough. Recalling that $b>\hat{a}$, we can choose now  $\delta>0$ sufficiently small, and combine \eqref{eq:CompareTouches} and \eqref{eq:CompareTouchesExtra} to obtain \eqref{eq:TargetSmall}, and thus \eqref{eq:TargetLowerBound}. This finishes the proof.
\end{proof}

Using Lemma \ref{lem:NoCrossingLowerBound}, we can control the last passage time from $p_0$ to any site in $B_{\delta,\varepsilon}^N$ using moderate deviation estimates from Proposition \ref{prop:123} for last passage percolation on the half-quadrant when starting and ending at the boundary, together with the results in Proposition~\ref{pro:LineToBoundary} and Lemma~\ref{lem:CoalescencePointToPoint} when starting the geodesic at the boundary, but ending away from the boundary. This is summarized in the following lemma. 

\begin{lem}\label{lem:LowerBoundSqueeze} Let $\varepsilon>0$, and choose $\delta=\delta(\varepsilon,\alpha,\beta)>0$ sufficiently small such that the statement in Lemma \ref{lem:NoCrossingLowerBound} holds. Then there exists a constant $\theta(\varepsilon,\delta)>0$ such that
\begin{align}%\label{eq:LowerBoundSqueeze}
 \mathcal{C}_1 &:=  \left\{  T(p_0,p_{x^{\ast}N(1-\varepsilon-\delta)}) < \frac{(\hat{a}+1)^2}{\hat{a}}x^{\ast}N(1-\varepsilon-\delta
    +\theta^2) \right\} \\
 \mathcal{C}_2 &:=  \left\{    T\Big(p_0,p_{x^{\ast}N(1-\varepsilon-\delta+\theta^2})  + \Big( \theta x^{\ast}N,-\frac{b+\hat{a}}{2} \theta x^{\ast}N\Big)\Big) > \frac{(\hat{a}+1)^2}{\hat{a}}x^{\ast}N(1-\varepsilon-\delta
    +\theta^2) \right\}\, 
\end{align} satisfy $\P(\mathcal{C}_1\cap \mathcal{C}_2) \geq 1- Ce^{-c\log^2 N}$ for some $c,C>0$, and all $N$ large enough. 
\end{lem}
\begin{proof} Using the estimates in Lemma \ref{lem:NoCrossingLowerBound}, it suffices to bound the probability of the above two events with respect to the last passage percolation  on the half-quadrant with boundary parameter $\alpha$. For any fixed $\theta>0$, Lemma \ref{lem:CombinationLemma} yields that  $\P(\mathcal{C}_1) \geq 1- Ce^{-c\log^2 N}$ for some $c,C>0$, and all $N$ large enough. To see that $\P(\mathcal{C}_2) \geq 1- Ce^{-c\log^2 N}$ holds for some $c,C>0$, and all $N$ large enough, observe that, heuristically, by Proposition~\ref{pro:LineToBoundary} the last passage times from $p_0$ to $p_{x^{\ast}N(1-\varepsilon-\delta+\theta^2}  +  ( y , -by )$ are of the same leading order for all $y \in [N]$.
Thus, choosing $\theta>0$ sufficiently small, the desired bound on the probability of the event $\mathcal{C}_2$ follows by Lemma \ref{lem:CoalescencePointToPoint} and a computation. 
\end{proof}
We have now all tools to show sharp lower bounds on the mixing time in Theorem \ref{thm:HighLow}.
\begin{proof}[Proof of Proposition \ref{pro:LowerHighDensity}]
Consider the TASEP with open boundaries started from the all empty initial configuration, and recall Lemma \ref{lem:CornerGrowthRepresentation} on the  coupling of the TASEP with open boundaries to last passage percolation on the slab. For fixed $\varepsilon>0$, set
\begin{equation}
\tilde{t} := \frac{(\hat{a}+1)^2}{\hat{a}}x^{\ast}N(1-\varepsilon-\delta
    +\theta^2)
\end{equation} with respect to sufficiently small constants $\delta$ and $\theta$ from Lemma \ref{lem:LowerBoundSqueeze}. Suppose that the event $\mathcal{C}_1$ holds. Then by the coupling between the TASEP with open boundaries and last passage percolation, we see that at most $x^{\ast}N(1-\varepsilon-\delta)$ particles have entered the segment until time $\tilde{t}$. Similarly, whenever the event $\mathcal{C}_2$ occurs, we see that until time $\tilde{t}$, at least $x^{\ast}N(1-\varepsilon-\delta+\theta^2-(b+\hat{a})/2)$ particles have traversed site $x^{\ast}N(1-\varepsilon-\delta+\theta^2+(b+\hat{a})/2+\theta)$ in the segment. Using these observations, a computation now shows that for some $c\in (0,1)$ and $\tilde{c}<1-\beta$, whenever the events $\mathcal{C}_1$ and $\mathcal{C}_2$ occur, the first $c N$ sites of the segment contain at time $\tilde{t}$ at most $\tilde{c} c N$ particles.  
Hence, combining Lemma \ref{lem:DensityEquilibrium} and Lemma \ref{lem:LowerBoundSqueeze}, and using the definition of the total variation distance, we see that for any fixed $\varepsilon^{\prime}>0$, 
\begin{equation}
    t_{\text{mix}}^N(1-\varepsilon^{\prime}) \geq \tilde{t}
\end{equation} for all $N$ sufficiently large. Since $\varepsilon>0$ for $\tilde{t}$ was arbitrary, this finishes the proof.
\end{proof}

\section{Proof of the upper bound in the coexistence phase}\label{sec:UpperCoexistence}

In this section, we present the upper bound on the mixing time of the TASEP with open boundaries in the coexistence line. In contrast to the high and the low density phase, where the geodesics are attracted to one boundary of the strip, we will show for the coexistence line that the geodesics of length of order $N^{2}$ intersect with positive probability both boundaries. Using then recent arguments in \cite{S:MixingTASEP} on the mixing time of the TASEP in the triple point, this allows us to conclude.

%Let $\gamma (M)$ be the geodesic from $(N,0)$ to $(M,M)$ and define the event
%\begin{equation}
%    \mathcal A _M:= \big\{ \text{There are } N\le x_1\le x_2 \le M \text{ such that }  (x_1,x_1),(x_2,x_2-N)\in \gamma (M)  \big\}.
%\end{equation}
%In words, $\mathcal A _M$ is the event that $\gamma $ intersects the left boundary and then the right boundary before reaching $(M,M)$.
%
%

\subsection{Traversing of geodesics in the slab on the coexistence line}

The following lemma shows that any geodesic connecting two sites on $\partial_1(\mathcal{S}_N)$ of distance of order $N^2$ has a positive probability to intersect $\partial_2(\mathcal{S}_N)$, uniformly in $N$.

\begin{lem}\label{prop:2} Let $\alpha=\beta<\frac{1}{2}$. For all $\varepsilon >0$, there exists some $\delta=\delta(\alpha,\beta,\varepsilon) >0$ such that the following holds. For all $N$ sufficiently large and $m\ge \varepsilon N^2$, we have that 
\begin{equation}
\P\left( \gamma(p_0,p_m) \cap \partial_2(\mathcal{S}_N) \neq \emptyset \right) \geq \delta . 
\end{equation}
\end{lem}

\begin{proof} Fix some environment $(\omega_{v})_{v \in \mathcal{S}_N}$ and note that it suffices to prove the lemma for $m:=\lceil \varepsilon N^2 \rceil $. Let $\tilde{\gamma}$ be the largest up-right path in $(\omega_{v})_{v \in \mathcal{S}_N}$ from $p_0$ to $p_m$, which does not touch $\partial_2(\mathcal{S}_N)$. 
Similarly, let $\bar{\gamma}_1$ be the largest path from $p_0$ to $p_m$ restricted to stay in the strip $\{(x,y) : y\le x\le y+\log ^9 (N^2)\}$.  Note that the last passage times in  $(\omega_{v})_{v \in \mathcal{S}_N \setminus \partial_2(\mathcal{S}_N)}$ are stochastically dominated by last passage percolation on the half-quadrant with boundary parameter $\alpha$. Thus, we get from Proposition~\ref{cor:1} that for all $N$ large enough
\begin{equation}\label{eq:6A}
\P(\tilde{\gamma} = \bar{\gamma}_1 ) \geq 1-C e^{-c\log^{2} N} \, 
\end{equation} with $c,C>0$. Moreover, by Theorem~\ref{thm:corwin} and \eqref{eq:6A}, for all $N$ sufficiently large, 
\begin{equation}\label{eq:6B}
    \mathbb P \big( T(\bar{\gamma}_1 ) \le m/\rho _{\alpha } \big) \ge c_{\varepsilon }
\end{equation} for some suitable constant $c_{\varepsilon }>0$. Let 
$\bar{\gamma}_2$  be the largest up-right path between $q_{N/2}$ and $q_{m-N/2}$ in $(\omega_{v})_{v \in \mathcal{S}_N}$, restricted to stay in  $\{ (x,y): y+N-\log ^9(N^2) \le x \le y+N   \}$. Again, by Theorem~\ref{thm:corwin} and Proposition~\ref{cor:1}, we get that
\begin{equation}\label{eq:6C}
    \mathbb P \big( T( \bar{\gamma}_2 ) > m/\rho _{\alpha } \big) \ge \tilde{c}_{\varepsilon }
\end{equation} for some suitable constant $\tilde{c}_{\varepsilon }>0$. Since the events in \eqref{eq:6B} and  \eqref{eq:6C} are defined with respect to disjoint parts of the environment $(\omega_{v})_{v \in \mathcal{S}_N}$, they are independent and therefore the intersection of the events in \eqref{eq:6A}, \eqref{eq:6B} and  \eqref{eq:6C} holds with probability at least $\delta :=c_{\epsilon }\tilde{c}_{\epsilon }/2$ for all $N$ sufficiently large.
Finally, note that $T(p_0,p_m) > T(\bar{\gamma}_2)$, and therefore on the intersection of the events in  \eqref{eq:6A}, \eqref{eq:6B} and  \eqref{eq:6C} we have that $T(p_0,p_m)\ge T(\bar{\gamma _2}) > T(\bar{\gamma _1 } )=T(\tilde{\gamma })$. It follows that the geodesic $\gamma (p_0,p_m)$ satisfies $\gamma (p_0,p_m) \cap \partial_2(\mathcal{S}_N) \neq \emptyset$ on the intersection of these events.
\end{proof}

As a consequence, we have the following lemma on intersecting $\partial_1(\mathcal{S}_N)$ and $\partial_2(\mathcal{S}_N)$.

\begin{lem}\label{thm:2}
For any $\varepsilon >0$, there exists $\delta >0$ such that for all $N$ sufficiently large and $m\ge \varepsilon N^2$ we have that $\P(\mathcal{A}_m) \geq \delta$ holds, where
\begin{equation}
\mathcal{A}_m :=  \big\{    p_{x_1},q_{x_2}\in \gamma(q_{-N/2},p_m)  \text{ for some } N\le x_1\le x_2 \le m \big\} \, .
\end{equation}
\end{lem}

%We can now easily prove Lemma~\ref{thm:2}.

\begin{proof}
Let $\gamma _1$ be the geodesic from $(N,0)$ to $(m/3+N,m/3)$ and let $\gamma _2$ be the geodesic from $(2m/3,2m/3)$ to $(m,m)$. Let $\mathcal B _1$ be the event that $\gamma _1$ intersects the upper boundary $\partial_1(\mathcal{S}_N)$ and let $\mathcal B _2$ be the event that $\gamma _2$ intersects the lower boundary $\partial_2(\mathcal{S}_N)$ of the slab. By Lemma~\ref{prop:2}, translation invariance and symmetry we have that $\mathbb P (\mathcal B _1)\ge c$ and $\mathbb P (\mathcal B _2 )\ge c$ for some suitable constant $c>0$. These events are independent, and therefore $\mathbb P (\mathcal B _1 \cap \mathcal B _2)\ge c^2$. This finishes the proof as the event $\mathcal B _1 \cap \mathcal B _2 $ implies the event $\mathcal{A}_m$.
\end{proof}

%The following theorem is used for the upper bound.
%Let 
%\begin{equation}
%    T_1(M):=T((0,0), (M+N,M))\quad \text{and} \quad T_2 (M) :=T((0,0),(M+N,M+N)).
%\end{equation}

%\begin{thm}\label{thm:3}
%For all $R\ge 1$ we have that 
% \begin{equation}
%     \mathbb P \Big( T_2(M)\ge  T_1(M)+\sqrt{N} \quad \text{for all $M$ with }\quad  |M -RN^2|\le N^{3/2} \Big) \ge \frac{1}{2}+e^{-CR}.
% \end{equation}
%\end{thm}
%
%\begin{cor}
% For time $t=RN^2/
% \rho _{\alpha }$ we have that 
%\begin{equation}
%    \mathbb P ( |\eta _t| < N/2  ) \ge 1/2+e^{-CR}-o(1).
%\end{equation}
%\end{cor}

\subsection{Proof of the upper bound in Theorem~\ref{thm:1}}

In the following, we prove the upper bound in Theorem~\ref{thm:1} using Lemma~\ref{thm:2}. The proof follows verbatim the arguments for Theorem 1.2 in \cite{S:MixingTASEP} when replacing Lemma 5.5 in \cite{S:MixingTASEP} by our Lemma \ref{thm:2}. Nevertheless, for the convenience of the reader and the sake of a self-contained proof, we will in the following outline the main steps of the argument, and refer to \cite{S:MixingTASEP} for  full details.   \\

 Recall from Section \ref{sec:Disagreement} the disagreement process $(\xi_t)_{t \geq 0}$ between the all full configuration $\mathbf{1}$ and the all empty configuration $\mathbf{0}$, i.e.\ the TASEP with open boundaries which initially contains only second-class particles. We define in the following  an event $A_T$, depending on the time $T>0$ such that whenever $A_T$ holds,  the configuration $\xi_T$ contains no second-class particles.  %This follows a similar strategy as presented in Section 5 of \cite{S:MixingTASEP} for the TASEP with open boundaries in the triple point. \
 In order to define $A_T$, we assign labels $u \in \mathbb{N}$ to the first class particles as they enter the segment, i.e.\ the first particle entering at site $1$ receives label $1$, the second particle label $2$, and so on. Let $\ell_t(u)$ be the position of the particle with label $u$, i.e.\ the $u^{\textup{th}}$ particle which entered the segment, at time $t \geq 0$. We let $A_T$ be the following event: There exists some $k=k(T)$, and a sequence of times $0=t_0 \leq t_1 \leq t_2 \leq \dots \leq t_k \in [0,T]$ such that 
\begin{equation}\label{eq:EmptyCondition}
\xi_{\ell_s(i) + 1} = 0
\end{equation} holds for all $s \in [t_{i-1},t_{i})$ with $i\in [k]$, as well as that we have for some $\tau \in  [t_{k-1},t_{k})$
\begin{equation}\label{eq:TerminationCondition}
\xi_{\tau}({k}) = N \, . 
\end{equation}
Here, we set $\xi_{\ell_s(u) + 1} = 0$ if label $u$ was not assigned by time $s$, or the respective particle reached site $N$. In words, $A_T$ guarantees a sequence of time intervals such that during the $i^{\text{th}}$ interval, the particle of label $i$ never sees a first or second-class particle at the site directly in front of it. In addition, by \eqref{eq:TerminationCondition}, the particle with label $k$ will reach site $N$ before time $T$. Intuitively, condition \eqref{eq:EmptyCondition} and induction ensure that at time $s \in [t_{i-1},t_{i})$, there are no second class particles to the left of the particle with label $i$. Together with condition \eqref{eq:TerminationCondition}, this guarantees that all second class particles have left by time $\tau \leq T$; see also Figure 8 in \cite{S:MixingTASEP} for a visualization. This is formalized in the next lemma. 
\begin{lem}\label{lem:AtToExit}
Let $\tau^{\prime}$ be the first time that the disagreement process $(\xi_t)_{t \geq 0}$ contains no second class particles. Then for all $T>0$, we have that $A_T \subseteq \{\tau^{\prime}<T \}$.
\end{lem}
\begin{proof}[Sketch of proof] Suppose that the event $A_T$ occurs with respect to some $0=t_0 \leq t_1 \leq t_2 \leq \dots \leq t_k \in [0,T]$ and $\tau \in  [t_{k-1},t_{k})$. Consider the largest integer $j-1$ such that the particle with label $j-1$ has not entered by time $t_{j-1}$. If no such time point exists, we set $j=1$. We claim that by induction, for all $i\in [j,k]$, no second-class particle is located to the left of the particle with label $i$ at time $s$ for all $s \in [t_{i-1},t_i)$. To see this, consider the time at which the particle with label $j$ enters the segment (recall that this is the $j^{\text{th}}$ particle which enters the segment). By condition \eqref{eq:EmptyCondition}, for all $s<t_{j}$, no second-class particle can be at time $s$ to the right of the particle with label $j$. For $i\in [j,k]$, assume now that particle $i$ has no second-class particles to its right at time $t_i$. Since it sees neither a first nor a second-class particle for all $s \in  [t_i,t_{i+1})$ to its right, and the particle with label $i+1$ is always placed to the left of the particle with label $i$, there is no second-class particle to the left of particle $i+1$ at time $t_{i+1}$. Thus, by induction together with condition \eqref{eq:TerminationCondition}, this ensures that all second-class particles left the segment by time $T$, implying that $\{\tau^{\prime} \leq T \}$ holds. 
\end{proof}

We will now formulate a sufficient condition for the event $A_T$. To do so, we first consider the projection of the second-class particles in $(\xi_t)_{t\geq 0}$ to first class particles. Observe that the resulting process is again a TASEP with open boundaries $(\eta_t)_{t \geq 0}$, started now from the all occupied initial condition $\mathbf{1}$. Crucially, since we do not need to distinguish between first and second-class particles for the occurrence of the event $A_T$, it suffices to study the event $A_T$ for the process $(\eta_t)_{t \geq 0}$. Let us stress that we keep the same notion for the labels of particles as they enter the segment. Studying the event $A_T$ for $(\eta_t)_{t \geq 0}$ has the advantage that we can use its last passage percolation representation, where the initial growth interface corresponding to $\mathbf{1}$ spans from $p_0$ to $q_{-N/2}$. \\ 

Our key observation is that we can provide a sufficient condition for the event $A_T$ to occur using \textbf{semi-infinite geodesics}, i.e.\ a lattice path $\pi^{\ast}=(v_1,v_2,\dots)$ in the last passage percolation interpretation of the TASEP with open boundaries such that $\pi^{\ast}$ restricted between $v_i$ and $v_j$ is equal to $\gamma(v_i,v_j)$ for all $i<j$. The almost-sure existence of a unique semi-infinite geodesic on the slab $\mathcal{S}_N$ for any starting site $v_1 \in \mathcal{S}_N$ is guaranteed by Lemma 3.3 in \cite{S:MixingTASEP}. The following lemma relates the event  $A_T$ for the TASEP with open boundaries to the event $\mathcal{A}_M$ for last passage percolation on the strip under the coupling from Lemma \eqref{lem:CornerGrowthRepresentation}. With a slight abuse of notation, we denote this coupling again by $\P$.
\begin{lem}\label{lem:EventAT} For all $t,m,N>0$, we have that
\begin{equation}
    \P(A_T) \geq \P(\mathcal{A}_m) - \P(T(q_{-N/2},\mathbb{L}_{m}) >T) \, .
\end{equation} 
\end{lem}
\begin{proof}[Sketch of proof]
We argue that $A_T$ occurs when the semi-infinite geodesic $\pi^{\ast}$ starting at $q_{-N/2}$ first touches the boundary $\partial_1(\mathcal{S}_N)$ at some $p_{x_1}$, and afterwards the boundary $\partial_2(\mathcal{S}_N)$ at some $q_{x_2}$, while in addition $T(q_{-N/2},\mathbb{L}_{m})  \leq T$ holds. To do so, let us note that the sites on the semi-infinite geodesic $\pi^{\ast}=(v_1=q_{-N/2},v_2,\dots)$ satisfy the recursion
\begin{equation}\label{eq:semiinf}
v_{i} = \textup{arg}\max\big( T(q_{-N/2},v_{i+1})-(0,1),T(q_{-N/2},v_{i+1})-(1,0)\big) 
\end{equation} 
for all $i>N$, which follows by an induction argument. Let us now give an interpretation of the relation \eqref{eq:semiinf} for the process $(\eta_t)_{t \geq 0}$.  For all $(i,j),(i+1,j)\in\pi^{\ast} $ with some $i,j\geq 1$, $T(q_{-N/2},(i,j))$ is the time at which the particle with label $j$ jumps from site $i-j-1$ to site $i-j$. Moreover, the particle with label $j-1$ has already reached site $i-j+2$. Hence, for all  $s \in [T(q_{-N/2},(i,j-1)), T(q_{-N/2},(i,j)) ) $,  the particle with label $j$ has no particle at its right neighbor. Similarly, for $(i,j),(i,j+1)\in\pi^{\ast} $, the particle with label $j+1$ sees no particles at its right neighbor for all times $s\in [T(q_{-N/2},(i,j)),T(q_{-N/2},(i,j+1)) )$. We refer to Figure 10 in \cite{S:MixingTASEP} for a visualization of this correspondence. Using the recursion \eqref{eq:semiinf} for the semi-infinite geodesic and its interpretation for the TASEP with open boundaries, we can choose $t_{j}$ in the definition of $A_T$ to be the largest last passage time of a site $(\cdot,j+x_1) \in \pi^{\ast}$ before the semi-infinite geodesic touches again the boundary at the point $q_{x_2}$. Following now the semi-infinite geodesic between $p_{x_1}$ and $q_{x_2}$, this yields that under the event $\mathcal{A}_m \cap \{T(q_{-N/2},\mathbb{L}_m) \leq T \}$, the event $A_T$ occurs with this choice of the sequence $(t_j)$. 
\end{proof}

%We have now all tools to show the upper bound in Theorem \ref{thm:1}. 
\begin{proof}[Upper bound in Theorem \ref{thm:1}]
In order to bound the $\varepsilon$-mixing time of the TASEP with open boundaries, it suffices by Lemma \ref{lem:DisagreementProcess} to show that with probability at least $1-\varepsilon$, all second-class particles in the disagreement process $(\xi_t)_{t \geq 0}$ have left. By Lemma~\ref{thm:2}, Lemma~\ref{lem:EventAT} and  Theorem \ref{thm:Ledoux} (dominating the environment on $\mathcal{S}_N$ by an i.i.d.\ Exponential-$\alpha$ environment on $\Z^2$), there exist some $T$ of order $N^2$ such that $\P(A_T)$ holds with some positive probability $\delta>0$, uniformly in $N$. When $\varepsilon<\delta$, we use the Markov property of the TASEP with open boundaries to note that at time $T \lceil \log(1/\varepsilon)/\log(1/(1-\delta)) \rceil$, all second-class particles must have left with probability at least $1-\varepsilon$. This finishes the proof.
\end{proof}

\section{Lower bound in the coexistence phase}\label{sec:LowerCoexistence}

The main result of this section is the lower bound in Theorem~\ref{thm:1}, which we state as the following proposition.

\begin{prop}\label{thm:11}
Let $\alpha=\beta<1/2$, $A\ge 1 $ and let $N\ge N_{0}(A)$ for some suitable $N_0\in \N$. Let $(\eta_t)_{t \geq 0}$ be the TASEP with open boundaries, started from the all empty initial condition. Set $t_0:=A N^2/\rho _{\alpha }$. We have that for some constant $C>0$
\begin{equation}
    \TV{ \P(\eta_{t_0}\in \cdot \, ) - \pi} \ge e^{-CA } \, .
\end{equation}
In particular, letting $\varepsilon :=e^{-CA}$, we obtain that for some suitable constant $c>0$ and all $N$ sufficiently large\begin{equation}
t^N_{\textup{mix}}(\varepsilon ) \ge t_0 \ge c\log (1/\varepsilon )N^2  \, .    
\end{equation}
\end{prop}

We start with the following bound on the stationary distribution. The stationary measure $\mu_N$ of the TASEP with open boundaries in the coexistence phase satisfies \begin{align}\begin{split}\label{eq:SymmetryInvariant}
 \mu_N (| \{ x \colon \eta(x)=0 \} | > N/2)  =
\begin{cases} \frac{1}{2} & \text{ if } N \text{ is odd}\\ 
\frac{1}{2}\big(1 - \mu_N ( | \{ x \colon \eta(x)=0 \} |  = N/2)\big) & \text{ if } N \text{ is even} 
\end{cases}
\end{split}
\end{align} by the symmetry between particles and empty sites. In particular, we have
\begin{equation}\label{eq:1/2}
    \mu_N (| \{ x \colon \eta(x)=0 \} | > N/2) \le 1/2 \, ,
\end{equation}
which we will use as an event in the definition \eqref{def:TVDistance} of the total variation distance for a lower bound on the mixing time. Recall the notion of $p_i$ and $q_i$ from \eqref{def:PointsPQ}.
%Next, we obtain a lower bound on the probability of this event in the TASEP process with empty initial condition. To this end, as before, we study the last passage percolation on the slab. For $n\in \mathbb Z$ define the points
%\begin{equation}
%    p_n:=(n,n) \quad \text{and} \quad  q_n:=(n+\lfloor N/2 \rfloor,n- \lceil N/2\rceil).
%\end{equation}
The following two results are the main input required for the proof of Proposition~\ref{thm:11}.

\begin{prop}\label{prop:even N}
There is a constant $C>0$ such that for all $A \ge 1$ and $N\ge N_0(A)$ 
\begin{equation}
    \mathbb P \Big( T ( 0,p_{n-1})  > T ( 0,q_n ) \text{ for all } n\in \mathbb N \text{ with } |n-A N^2|\le N^{3/2} \Big) \ge 1/2+e^{-CA } \, .  
\end{equation}
\end{prop}
\begin{lem}\label{lem:M_1}
For $A\ge 1$ and $N\ge N_0(A)$, define 
\begin{equation}
    M_1:=\big| \big\{ m\ge 0 : T(0,p_m) \le AN^2/\rho _{\alpha } \big\} \big| \, .
\end{equation}
Then we have that
\begin{equation}
    \lim_{N\rightarrow \infty}\mathbb P \big( |M_1-AN^2|\le N^{3/2} \big) =0 \, .
\end{equation}
\end{lem}
Intuitively, Lemma \ref{lem:M_1} estimates the number of particles which have entered the segment until time $ AN^2/\rho _{\alpha }$, while Proposition \ref{prop:even N} will allow us to estimate the number of particles in the segment at the $n^{\text{th}}$ time a particle enters, where $|n-A N^2|\le N^{3/2}$. Let us remark that the choice of the exponent $3/2$ in Proposition \ref{prop:even N} is not optimal, but suffices for our arguments. Using  Proposition~\ref{prop:even N} and Lemma~\ref{lem:M_1}, we can easily prove Proposition~\ref{thm:11}.
\begin{proof}[Proof of Proposition~\ref{thm:11}]
We claim that it suffices to show 
\begin{equation}\label{eq:7}
    \mathbb P \big( \big| \big\{x: \eta _{t_0}(x)=0 \big\}\big| > N/2 \big) \ge 1/2+e^{-CA} \,  .
\end{equation}
Indeed, using this estimate and \eqref{eq:1/2} we obtain $ \TV{ \P(\eta_{t_0}\in \cdot \, ) - \pi} \ge e^{-CA}$.
We turn to prove \eqref{eq:7}. To this end, notice that $M_1$ is the number of particles that entered the segment by time $t_0$ and similarly,  
\begin{equation}
    M_2:=\big| \big\{ m\ge 0 : T\big( 0,(N+m,m) \big) \le t_0 \big\} \big| 
\end{equation}
is the number of particles that left the segment by time $t_0$. Thus, we have 
\begin{equation}
    \mathbb P \big( \big| \big\{x: \eta _{t_0}(x)=0 \big\}\big| > N/2 \big) \ge \mathbb P \big( M_1 -M_2<N/2 \big) \, .
\end{equation}
Next, on the intersection of the events in Proposition~\ref{prop:even N} and Lemma~\ref{lem:M_1} we have that
\begin{equation}
    T(0,q_{M_1})<T(0,p_{M_1-1})\le t_0 \,  ,
\end{equation}
where we use the definition of $M_1$ in the last inequality. By the definition of $M_2$
\begin{equation}
    M_2 \ge M_1-\lceil N/2 \rceil +1>M_1-N/2 \, .
\end{equation}
This finishes the proof of Proposition~\ref{thm:11} using Proposition~\ref{prop:even N} and Lemma~\ref{lem:M_1}.
\end{proof}

%The rest of this section is devoted to the proof of Proposition~\ref{prop:even N} and Lemma~\ref{lem:M_1}. 
For the proof of Proposition~\ref{prop:even N} and Lemma~\ref{lem:M_1}, we need the following three lemmas.

\begin{lem}\label{lem:11}
There is a constant $C>0$ such that for all $A>1$, $N\ge N_0(A)$ and $m\le A N^2$
\begin{equation}
\mathbb P \big( T ( p_0,p_m)  \ge T ( p_0,q_m )  \big) \ge 1/2+e^{-CA}. 
\end{equation}
\end{lem}

\begin{lem}\label{lem:12}
For any $\delta >0$ sufficiently small, $N\ge N_0(\delta )$ and $ N^2\le m\le N^3$ we have
\begin{equation}
\mathbb P \Big( \big|T ( p_0,p_m)  - T ( p_0,q_m ) \big) \big| \le \delta ^2 N  \Big) \le \delta .
\end{equation}
\end{lem}

\begin{lem}\label{lem:13}
For any $\delta >0$ sufficiently small, $N\ge N_0(\delta )$ and $N^2 \le m\le N^3$ we have 
\begin{equation}
    \mathbb P \big( T(p_0,p_m)\ge T(p_0,q_m)+\delta N \text{ and } \exists m< n\le m+\delta ^2 N^2,\  T(p_0,p_{n-1})\le T(p_0,q_{n}) \big) \le \delta .
\end{equation}
\end{lem}

%We can now prove Proposition~\ref{prop:even N}

\begin{proof}[Proof of Proposition~\ref{prop:even N} using Lemmas \ref{lem:11} to  \ref{lem:13}]

 Let $A>1$ be sufficiently large and set
 \begin{equation}
    m=\lfloor AN^2-N^{3/2} \rfloor -1 \, . 
 \end{equation}
 for all $N\ge N_0$, with some constant $N_0=N_0(A)$. By Lemma~\ref{lem:11} we have that 
\begin{equation}
\mathbb P \big( T ( p_0,p_m)  \ge T ( p_0,q_m )  \big) \ge 1/2+e^{-C_1A}
\end{equation}
for some $C_1>0$. Next, by Lemma~\ref{lem:12} with $\delta =e^{-2C_1A}$ we obtain that 
\begin{equation}
\mathbb P \big( T ( p_0,p_m)  \ge T ( p_0,q_m )+e^{-C_2A}N  \big) \ge 1/2+e^{-C_2A}
\end{equation}
for some $C_2>0$. Finally, by Lemma~\ref{lem:13} with $\delta =e^{-2C_2A} $ we have 
\begin{equation}
    \mathbb P \big(\forall m< n\le m+e^{-C_3A} N^2,\  T(p_0,p_{n-1})\ge T(p_0,q_{n}) \big) \ge 1/2+e^{-C_3A}
\end{equation}
for some $C_3>0$. This finishes the proof of the proposition.
\end{proof}

\subsection{Proof of the auxiliary lemmas}

It remains to show Lemma~\ref{lem:M_1} to Lemma~\ref{lem:13}. For the proofs, we require some notation. Fix some $N\in \N$. Similar to \eqref{def:RestrictedLPTs}, we define \textbf{restricted passage time} $\bar{T}(p_i,p_j)$ with respect to the environment $(\omega_v)_{v\in \mathcal{S}_N}$ from Section \ref{sec:TASEPLPPcorres} for all $i\le j$ by
\begin{equation}
\bar{T}(p_i,p_j) := \max_{\pi \in \Pi_{B^1}^{p_i,p_j}} \sum_{z \in \pi \setminus \{p_j\}} \omega_z \, ,
\end{equation}
where we restrict the up-right path to stay in the set $B^1 = \{ (v_1,v_2) \in \Z^2 : v_2\le v_1 \le v_2+\log ^9 N \}$. The \textbf{restricted geodesic} $\bar{\gamma}(p_i,p_j)$ is the unique path achieving the maximum in the definition of $\bar{T}(p_i,p_j)$. Similarly, define $\bar{T}(q_i,q_j)$ for $i\le j$ to be the maximal passage time of an up-right path restricted to stay in $B^2 = \{ (v_1,v_2) \in \Z^2 : v_2+N-\log ^9 N \le v_1 \le v_2+N  \}$, and the restricted geodesic $\bar{\gamma }(q_i,q_j)$ to be the path achieving the maximum. Intuitively, by Proposition~\ref{cor:1}, the restricted passage times behave with high probability as the last passage times in the half-quadrant. Next, for $n\in \mathbb Z$, and $r,d>0$, we define the event
\begin{equation}\label{eq:def of B}
\begin{split}
    \mathcal B(n,r,d) :&= \Big\{ \big| \bar{T}   ( p_i,p_j ) -  (j-i)/\rho _\alpha  \big|\le dN \text{ for all }n\le i\le j \le n+rN^2 \Big\} \\
    &\ \cap \Big\{  \big| \bar{T}   ( q_i,q_j ) -  (j-i)/\rho _\alpha \big| \le dN  \text{ for all } n\le i\le j \le n+rN^2 \Big\}.
\end{split}
\end{equation}
We show in the sequel that on the above event, when $d\le D:=(\rho _{\alpha }^{-1}-4)/9$, it is unlikely for geodesics to jump from one boundary of the slab to the other. Note that $\mathcal B(n,r,d)$ is the intersection of two independent events that have the same probability. By Proposition~\ref{cor:1},  this probability can be bounded using the corresponding event with the half-quadrant last passage times. It follows by Lemma~\ref{lem:brownian} that when $r\ge d^2$ and $N\ge N_0(r,d)$,  we have $\mathbb P \big( \mathcal B (n,r,d) \big) \ge e^{-Cr/d^2}$ for some constant $C>0$. Similarly, by Lemma~\ref{lem:brownian 2}, when $r< d^2$ and $N\ge N_0(r,d)$, we have $\mathbb P \big( \mathcal B (n,r,d) \big)\ge 1-e^{-cd^2/r } $ for some constant $c>0$. 
Next, for two points $v,u\in \mathcal{S}_N$ with $v \succeq u$, we define $\tilde{T}(u,v)$ to be the maximal passage time of an up-right path that is restricted to stay in  $\mathcal{S}_N \setminus \partial(\mathcal{S}_N)$, except for possibly its endpoints. Observe that the last passage times $\tilde{T}(u,v)-\omega _{u}$ are stochastically dominated by the corresponding last passage times $T(u,v)$ on the full space. We define the event
\begin{equation}
    \mathcal{C} _1  = \big\{  \tilde{T}   (p_i,p_j) <  \bar{T}   (p_i,p_j ) \text{ and } \tilde{T}(q_i,q_j) < \bar{T}(q_i,q_j) \text{ for all }  i,j \in[N^3] \text{ with } i-j \geq \frac{1}{2}\log^{9}(N)\big\}
\end{equation}
 and it follows from Theorem~\ref{thm:Ledoux} to bound $\tilde{T}   (p_i,p_j)$, and Lemma \ref{lem:CombinationLemma} to control $\bar{T}   (p_i,p_j )$ that $\mathbb P (\mathcal{C} _1 )\ge 1-e^{-c_1\log ^2N}$ for some $c_1>0$, and all $N$ sufficiently large. Moreover, let
\begin{equation}
    \mathcal{C} _2 := \Big\{ \max \big( \tilde{T}(p_i,q_j), \tilde{T}(q_i,p_j) \big) \le (j-i)/\rho _\alpha -4DN  \text{ for all } 0\le i\le j \le N^3 \Big\} \,  .
\end{equation}
We obtain from Theorem~\ref{thm:Ledoux} that $\mathbb P (\mathcal{C} _2 )\ge 1-e^{-c_2\log ^2 N}$ for some $c_2>0$. The next claim shows that geodesics do not traverse between $\partial_1(\mathcal{S}_N)$ and $\partial_2(\mathcal{S}_N)$ on $\mathcal B (n,r,d) \cap \mathcal{C}_1  \cap \mathcal{C} _2$.

\begin{claim}\label{claim:induction}
Let $n\in \mathbb \N$ and $r>0$ be such that $n+rN^2\le N^3$. When $d\le D$, then on the intersection $\mathcal B (n,r,d) \cap \mathcal{C}_1  \cap \mathcal{C} _2$, we have for all $n\le i\le j\le n+rN^2$ 
\begin{equation}\label{eq:induction1}
    \gamma (p_i,p_j) =\bar{\gamma }(p_i,p_j) \quad \gamma (q_i,q_j)=\bar{\gamma }(q_i,q_j)
\end{equation} as well as that
\begin{equation}\label{eq:induction2}
  \max(T(p_i,q_j),T(q_i,p_j)) \le (j-i)/\rho _\alpha -2DN \, .
\end{equation}
\end{claim}

\begin{proof}
Suppose the event $\mathcal B (n,r,d) \cap \mathcal{C}_1  \cap \mathcal{C} _2$ holds. We prove the claim by induction on $j-i$. Let $n\le i\le j \le n+r N^2$, and suppose that \eqref{eq:induction1} and \eqref{eq:induction2} hold for all $n\le i'\le j'\le n+r N^2$ such that $j'-i'<j-i$. We first argue that $\gamma (p_i,p_j)=\bar{\gamma }(p_i,p_j)$. Note that this is immediate when $i-j\leq \log^{9}(N)/2$. Else, we claim that for all $i<k<j$,  we have that $p_k\notin \gamma (p_i,p_j)$. To see this, note that for every  $x=(x_1,x_2)\in \gamma (p_i,p_j)$ such that $ x_2+\log ^9 N < x_1 \le x_2N$, we would either have $x\in \gamma (p_i,p_k)$ or $x\in \gamma (p_k,p_j)$,  contradicting the induction hypothesis. Next observe that $\gamma (p_i,p_j)$ must intersect the lower boundary $\partial_2(\mathcal{S}_N)$, as we obtain a contradiction to the event $\mathcal{C}_1$, otherwise. Hence, there exists $i<k<j$ such that $q_k\in \gamma (p_i,p_j)$. By the induction hypothesis on the event $\mathcal B (n,r,d)$ we have 
\begin{equation}
\begin{split}
    T(p_i,p_j) \le T(p_i,q_k)+T(q_k,p_j) &\le (k-i)/\rho _\alpha -2DN +(j-k)/\rho _\alpha  -2DN \\
    &< (j-i)/\rho _\alpha -DN\le \bar{T}(p_i,p_j) \le T(p_i,p_j) \, .
\end{split}
\end{equation}
This is a contradiction to $\mathcal{C}_2$. In total, we conclude that $\gamma (p_i,p_j) =\bar{\gamma }(p_i,p_j)$. The proof that $\gamma (q_i,q_j)=\bar{\gamma }(q_i,q_j)$ goes analogously. In order to obtain the bound on $T(p_i,q_j)$ in \eqref{eq:induction2}, let $k< j$ be such that $p_k$ is the last intersection of the geodesic $\gamma (p_i,q_j)$ with the left boundary of the slab. Similarly, let $q_m$ be the first intersection with the right boundary after $p_k$. Note that possibly $p_k=p_i$ or $q_m=q_j$; see Figure~\ref{fig:DorsFigure}. By the induction hypothesis, we have that $\gamma (p_i,p_k)=\bar{\gamma }(p_i,p_k)$ and $\gamma (q_m,q_j)=\bar{\gamma }(q_q,q_t)$. Thus, on the event $ \mathcal{C} _2\cap \mathcal B (n,r,d)$ 
\begin{align*}
    T(p_i,q_j) &\le T(p_i,p_k)+T(p_k,q_m)+T(q_m,q_j) = \bar{T}(p_i,p_k)+\tilde{T}(p_k,q_m)+\bar{T}(q_m,q_j) \\
    \le&  (k-i)\rho_\alpha^{-1} +DN + (m-k)/\rho _\alpha  -4DN+ (j-m)\rho _\alpha^{-1}  +DN \le (j-i)\rho_\alpha^{-1}  -2DN \, .
\end{align*} The bound on $T(q_i,p_j)$ goes verbatim, and thus we obtain \eqref{eq:induction2}.
\end{proof}

%With these insights, we can now prove Lemma~\ref{lem:11}.

\begin{proof}[Proof of Lemma~\ref{lem:11}]
Suppose first that $N$ is even and let $m\le AN^2$. Define the events
\begin{equation}
    \mathcal A _1 := \big\{ T ( p_0,q_m )  \ge T ( p_0,p_m ) \big\} \quad \text{and} \quad \mathcal A _2 := \big\{ T ( q_0 ,p_m )  \ge T ( q_0,q_m ) \big\} \, .
\end{equation}\begin{figure}
    \centering
\begin{tikzpicture}[scale=.45]

\draw[densely dotted, line width=1pt] (0,0) -- (2.5,-2.5);

\draw[densely dotted, line width=1pt] (0+9,0+9) -- (2.5+9,-2.5+9);

%\draw[densely dotted, line width=1pt] (1.525,1.525) -- (1.525+2.5,1.525-2.5);
%\draw[densely dotted, line width=1pt] (9-1.525,9-1.525) -- (9-1.525+2.5,9-1.525-2.5);

   \draw[line width=1.2pt] (0,0) -- (9,9);
   \draw[line width=1.2pt] (2.5,-2.5) -- (11.5,6.5);
   
%(5.6-2.2,3.1-1.6+0.2)..(5.6,3.1)..(5.68,3.18) ..(5.6+0.7,3.1+0.9)

 \draw[red, line width =2 pt] (0,0) to[curve through={(0.39,0.39)..(0.4,0.4).. ..(1.5,0.3)..(2.75,0.5).. (5.6-2.2,3.1-1.6+0.2)..(5.6,3.1)..(5.68,3.18) ..(5.6+0.7,3.1+0.9)..(5.6+2.6,3.1+1.6) .. (2+9,-2.6+9)}] (2.5+9,-2.5+9);
%\draw[red, line width =2 pt] (9,9) to[curve through={(9+0.2,9-0.7) ..(9-0.1,9-1.5)..(9-0.2,9-2.45) }] (9-0.4,9-2.65);

\draw[blue, line width =2 pt] (2.5,-2.5) to[curve through={(2.6,-1.8) ..(3.5,-1.38)..(5.5,2.5)..(7,3)..(8,6)}] (9,9);

\node[scale=1] (x1) at (8.4-0.15,4.2-0.15) {$z_{\ast}$};
	
%	\node[scale=1] (x1) at (-0.7,0){$p_{0}$} ;	

	\node[scale=1] (x1) at (0.4,1){$p_{i}$} ;

		\node[scale=1] (x1) at (-0.7,0){$p_0$} ;
		\node[scale=1] (x2) at (2.5+0.7,-2.5){$q_0$} ;
		\node[scale=1] (x3) at (9+0.8,9){$p_m$} ;
		\node[scale=1] (x4) at (9+2.5+0.8,9-2.5){$q_m$} ;
		
	%	\draw[line width =1 pt] (-0.3,0.3) -- (-0.5,0.5)   --  (-0.5+1.525,0.5+1.525) --  (-0.3+1.525,0.3+1.525);

%		\draw[line width =1 pt] (9+0.3+2.5,9-0.3-2.5) -- (9+0.5+2.5,9-0.5-2.5)   --  (0.5+9+2.5-1.525,-0.5+9-1.525-2.5) --  (0.3+9+2.5-1.525,-0.3+9-1.525-2.5);
		
	%	\node[scale=1] (x3) at (-0.8+0.75-0.9,0.8+0.75){$C^{-1}N^{3/2}$} ;
		
		%	\node[scale=1] (x3) at (12.5,5){$C^{-1}N^{3/2}$} ;

	%		\node[scale=1] (x3) at (11.5,1){$|a_4-a_1|_1=CN^{3/2}$} ;

%\node[scale=1] (x5) at (9-1.525-0.15-3.5,9-1.525-0.15){$a_4-(C^{-1} N^{3/2}, C^{-1} N^{3/2})$};
%\node[scale=1] (x5) at (1.525-0.15-3.5,1.525-0.15){$a_1+(C^{-1} N^{3/2}, C^{-1} N^{3/2})$};

	\filldraw [fill=black] (7.88-0.15,4.62-0.15) rectangle (7.88+0.15,4.62+0.15);  

%	\filldraw [fill=black] (2.5-0.15+1.4,-2.5-0.15+1.4) rectangle (2.5+0.15+1.4,-2.5+0.15+1.4);  

	\filldraw [fill=black] (2.5-0.15,-2.5-0.15) rectangle (2.5+0.15,-2.5+0.15);   

 	\filldraw [fill=black] (2.5+9-0.15,-2.5+9-0.15) rectangle (2.5+9+0.15,-2.5+9+0.15);     
 
	\filldraw [fill=black] (-0.15+0.395,-0.15+0.395) rectangle (0.15+0.395,0.15+0.395);

	\filldraw [fill=black] (-0.15,-0.15) rectangle (0.15,0.15);   

 	\filldraw [fill=black] (9-0.15,9-0.15) rectangle (9+0.15,9+0.15);     	
 	
	\end{tikzpicture}	
    \caption{Crossing of paths at $z_{\ast}$ in Lemma~\ref{lem:11}, and construction of the point $p_i$ in the  path decomposition in the proof of Lemma \ref{lem:12}.}
    \label{fig:DorsFigure}
\end{figure}
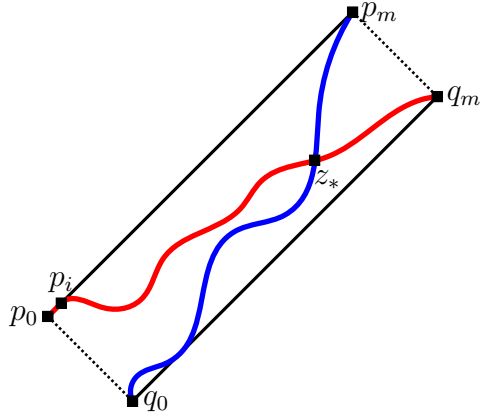 We claim the events $\mathcal A _1$ and $\mathcal A _2$ are almost surely disjoint. To see this, let $z_{\ast}$ be the first intersection point of the geodesics $\gamma (p_0,q_m)$ and $\gamma (q_0,p_m)$; see Figure~\ref{fig:DorsFigure}. On the event $\mathcal A _1$, the path starting from $p_0$, going along $\gamma (p_0,q_m)$ up to $z_{\ast}$, and then along $\gamma (q_0,p_m)$ to $p_m$ has a passage time smaller than the passage time of $\gamma (p_0,q_m)$. It follows that on $\mathcal A _1$ we have $T(z_{\ast},p_m)\le T(z_{\ast},q_m)$. By symmetry, on $\mathcal A _2$ we have $T(z_{\ast},q_m)\le T(z_{\ast},p_m)$. Thus, $\mathbb P (\mathcal A _1 \cap \mathcal A _2)=0$.
 By Claim~\ref{claim:induction}, we have that $\mathcal B(n,A,D) \cap \mathcal{C} _1 \cap \mathcal{C} _2 \subseteq \mathcal A _1^c \cap \mathcal A _2 ^c$. Thus, using the estimates after \eqref{eq:def of B} we obtain $\mathbb P( \mathcal A _1 ^c \cap \mathcal A _2 ^c ) \ge \mathbb P \big( \mathcal B(n,A,D) \cap \mathcal{C} _1 \cap \mathcal{C} _2 \big) \ge e^{-CA }$. Finally, by symmetry we have
\begin{equation}
    1 = \mathbb P ( \mathcal A _1)+\mathbb P (\mathcal A _2 )+\mathbb P (\mathcal A _1 ^c \cap \mathcal A _2^c) \ge 2\mathbb P (\mathcal A _1 ) +e^{-CA } \, .
\end{equation}
and therefore $\mathbb P (\mathcal A _1) \le 1/2 -e^{-CA}$. This finishes the proof of the lemma when $N$ is even. For odd $N$, the only part that fails in the above proof is the argument that uses the reflection symmetry of the slab and the fact that $\mathbb P (\mathcal A _1)=\mathbb P (\mathcal A _2)$. This equality holds only when $N$ is even. To overcome this issue we use the monotonicity of passage times and an analogous symmetry that holds when $N$ is odd. More precisely, let us define the events
\begin{align}
    \mathcal A _3 &:= \big\{ T ( p_0,q_m )  \ge T ( p_0,p_m ) \big\} \\ \mathcal A _4 &:= \big\{ T ( p_0,q_{m+1} )  \ge T ( p_0,p_m ) \big\} \\
     \mathcal A _5 &:= \big\{ T ( q_1 ,p_m )  \ge T ( q_1,q_{m+1} ) \big\} \, .
\end{align}
By the same arguments as above, we have that $\mathbb P (\mathcal A _4 \cap \mathcal A _5 )=0$ and $\mathbb P (\mathcal A _4 ^c \cap \mathcal A _5^c )\ge e^{-CA}$. Moreover, $T(p_0,q_{m+1})>T ( p_0,q_m )$, and therefore $\mathbb P (\mathcal A _4) \ge \mathbb P (\mathcal A _3)$. Finally, the map
\begin{equation}
    (a,b)\mapsto \big( b+\lceil N/2 \rceil ,a-\lfloor N/2 \rfloor \big)
\end{equation}
is an automorphism of the slab that sends $p_0$ to $q_1$, $q_m$ to $p_m$ and $p_m$ to $q_{m+1}$. Hence, we get $\mathbb P (\mathcal A _3 )=\mathbb P (\mathcal A _5)$ and obtain
\begin{equation}
    1\ge \mathbb P ( \mathcal A _4)+\mathbb P (\mathcal A _5 )+\mathbb P (\mathcal A _4 ^c \cap \mathcal A _5^c) \ge \mathbb P ( \mathcal A _3)+\mathbb P (\mathcal A _5 ) +e^{-CA }=2\mathbb P ( \mathcal A _3)+e^{-CA } \, .
\end{equation}
This finishes the proof of the lemma.
%{\color{red} we need a general notation $\gamma (v,w)$ is the geodesic from $v$ to $w$.}
\end{proof}

We turn to prove Lemma~\ref{lem:12}. To this end, we define the event
\begin{equation}
    \mathcal{C} _3 := \big\{ \tilde{T}(p_i,q_j)<T(p_i,q_j) \text{ and } \tilde{T}(q_i,p_j)<T(q_i,p_j)  \text{ for all } i,j \in [N^3] \text{ with } j\ge i+N^{3/2} \big\} \, .
\end{equation}
We claim that $\mathbb P (\mathcal{C} _3) \ge 1-e^{-c_3\log ^2N}$ for some constant $c_3>0$, and all $N$ sufficiently large. Indeed, by Theorem~\ref{thm:Ledoux}  $\tilde{T}(p_i,q_j) \le (4+D)(j-i)$ holds with probability at least $1-e^{-c_3 N}$. Lemma \ref{lem:CombinationLemma} now yields that $T(p_i,q_j) \ge \bar{T}(p_i,q_j-(N,0)) \ge (\rho _\alpha ^{-1}-D)(j-i)$ with probability at least $1-e^{-c_3\log ^2 N}/2$ for $N$ large enough. Finally, define
\begin{equation}
\begin{split}
    \mathcal{C} _4 := &\big\{\big|  \bar{T}(p_i,p_j)- \bar{T}(p_i,p_k)-\bar{T}(p_k,p_j)\big|\le \log ^{12}N \text{ for all }  0\le i\le k \le  j\le N^3\big\}\\
    \cap &\big\{ \big|  \bar{T}(q_i,q_j)- \bar{T}(q_i,q_k)-\bar{T}(q_k,q_j)\big|\le \log ^{12}N \text{ for all } 0\le i\le k \le  j\le N^3\big\}.
\end{split}
\end{equation}
By Corollary~\ref{claim:triangle} and Proposition~\ref{cor:1} we have that $\mathbb P (\mathcal{C} _4)\ge 1-e^{-c_4\log ^2 N}$ for some constant $c_4>0$. Let  $\mathcal{C} :=\mathcal{C} _1 \cap \mathcal{C} _2 \cap \mathcal{C} _3 \cap \mathcal{C} _4 $ and note that $\mathbb P (\mathcal{C} )\ge 1-e^{-c\log ^2N}$ for some suitable constant $c>0$ and all $N$ sufficiently large. We can now prove Lemma~\ref{lem:12}.

\begin{proof}[Proof of Lemma~\ref{lem:12}]
Let $\delta >0$ be sufficiently small and set $n=\lfloor m-\delta N^2 \rfloor$. Define 
\begin{equation}
    T_1:=T(p_0,p_n)+ \bar{T}(p_n,p_m) ,\quad T_2:=T(p_0,q_n)+ \bar{T}(q_n,q_m) 
\end{equation}
and the event $\mathcal A := \big\{ |T_1-T_2|\ge 2\delta ^2N \big\}$. First, we claim that $\mathbb P ( \mathcal A ) \ge 1-C\delta ^{3/2} $ for some constant $C>0$. To see this, let $\mathcal F $ be the sigma algebra generated by all the weights $\omega _v$ for $v$  outside of the box 
\begin{equation}
\big\{ (x_1,x_2)\in \mathcal{S}_N : n\le x_2 \le m, \  x_2 \le  x_1 \le x_2 +\log ^9 N    \big\} \, .
\end{equation}
Note that $T_2$ as well as $T(0,p_n) $ are measurable in $\mathcal F $ while $\bar{T}(p_n,p_m)$ is independent of $\mathcal F $. Thus, almost surely for some constant $C>0$
\begin{equation}\label{eq:Exptaking}
    \mathbb P \big( \mathcal A^c \mid  \mathcal F  \big) = \mathbb P \big( \bar{T}(p_n,p_m) \in [T_3-2\delta ^2N, T_3+2\delta ^2N]   \mid \mathcal F  \big) \le C\delta^{3/2} \, , 
\end{equation}
where $T_3:=T_2-T(p_0,p_n)$ is $\mathcal{F}$-measurable. Note that the last inequality follows since by Theorem~\ref{thm:corwin} and Proposition~\ref{cor:1}, the random variable $\bar{T}(p_n,p_m)$ is asymptotically normal with variance $ \sigma \delta N^2$ for some $\sigma>0$. Taking expectations on both sides  of \eqref{eq:Exptaking}, we get $\mathbb P (\mathcal A )\ge 1-C\delta ^{3/2}$.
Next, set $k=\lceil m-2\delta N^2 \rceil $. It suffices to prove that 
\begin{equation}\label{eq:far}
    \mathcal{C} \cap \mathcal A \cap \mathcal B(k,2\delta ,D) \subseteq \big\{ |T(p_0,p_m)-T(p_0,q_m)|\ge \delta ^2 N \big\} 
\end{equation}
for all $\delta>0$ sufficiently small. Indeed, using the estimates following \eqref{eq:def of B} we obtain that $ \mathbb P \big( \mathcal B (k,2\delta ,D) \big) \ge 1- e^{-C/\delta } $ for some constant $C>0$. Therefore, we get  $\mathbb P \big( |T(p_0,p_m)-T(p_0,q_m)|\ge \delta ^2 \big) \ge 1-2C\delta ^{3/2}\ge 1-\delta $,  provided that $\delta $ is sufficiently small. \\

We turn to prove \eqref{eq:far}. Let $p_i$ be the last intersection point of the geodesic $\gamma (p_0,q_m)$ with the upper boundary $\partial_1(\mathcal{S}_N)$ of the slab, and let $q_j$ be the first intersection of $\gamma (p_0,q_m)$ with the boundary $\partial_2(\mathcal{S}_N)$ after $p_i$; see also Figure \ref{fig:DorsFigure}. 
Similarly, in the case where $\gamma (p_0,p_m)$ intersects the lower boundary of the slab, we let $q_{i'}$ be the last intersection with the upper boundary $\partial_1(\mathcal{S}_N)$, and let $p_{j'}$ be the first intersection with the lower boundary after $q_{i'}$. 
In the case where $\gamma (p_0,p_m)$ is disjoint from the lower boundary, we set $i'=j'=0$. By Claim~\ref{claim:induction}, on the event $\mathcal{C} \cap \mathcal B(k,2\delta ,D) \cap \{ i>k\}$, we have that 
\begin{equation}
    T(p_i,q_m)\le (m-i)/\rho _\alpha -2DN \le \bar{T}(p_i,p_m)-DN\le  T(p_i,p_m)-DN \, .
\end{equation}
Thus, using that both $\gamma (p_0,p_m)$ and $\gamma (p_0,q_m)$ pass through the point $p_i$, we obtain that $T(p_0,q_m) \le T(p_0,p_m)-DN$ and in particular $|T(p_0,q_m) - T(p_0,p_m)|\ge \delta ^2 N $ as long as $\delta <\sqrt{D}$. Using the same arguments on the event $\mathcal{C} \cap \mathcal B(k,2\delta ,D) \cap \{ i'>k\}$ we have that $|T(0,q_m) - T(p_0,p_m)|\ge \delta ^2 N $. Finally, suppose that the event $\mathcal{C} \cap \mathcal A \cap \mathcal B(k,2\delta ,D) \cap \{ i,i'\le k \}$ holds. In this case, by the definition of $\mathcal{C} _3$ we have that $j'\le n$, and therefore both $\gamma (p_0,p_m)$ and $\gamma (p_0,p_n)$ contain $p_{j'}$. Thus, we have 
\begin{equation}
    |T(p_0,p_m)-T_1| = \big| T(p_{j'},p_m)-T(p_{j'},p_n)-\bar{T}(p_n,p_m) \big| \, .
\end{equation}
Moreover, note that $\gamma (p_{j'},p_m)$ does not intersect the right boundary of the slab, and therefore on $\mathcal{C} _1$, it is contained in $\{ x_2\le x_1 \le x_2+\log ^9N\}$. It follows that $\gamma (p_{j'},p_n)$ is also contained in $\{ x_2\le x_1 \le x_2+\log ^9N\}$, and therefore $T(p_{j'},p_m)=\bar{T}(p_{j'},p_m)$ and $T(p_{j'},p_n)=\bar{T}(p_{j'},p_n)$. Thus, using the definition of $\mathcal{C} _4$ we obtain for all $N$ sufficiently large
\begin{equation}
    |T(p_0,p_m)-T_1|=\big| \bar{T}(p_{j'},p_m)-\bar{T}(p_{j'},p_n)-\bar{T}(p_n,p_m) \big|\le \log ^{13}N \, .
\end{equation}
By the same arguments, we have that $|T(p_0,q_m)-T_2|\le \log ^{13}N$. Thus, by the definition of $\mathcal A$, we have $|T(p_0,p_m)-T(p_0,q_m)| \ge \delta ^2 N$. This finishes the proof of the lemma.
\end{proof}

We turn to prove Lemma~\ref{lem:13}.

\begin{proof}[Proof of Lemma~\ref{lem:13}]
For fixed $m\in [N^2,N^3]$, we set $k=\lceil m- \delta ^3N^2 \rceil$ and consider the events
\begin{align}
    \mathcal D_1 &:=\big\{ T(p_0,p_m)\ge T(p_0,q_m)+\delta N  \big\} \\ \mathcal D_2 &:= \big\{  \ T(p_0,p_{n-1}) \ge T(p_0,q_n) \text{ for all } m\le n \le m+\delta ^3N^2 \big\} \, .
\end{align}
We claim that 
\begin{equation}\label{eq:7.7}
    \mathcal D_1 \cap \mathcal{C} \cap \mathcal B(k,2\delta ^3, \delta /4) \subseteq \mathcal D_2 \, .
\end{equation}
Indeed, suppose that the event on the left hand side of \eqref{eq:7.7} holds and let $m\le n \le n+\delta ^2 N^2$. Let $p_i$ be the last intersection of the geodesic $\gamma(p_0,q_n)$ with the left boundary and let $q_j$ be the first intersection with the right boundary after $p_i$. Suppose first that $\{j\le m\}$ holds. In this case, we have that $q_j\in \gamma (p_0,q_m)$,  and using the definition of $\mathcal{C} _4$ and $\mathcal{C} _1$, we get 
\begin{equation}
\begin{split}
    &|T(p_0,q_m)-T(p_0,q_n)-(n-m)/\rho _{\alpha }|= |T(q_j,q_m)-T(q_j,q_n)-(n-m)/\rho _\alpha | \\
    &=|\bar{T}(q_j,q_m)-\bar{T}(q_j,q_n)-(n-m)/\rho _\alpha | \le \log ^{13}N + |\bar{T}(q_m,q_n)-(n-m)/\rho _\alpha | \le \delta N/3 \, 
\end{split}
\end{equation}
for all $N$ sufficiently large. Therefore, in this case, 
\begin{equation}
\begin{split}
    T(p_0,q_n) \le T(p_0,q_m) +(n-m)/\rho _{\alpha } +\delta N/3 \le T(p_0,p_m) +(n-m)/\rho _{\alpha } -2\delta N/3 \\
    \le T(p_0,p_{n-1}) -\bar{T}(p_m,p_{n-1})+(n-m)/\rho _{\alpha }-2\delta N/3 \le T(p_0,p_{n-1})+\delta N/3 \, .
\end{split}
\end{equation}
Next, suppose that $\{j\ge m\}$. In this case, by the definition of $\mathcal{C} _3$, we have that $i\ge k$. Using the same arguments as in the proof of Lemma~\ref{lem:12}, we get
\begin{equation}
    T(p_0,p_{n-1})-T(p_0,q_n)=T(p_i,p_{n-1})-T(p_i,q_{n})\ge DN \, . 
\end{equation}
This establishes the relation in \eqref{eq:7.7}. Bounding the probability of the events on the left hand side of \eqref{eq:7.7} by the previous lemma finishes the proof. 
\end{proof}

%It remains to show Lemma~\ref{lem:M_1}.

\begin{proof}[Proof of Lemma~\ref{lem:M_1}]
Let $m=\lfloor AN^2 +N^{3/2}\rfloor $. We have that with probability at least $1-e^{-c\log ^2N}$ for some constant $c>0$ and all $N$ sufficiently large
\begin{equation}
    T(p_0,p_m)\ge \bar{T}(p_0,p_m)\ge m/\rho _{\alpha }-N\log ^4N > AN^2/\rho _\alpha =t_0 \, ,
\end{equation}
where the second inequality holds by Lemma \ref{lem:CombinationLemma}. Thus, by the definition of $M_1$, we have that $\lim_{N\rightarrow \infty} P (M_1\le m) = 1$. We now show the corresponding lower bound on $M_1$. Let $k=\lceil  AN^2 -N^{3/2} \rceil $ and define the event
\begin{equation}\label{def:Edef}
    \mathcal E :=  \mathcal{C} \cap \bigcap _{n\le l\le k} \big\{ \bar{T}(p_{n},p_{l}),\bar{T}(q_{n},q_{l}) \le (n-l)/\rho _\alpha -N\log ^4N \big\}  \cap \bigcap _{n=0}^{k} \mathcal B\big(n, \log ^{-9}N, D \big) \, .
\end{equation}
By Lemma \ref{lem:CombinationLemma} and a union bound on the complements of the events at the right-hand side of \eqref{def:Edef}, we get that $\mathbb P (\mathcal E ) \ge 1-e^{-c_1\log ^2 N}$ for some $c_1>0$, and all $N$ sufficiently large. Thus, it suffices to prove that on the event $\mathcal E $, we have $T(p_0,p_k)<t_0$. To this end, we decompose the geodesic $\gamma (p_0,p_k)$ in the following way, see also Lemma \ref{lem:NoCrossing}. Let 
\begin{equation}
    0=i_1\le i_1'< j_1\le j_1'< i_2 \le i_2'< j_2\le \cdots < i_s\le i_s'< j_s\le j_s'=k \, ,
\end{equation}
such that $p_{i_1'}$ is the last intersection of $\gamma (p_0,p_k)$ with the left boundary before reaching the right boundary for the first time, $q_{j_1}$ is the first intersection of $\gamma (p_0,p_k)$ with the right boundary after $p_{i_1'}$, $q_{j_1'}$ is the last intersection point of $\gamma (p_0,p_k)$ with the right boundary before returning to right boundary after $i_1$, and so on. Using the definition of $\mathcal{C} _1 $ we have 
\begin{equation}
\begin{split}
    T(p_0,p_k)&\le \sum _{l=1}^s \bar{T}(p_{i_l},p_{i_l'})+\tilde{T}(p_{i_l'},p_{j_l})+\bar{T}(p_{j_l},p_{j_l'})+\tilde{T}(p_{j_l'},p_{i_{l+1}}) \\
    &\le \sum _{l=1}^s \Big( (i_i'-i_l)/\rho _\alpha +(j_l-i_l')/\rho _\alpha +(j_l-j_l')/\rho _\alpha +(i_{l+1}-j_l')/\rho _\alpha  +2N\log ^4N\Big)\\
    &\le k/\rho _\alpha +sN\log ^4N \, ,
\end{split}
\end{equation}
where in the second inequality, we use the definition of $\mathcal{C} _2$. Finally, note that by Claim~\ref{claim:induction}, on the event $\mathcal E $, we have that $i_{l+1}-i_l'\ge N^2/\log ^9N$. Therefore, we see that $s\le \log ^{10}N$, and obtain $T(p_0,p_k)\le k/\rho _\alpha +N\log ^{14}N<t_0$, allowing us to conclude.
\end{proof}

\textbf{Acknowledgment.} We are deeply grateful to Nicos Georgiou for helpful discussions, and insights on last passage percolation. Further, we thank Sa'ar Zehavi, Ron Peled and Daniel Hadas for fruitful discussion, and the anonymous referees for various comments which helped to significantly improve the paper. Lastly, we thank Lingfu Zhang and Allan Sly for (independently) pointing out the proof of Lemma~\ref{claim:lingfu allan}.

\bibliography{HighLowDensityFix}
\bibliographystyle{abbrv}

\vspace{-0.2cm}

\appendix

\section{Proof of Lemma \ref{lem:variance bound}}

In order to prove Lemma \ref{lem:variance bound} on the variance of last passage times in the half-quadrant, we use the following general result.  The presented proof was given to us independently by Allan Sly and Lingfu Zhang.

\begin{lem}\label{claim:lingfu allan} Recall that $p_i=(i,i)$ for all $i\in \Z$. 
   Let $Q'(u,v)$ be the passage time of last passage percolation on $\mathbb Z ^2$ with i.i.d. weights $y_v$ such that $\mathbb E[y_v^5]<\infty $ for all $v\in \Z^2$. Then there exists some constant $C>0$ such that for all $n\in \N$
   \begin{equation}
       \mathbb E \big[ Q'\big( p_0,p_n \big) \big] \le Cn \, .
   \end{equation}
\end{lem}

\begin{proof}
We have that $y_v\le \sum _{k=0}^{\infty }2^{k+1} y_v(k)$ where 
\begin{equation}
    y_v(k):=\mathds 1 \big\{ y_v\in [2^k-1,2^{k+1}] \big\} \, .
\end{equation}
Therefore $Q'(u,v)\le \sum _{k=0}^{\infty }2^{k+1} Q'_k(u,v)$, where $Q'_k(u,v)$ is the passage time with the weights $y_v(k)$. Thus, it suffices to show that
\begin{equation}\label{eq:373}
       \mathbb E \big[ Q'_k\big( p_0,p_n \big) \big] \le C2^{-2k}n \, .
   \end{equation}
To this end, we define for $v\in \mathbb N ^2$  
\begin{equation}
    x_v(k):=\sum_{u\in 2^{2k}v+[0,2^{2k}]^2} y_u(k)
\end{equation}
and consider the last passage time $Q''_k(u,v)$ with the new weights $x_v(k)$. We clearly have that $Q'_k\big(p_0,p_n\big) \le Q''_k\big(p_0,(\lceil 2^{-2k}n\rceil ,\lceil 2^{-2k}n\rceil)\big)$ and moreover, for any integer $m \ge 1$
\begin{equation}
    \mathbb P \big( x_v(k)\ge m \big) \le 2^{4mk} \cdot \mathbb P \big( y_v(k) =1 \big)^{ m } \le 2^{4mk} \cdot \mathbb P \big( y_v \ge 2^{k}-1 \big)^{ m } \le (C2^{-k})^m \, .
\end{equation}
The last inequality follows from Markov's inequality and since $\mathbb E [y_v^5]<\infty$. Note that $x_v(k)$ is an integer random variable with $x_v(k)\preccurlyeq C w$. Here, $w$ is an independent Exponential-$1$-random variable and $\preccurlyeq$ denotes stochastic domination. Thus, we have that 
\begin{equation}
    \mathbb E \big[Q''_k\big(p_0,(\lceil 2^{-2k}n\rceil ,\lceil 2^{-2k}n\rceil)\big)  \big] \le C\mathbb E \big[Q\big(p_0,(\lceil 2^{-2k}n\rceil ,\lceil 2^{-2k}n\rceil)\big)  \big] \le C2^{-2k}n \, ,
\end{equation}
where $Q$ is the usual exponential last passage percolation from Section~\ref{sec:exponential lpp}, and where the last inequality follows from Theorem~\ref{thm:Ledoux}. This finishes the proof of \eqref{eq:373}.
\end{proof}

\begin{proof}[Proof of Lemma~\ref{lem:variance bound}]
The last passage time $H(p_0,p_n)$ is a function of the independent exponential variables $\omega _v$ for $v\in \mathcal H $. Thus, using the Efron-Stein inequality, we have that 
\begin{equation}
    \text{Var} \big(H(p_0,p_n) \big) \le \sum _{v\in \mathcal H}  \mathbb E \big[\big( H(p_0,p_n)-H_v(p_0,p_n) \big)_+^2 \big] \, ,
\end{equation}
where $H_v(p_0,p_n)$ is defined in the same way as $H(p_0,p_n)$, but with the variable $\omega _v$ replaced by a copy of it $\omega _v'$, independent of all other variables. Here, for all $u\neq v$, we keep the old variables $\omega _u$. For $x\in \mathbb R $, let $x_+:=\max (x,0)$. It is clear that $H_v(p_0,p_n)<H(p_0,p_n)$ only when $v$ is on the geodesic $\gamma (p_0,p_n)$ from in the old environment, and in this case $H(p_0,p_n)-H_v(p_0,p_n) \le \omega _v$. Thus  
\begin{equation}
    \text{Var} \big(H(p_0,p_n) \big) \le  \mathbb E  \Big[ \sum _{v\in \mathcal H}  \omega _v^2 \cdot  \mathds 1 \{ v\in \gamma(p_0,p_n)  \} \Big] \, .
\end{equation}
The last expectation is bounded by the expectation of the last passage time on the quadrant with the weights $\omega^2$, where $\omega$ are independent Exponential-$\alpha$-distributed, and which is therefore at most linear by Claim~\ref{claim:lingfu allan}.
The estimate 
\begin{equation}
    \big| \mathbb E [H(p_0,p_n)] -n/\rho _\alpha  \big|\le C\sqrt{n}
\end{equation}
follows from Theorem~\ref{thm:corwin}, and the bound $\text{Var}\big( H(p_0,p_n) \big) \le Cn$ together with Chebyshev's inequality.
\end{proof}

%\subsection{Proof of Proposition \ref{pro:}}

%test

\end{document}